\documentclass[12pt]{amsart}
\usepackage{a4wide,enumerate,xcolor,graphicx}
\usepackage{amsmath,comment}
\usepackage{scrextend} % for labeling environment
\allowdisplaybreaks

\let\pa\partial
\let\na\nabla
\let\eps\varepsilon
\newcommand{\N}{{\mathbb N}}
\newcommand{\R}{{\mathbb R}}
\newcommand{\diver}{\operatorname{div}}

\newcommand{\K}{{\mathcal K}}
\newcommand{\tildeK}{\widetilde{{\mathcal K}}_\zeta}
\newcommand{\DD}{\mathrm{D}}
\newcommand{\E}{\mathbb{E}}
\newcommand{\afrak}{a}

\DeclareMathOperator*{\esssup}{ess\,sup}

\newtheorem{theorem}{Theorem}
\newtheorem{lemma}[theorem]{Lemma}

\newtheorem{proposition}[theorem]{Proposition}

\newtheorem{corollary}[theorem]{Corollary}

\begin{document}

\title[Porous-medium equation with fractional diffusion]{Analysis and 
mean-field derivation of a \\ porous-medium equation with fractional diffusion}

\author[L. Chen]{Li Chen}
\address{University of Mannheim, School of Business Informatics and Mathematics, 
68131 Mann\-heim, Germany}
\email{chen@math.uni-mannheim.de}

\author[A. Holzinger]{Alexandra Holzinger}
\address{Institute for Analysis and Scientific Computing, Vienna University of
	Technology, Wiedner Hauptstra\ss e 8--10, 1040 Wien, Austria}
\email{alexandra.holzinger@tuwien.ac.at}

\author[A. J\"ungel]{Ansgar J\"ungel}
\address{Institute for Analysis and Scientific Computing, Vienna University of
	Technology, Wiedner Hauptstra\ss e 8--10, 1040 Wien, Austria}
\email{juengel@tuwien.ac.at}

\author[N. Zamponi]{Nicola Zamponi}
\address{University of Mannheim, School of Business Informatics and Mathematics, 
68131 Mann\-heim, Germany}
\email{nzamponi@mail.uni-mannheim.de}

\date{\today}

\thanks{The second and third authors have been partially supported by the 
Austrian Science Fund (FWF), grants P30000, P33010, F65, and W1245. 
The fourth author acknowledges support from the Alexander von Humboldt Foundation.
This work has received funding from the European 
Research Council (ERC) under the European Union's Horizon 2020 research and 
innovation programme, ERC Advanced Grant no.~101018153.}

\begin{abstract}
A mean-field-type limit from stochastic moderately interacting many-particle
systems with singular Riesz potential is performed, 
leading to nonlocal porous-medium
equations in the whole space. The nonlocality is given by the inverse of a
fractional Laplacian, and the limit equation can be interpreted as a transport
equation with a fractional pressure.
The proof is based on Oelschl\"ager's approach and
a priori estimates for the associated diffusion equations, coming from
energy-type and entropy inequalities as well as parabolic regularity. 
An existence analysis of the fractional porous-medium equation is also provided, 
based on a careful regularization procedure, new variants of fractional
Gagliardo--Nirenberg inequalities, and the div-curl lemma. A consequence
of the mean-field limit estimates is the propagation of chaos property.
\end{abstract}

% \paragraph{Keywords:}

% \paragraph{AMS classification:}

\maketitle

%%%%%%%%%%%%%%%%%%%%%%%%%%%%%%%%%%%%%%%%%%%%%%%%%%%%%%%%%%%%%%%%%%%%%%%%%%

%\red{\tableofcontents}

\section{Introduction}

The aim of this paper is to derive and analyze the following nonlocal
porous-medium equation:
\begin{equation}\label{1.eq}
  \pa_t\rho = \diver(\rho\na P), \quad P = (-\Delta)^{-s}f(\rho), \quad
	\rho(0)=\rho^0\quad\mbox{in }\R^d,
\end{equation}
where $0<s<1$, $d\ge 2$, and $f\in C^1([0,\infty))$ is a nondecreasing function
satisfying $f(0)=0$. This model describes a particle system that evolves
according to a continuity equation for the density $\rho(x,t)$
with velocity $v=-\na P$. The velocity
is assumed to be the gradient of a potential, which expresses Darcy's law. 
The pressure $P$ is related to the density in a nonlinear and nonlocal way through 
$P=(-\Delta)^{-s}f(\rho)$. The nonlocal operator $(-\Delta)^{-s}$ can be
written as a convolution operator with a singular kernel,
\begin{equation}\label{1.def}
  (-\Delta)^{-s}u = \K*u, \quad \K(x) = c_{d,-s}|x|^{2s-d}, \quad x\in\R^d,
\end{equation}
where $c_{d,-s}=\Gamma(d/2-s)/(4^s\pi^{d/2}\Gamma(s))$ and
$\Gamma$ denotes the Gamma function \cite[Theorem 5]{Sti18}.

If $s=0$, we recover the
porous-medium equation (for nonnegative solutions), 
while the case $s=1$ was investigated in \cite{CRS96,E94} with $f(u)=u$ 
for the evolution of the vortex density in a superconductor.
Other applications include particle systems with long-range interactions
and dislocation dynamics as a continuum \cite[Sec.~6.2]{Vaz17}.

Equation \eqref{1.eq} was first analyzed in \cite{CaVa11} with $f(u)=u$
for nonnegative solutions 
and in \cite{BIK11} with $f(u)=|u|^{m-2}u$ ($m>1$) for sign-changing solutions. 
The nonnegative solutions have the interesting property that they propagate with 
finite speed, which is not common in other fractional diffusion models
\cite{CaVa11,STV14}.
Equation \eqref{1.eq} was probabilistically interpreted in \cite{DeG19},
and it was shown that the probability density of a so-called random flight
process is given by a Barenblatt-type profile. 
Previous mean-field limits leading to \eqref{1.eq} were concerned
with the linear case $f(u)=u$ only; see \cite{Due16} (using the technique
of \cite{Ser17}) and \cite{ORT21}
(including additional diffusion as in \eqref{1.rho} below).
In \cite{ChJe21}, equation \eqref{1.eq} (with $f(u)=u$) 
was derived in the high-force regime from
the Euler--Riesz equations, which can be derived in the mean-field limit 
from interacting particle systems \cite{Ser20}. 
A direct derivation from particle systems with L\'evy noise was proved in
\cite{DPR20} for cross-diffusion systems, but still with $f(u)=u$.
Up to our knowledge, a rigorous derivation of \eqref{1.eq} 
from stochastic interacting particle systems for general nonlinearities $f(u)$ 
like power functions is missing in the literature. In this paper, we fill this gap.

\subsection{Problem setting}

Equation \eqref{1.eq} is derived from an interacting particle system with
$N$ particles, moving in the whole space $\R^d$. Because of the singularity
of the integral kernel and the degeneracy of
the nonlinearity, we approximate \eqref{1.eq} using three levels. 
First, we introduce a parabolic regularization adding a Brownian motion 
to the particle system with diffusivity $\sigma\in(0,1)$ and replacing $f$ by a 
smooth approximation $f_\sigma$. Second, we replace the interaction kernel 
$\K$ by a smooth kernel $\K_\zeta$ with compact support, where $\zeta>0$. 
Third, we consider
interaction functions $W_\beta$ with $\beta\in(0,1)$, which approximate the delta
distribution. (We refer to Subsection \ref{sec.main} for the precise definitions.)

The particle positions are represented on the {\em microscopic level}
by the stochastic processes $X_i^N(t)$ evolving according to
\begin{equation}\label{1.X}
\begin{aligned}
  & dX_i^N(t) = -\na\K_\zeta*f_\sigma\bigg(\frac{1}{N}\sum_{j=1,\,j\neq i}^{N}
	W_\beta(X_j^N(t)-X_i^N(t))\bigg)dt + \sqrt{2\sigma}dB_i^N(t), \\
	& X_i^N(0) = \xi_i, \quad i=1,\ldots,N,
\end{aligned}
\end{equation}
where the convolution has to be understood with respect to $x_i$,
$(B_i^N(t))_{t\ge 0}$ are independent $d$-dimensional Brownian motions defined
on a filtered probability space $(\Omega,\mathcal{F},\mathcal{F}_t,\mathbb{P})$,
and $\xi_i$ are independent identically distributed random variables in $\R^d$
with the same probability density function $\rho^{0}_\sigma$ (defined in
\eqref{1.kappa} below). 

The mean-field-type limit is performed in three steps. First,
for fixed $(\sigma,\beta,\zeta)$, system \eqref{1.X} is approximated for $N\to\infty$
on the {\em intermediate level} by
\begin{equation}\label{1.barX}
\begin{aligned}
  & d\bar{X}_i^N(t) = -\na\K_\zeta*f_\sigma\big(W_\beta*\rho_{\sigma,\beta,\zeta}
	(\bar{X}_i^N(t),t)\big)dt + \sqrt{2\sigma}dB_i^N(t), \\
	& \bar{X}_i^N(0) = \xi_i, \quad i=1,\ldots,N,
\end{aligned}
\end{equation}
where $\rho_{\sigma,\beta,\zeta}$ is the probability density function of
$\bar{X}_i^N$ and a strong solution to 
\begin{equation}\label{1.rho3}
  \pa_t\rho_{\sigma,\beta,\zeta} - \sigma\Delta\rho_{\sigma,\beta,\zeta}
	= \diver\big(\rho_{\sigma,\beta,\zeta}
	\na\K_\zeta*f_\sigma(W_\beta*\rho_{\sigma,\beta,\zeta})
	\big), \quad \rho_{\sigma,\beta,\zeta}(0) = \rho^{0}_\sigma\quad\mbox{in }\R^d.
\end{equation}
System \eqref{1.barX} is uncoupled, since $\bar{X}_i^N$ depends on $N$ only
through the initial datum.

Second, passing to the limit $(\beta,\zeta)\to 0$ in the intermediate system leads on
the {\em macroscopic level} to
\begin{equation}\label{1.hatX}
\begin{aligned}
  & d\widehat{X}_i^N(t) = -\na\K* f_\sigma(\rho_\sigma(\widehat{X}_i^N(t),t))dt
	+ \sqrt{2\sigma}dB_i^N(t), \\
	& \widehat{X}_i^N(0) = \xi_i, \quad i=1,\ldots,N,
\end{aligned}
\end{equation}
where $\rho_\sigma$ is the density function of $\widehat{X}_i^N$ and a weak
solution to 
\begin{equation}\label{1.rho}
  \pa_t\rho_\sigma = \sigma\Delta\rho_\sigma 
	+ \diver(\rho_\sigma\na(-\Delta)^{-s}f_\sigma(\rho_\sigma)), \quad
	\rho_\sigma(0)=\rho_\sigma^0\quad\mbox{in }\R^d.
\end{equation}

We perform the limits $N\to\infty$ and $(\beta,\zeta)\to 0$ simultaneously.
The logarithmic scaling $\beta\sim(\log N)^{-\mu}$ for some $\mu>0$
corresponds to the moderately interacting particle regime, according to the 
notation of Oelschl\"ager \cite{Oel89}, while
the smoothing parameter $\zeta$ is allowed to depend algebraically on $N$, i.e.\
$\zeta\sim N^{-\nu}$ for some $\nu>0$; see Theorem \ref{thm.error} for details.
Our approach also implies the two-step limit but leading to weak convergence only,
compared to the convergence in expectation obtained in Theorem \ref{thm.error}.

Third, the limit $\sigma\to 0$ is performed on the level of the diffusion 
equation, based on a priori estimates uniform in $\sigma$ and the div-curl lemma.

%%%%%%%%%%%%%%%

\subsection{State of the art}

We already mentioned that the existence of weak solutions to \eqref{1.eq} with
$f(u)=u$ was proved first in \cite{CaVa11}. The convergence of the weak solution to
a self-similar profile was shown by the same authors in \cite{CaVa11a}. 
The convergence becomes exponential, at least in one space dimension, 
when adding a confinement potential \cite{CHSV15}. Equation \eqref{1.eq} with $f(u)=u$
was identified as the Wasserstein gradient flow of a square fractional Sobolev
norm \cite{LMS18}, implying time decay as well as energy and entropy estimates.
The H\"older regularity of solutions to \eqref{1.eq} 
was proved in \cite{CSV13} for $f(u)=u$ and in \cite{ITV20} for $f(u)=u^{m-1}$ 
and $m\ge 2$. 

In the literature, related equations have been analyzed too.
Equation \eqref{1.eq} for $f(u)=u$ and the limit case $s=1$
was shown in \cite{AmSe08} to be the Wasserstein gradient flow on the
space of probability measures,
leading to the well-posedness of the equation and energy-dissipation inequalities.
The existence of local smooth solutions to the regularized equation \eqref{1.rho} 
are proved in \cite{ChJe21a}. 
The solutions $\pa_t\rho=\diver(\rho^{m-1}\na P)$ with $P=(-\Delta)^{-s}\rho$
in $\R^d$ propagate with finite speed if and only if $m\ge 2$ \cite{STV14}.
The existence of weak solutions to this equation with 
$P=(-\Delta)^{-s}(\rho^{n})$ and $n>0$
is proved in \cite{NgVa18} (in bounded domains). 
While \eqref{1.eq} has a parabolic-elliptic structure,
parabolic-parabolic systems have been also investigated. For instance, the
global existence of weak solutions to $\pa_t\rho = \diver(\rho\na P)$ and 
$\pa_t P + (-\Delta)^s P = \rho^\beta$, where $\beta>1$, was shown in
\cite{CGZ20}. In \cite{DGZ20}, the algebraic decay towards the steady state 
was proved in the case $\beta=2$. 
We also mention that fractional porous-medium equations of the type
$\pa_t\rho+(-\Delta)^{s/2}f(\rho)=0$ in $\R^d$ have been studied in the literature;
see, e.g., \cite{PQRV18}. Compared to \eqref{1.eq}, this problem has infinite
speed of propagation. For a review and comparison of 
this model and \eqref{1.eq}, we refer to \cite{Vaz12}.

There is a huge literature concerning mean-field limits leading to diffusion
equations, and the research started already in the 1980s; we refer to the
reviews \cite{Gol03,JaWa17} and the classical works of Sznitman \cite{Szn84,Szn91}. 
Oelschl\"ager proved the mean-field limit in weakly interacting particle systems
\cite{Oel84}, leading to deterministic nonlinear processes, and moderately interacting
particle systems \cite{Oel90}, giving porous-medium-type equations with quadratic
diffusion. First investigations of moderate interactions in stochastic particle 
systems with nonlinear diffusion coefficients were performed in \cite{JoMe98}. 
The approach of moderate interactions was extended in \cite{CDHJ20,CDJ19} 
to multi-species systems, deriving population cross-diffusion systems.
Reaction-diffusion equations with nonlocal terms were derived in the mean-field
limit in \cite{IRS12}. The large population limit of point measure-valued Markov 
processes leads to nonlocal Lotka--Volterra systems with cross diffusion
\cite{FoMe15}. Further references can be found in \cite[Sec.~1.3]{ORT21}.

Compared to previous works, we consider a singular kernel $\K$ and derive
a partial differential equation without Laplace diffusion by taking
the limit $\sigma\to 0$. The authors of \cite{FiPh08} derived the viscous
porous-medium equation by starting from a stochastic particle system with a
double convolution structure in the drift term, similar to \eqref{1.barX}.
The main difference to our work is that (besides
different techniques for the existence and regularity of solutions to the
parabolic problems) we consider a singular kernel in one part
of the convolution and a different scaling for the approximating regularized 
kernel $\K_\zeta=\K\omega_\zeta*W_\zeta$, where $\omega_\zeta$ is a 
$W^{1,\infty}(\R^d)$ cutoff function (see Section 1.3 for the exact definition),
in comparison to the interaction scaling
$W_\beta*\rho_{\sigma,\beta,\zeta}$. The two different scalings $\beta$ and $\zeta$
allow us to establish a result, for which the kernel regularization on the particle
level does not need to be of logarithmic type but of power-law type only.

%%%%%%%%%%%%%%%%

\subsection{Main results and key ideas}\label{sec.main}

We impose the following hypotheses:
\begin{enumerate}
\item[(H1)] Data: Let $0<s<1$, $d\ge 2$.
\item[(H2)] $\rho^0\in L^\infty(\R^d)\cap L^1(\R^d)$ satisfies $\rho^0 \geq 0$
in $\R^d$ and $\int_{\R^d}\rho^0(x)|x|^{2d/(d-2s)}dx<\infty$.
\item[(H3)] Nonlinearity: $f\in C^{1}([0,\infty))$ is nondecreasing, $f(0)=0$,
and $u\mapsto uf(u)$ for $u>0$ is strictly convex.
\end{enumerate}

Let us discuss these assumptions. 
We assume that $d\ge 2$; the case $d=1$ can be treated if $s<1/2$; see
\cite{CaVa11}. Extending the range of $s$ to $s<0$ leads to the fractional
(higher-order) thin-film equation, which is studied in \cite{Lis20}. 
The case $1<s<d/2$ may be considered too, since it yields better regularity
results; we leave the details to the reader. On the other hand, the case $s\ge d/2$
is more delicate since the multiplier in the definition of $(-\Delta)^{-s}$
using Fourier transforms does not define a tempered distribution. The case
$s=d/2$ for $d\le 2$ (with a logarithmic Riesz kernel) was analyzed in \cite{Due16}. 
We need the moment bound for the initial datum $\rho^0$ to 
prove the same moment bound for $\rho_\sigma$, which in turn is used
several times, for instance to show the entropy balance and the convergence
$\rho_\sigma\to\rho$ as $\sigma\to 0$ 
in the sense of $C^0_{\rm weak}([0,T];L^1(\R^d))$.
The monotonicity of $f$ and the strict convexity of $u\mapsto uf(u)$ are needed to 
prove the strong convergence of $(\rho_\sigma)$, which then allows us to 
identify the limit of $(f_\sigma(\rho_\sigma))$.
An example of a function satisfying Hypothesis (H3) is
$f(u)=u^\beta$ with $\beta\ge 1$.

Our first main result is concerned with the existence analysis of \eqref{1.eq}.
We write $\|\cdot\|_p$ for the $L^p(\R^d)$ norm.

\begin{theorem}[Existence of weak solutions to \eqref{1.eq}]\label{thm.ex}
Let Hypotheses (H1)--(H3) hold. Then there exists a weak solution $\rho\ge 0$ to
\eqref{1.eq} satisfying {\em (i)} the regularity
\begin{align*}
  & \rho\in L^\infty(0,\infty;L^1(\R^d)\cap L^\infty(\R^d)), \quad
	\na(-\Delta)^{-s/2}f(\rho)\in L^2(0,\infty;L^2(\R^d)), \\
	& \pa_t\rho\in L^2(0,\infty;H^{-1}(\R^d)),
\end{align*}
{\em (ii)} the weak formulation
\begin{equation}\label{1.weak}
  \int_0^T\langle\pa_t\rho,\phi\rangle dt 
	+ \int_0^T\int_{\R^d}\rho\na(-\Delta)^{-s}f(\rho)\cdot\na\phi dxdt = 0
\end{equation}
for all $\phi\in L^2(0,T;H^1(\R^d))$ and $T>0$, {\em (iii)} the initial datum
$\rho(0)=\rho^0$ in the sense of $H^{-1}(\R^d)$, and {\em (iv)} the following
properties for $t>0$:
\begin{itemize}
\item Mass conservation: $\|\rho(t)\|_1 = \|\rho^0\|_1$, 
\item Dissipation of the $L^\infty$ norm: 
$\|\rho(t)\|_{\infty} \le \|\rho^0\|_\infty$, 
\item Moment estimate: $\sup_{0<t<T}\int_{\R^d}\rho(x,t)|x|^{2d/(d-2s)}dx \le C(T)$,
\item Entropy inequality:
$$
  \int_{\R^d}h(\rho(t))dx + \int_0^t\int_{\R^d}|\na(-\Delta)^{-s/2}f(\rho)|^2dxds
	\le \int_{\R^d}h(\rho^0)dx. %\label{1.ent}
$$
\end{itemize}
\end{theorem}

Note that the Hardy--Littlewood--Sobolev-type inequality \eqref{HLS2} 
(see Appendix \ref{sec.frac}) implies that
$$
  \|\rho\na(-\Delta)^{-s}f(\rho)\|_2
	= \|\rho(-\Delta)^{-s/2}[\na(-\Delta)^{-s/2}f(\rho)]\|_2
	\le C\|\rho\|_{d/(2s)}\|\na(-\Delta)^{-s/2}f(\rho)\|_2,
$$
such that $\rho\na(-\Delta)^{-s}f(\rho)\in L^2(\R^d)$, and the weak formulation 
\eqref{1.weak} is defined.

The key ideas of the proof of Theorem \ref{thm.ex} are as follows.
A priori estimates for strong solutions $\rho_\sigma$ to the regularized equation
\eqref{1.rho} are derived from mass conservation, the entropy
inequality, and energy-type bounds. The
energy-type bound allows us to show, for sufficiently small $\sigma>0$, that
the $L^\infty$ norm of $\rho_\sigma$ is bounded by the $L^\infty$ norm of $\rho^0$, 
up to some factor depending on the moment bound for $\rho^0$.
The existence of a strong solution $\rho_\sigma$ is proved by regularizing
\eqref{1.rho} in a careful way to deal with the singular kernel. 
The regularized equation
is solved locally in time by Banach's fixed-point theorem. Entropy estimates
allow us to extend this solution globally in time and to pass to the
de-regularization limit.
The second step is the limit $\sigma\to 0$ in \eqref{1.rho}. Since the bounds 
only provide weak convergence of (a subsequence of) $\rho_\sigma$, the 
main difficulty is the identification of the nonlinear limit $f_\sigma(\rho_\sigma)$.
This is done by applying the div-curl lemma and 
exploiting the monotonicity of $f$ and the strict convexity of 
$u\mapsto uf(u)$ \cite{FeNo09}. 

We already mentioned that the existence of local smooth solutions $\rho_\sigma$
to \eqref{1.rho} has been proven in \cite{ChJe21}. However, we provide an
independent proof that allows for global strong solutions and that yields
a priori estimates needed in the mean-field limit.

Our second main result is concerned with the mean-field-type limit. For this,
we need some definitions. Define 
\begin{equation}\label{1.fsigma}
  f_\sigma(u) = \int_0^u (\Gamma_\sigma*(f'\mathrm{1}_{[0,\infty)}))(w)
	\widetilde\Xi(\sigma w) dw\quad u\in\R,
\end{equation} 
where the mollifier $\Gamma_\sigma$ for $\sigma>0$ is given by 
$\Gamma_\sigma(x)=\sigma^{-1}\Gamma_1(x/\sigma)$, and $\Gamma_1\in C_0^\infty(\R)$
satisfies $\Gamma_1\ge 0$, $\|\Gamma_1\|_1=1$, while the cutoff function 
$\widetilde\Xi\in C_0^\infty(\R)$ satisfies 
$0\le\widetilde\Xi\le 1$ in $\R$ and $\widetilde\Xi(x)=1$ for $|x|\le 1$.
Then $f_\sigma\in C^\infty(\R)$, $f'_\sigma\ge 0$, $f_\sigma(0)=0$, and 
the derivatives $\DD^k f_\sigma$ are bounded and compactly supported for all $k\geq 1$. 
In a similar way, we introduce the mollifier function $W_\beta$ for $\beta>0$ 
and $x\in\R^d$ by
\begin{equation}\label{1.Wbeta}
  W_\beta(x) = \beta^{-d}W_1(x/\beta),\quad
	W_1\in C_0^\infty(\R^d)\mbox{ is symmetric},\ W_1\ge 0,\ \|W_1\|_1 = 1.
\end{equation}
Let us define the cutoff version of the singular kernel $\K$ by 
\begin{align}
  & \tildeK := \K \omega_\zeta, \text{ where the cut-off function }
	\omega_\zeta \in W^{1,\infty}(\R^d) \text{ is such that} \nonumber \\
  & 0 \leq \omega_\zeta(x) \leq 1\mbox{ for }x\in\R^d, \quad
	\|\na\omega_\zeta\|_\infty\leq 2\zeta, \label{1.TildeK} \\
  & \omega_\zeta(x) = 1\mbox{ for all } |x|\le \zeta^{-1},\quad
  \omega_\zeta(x) = 0\mbox{ for all } |x|\geq 2\zeta^{-1}. \nonumber
\end{align}
Then the regularized kernel $\K_\zeta$ is given by 
$$
  \K_\zeta(x):=\tildeK*W_\zeta(x)\quad\text{for all } x\in \R^d,
$$
where $\zeta>0$. Let the cutoff function $\Xi\in C_0^\infty(\R^d)$ satisfy 
$0\le\Xi\le 1$ in $\R^d$ and $\Xi(x)=1$ for $|x|\le 1$. Then we define the regularized
initial datum for $x\in\R^d$ by
\begin{equation}\label{1.kappa}
  \rho^0_\sigma(x) = \kappa_\sigma(W_\sigma*\rho^0)(x)\Xi(\sigma x), \quad\mbox{where }
	\kappa_\sigma = \frac{\int_{\R^d}\rho^0(y)dy}{\int_{\R^d}(W_\sigma*\rho^0)(y)
	\Xi(\sigma y)dy}.
\end{equation}
This definition guarantees the mass conservation since 
$\|\rho^0_\sigma\|_1=\|\rho^0\|_1$; see Section \ref{sec.basic}.

\begin{theorem}[Error estimate for the stochastic system]\label{thm.error}
Let $X_i^N$ and $\widehat{X}_i^N$ be the solutions to \eqref{1.X} and
\eqref{1.hatX}, respectively.  We assume that 
$\zeta^{-2s-1} \le C_1N^{1/4}$ for some constant $C_1>0$. 
Let $\delta\in(0,1/4)$ and $\afrak:= \min\{1,d-2s\}>0$.
Then there exist constants $\eps>0$, depending on $\sigma$ and $\delta$, and
$C_2>0$, depending on $\sigma$ and $T$,
such that if $\beta^{-3d-7}\le \eps\log N$ then
$$
  \E\bigg(\sup_{0<s<T}\max_{i=1,\ldots,N}|(X_i^N-\widehat{X}_i^N)(s)|\bigg) 
	\le C_2(\beta+\zeta^{\afrak} )\to 0 \text{ as }
	(N,\zeta, \beta) \to (\infty,0,0).
$$
\end{theorem}
The theorem is proved by estimating the differences
\begin{align*}
  E_1(t) &:= \E\bigg(\sup_{0<s<t}\max_{i=1,\ldots,N}|(X_i^N-\bar{X}_i^N)(s)|\bigg), \\
	E_2(t) &:= \E\bigg(\sup_{0<s<t}\max_{i=1,\ldots,N}|(\bar{X}_i^N
	-\widehat{X}_i^N)(s)|\bigg),
\end{align*}
and applying the triangle inequality. For the first difference,
we estimate expressions like $\|\DD^k\K_\zeta*u\|_\infty$ for appropriate
functions $u$ and $\|\DD^k W_\beta\|_\infty$ for $k\in\N$ 
in terms of negative powers of $\beta$ (here, $\DD^k$ denotes the $k$th-order
partial derivatives). Using properties of Riesz potentials,
in particular Hardy--Littlewood--Sobolev-type
inequalites (see Lemmas \ref{lem.hls} and \ref{lem.hlsh}),
we show that for some $\mu_i>0$ ($i=1,2,3$),
$$
  E_1(t) \le C(\sigma)\beta^{-\mu_1}\int_0^t E_1(s)ds 
	+ C(\sigma)\beta^{-\mu_2}\zeta^{-\mu_3}N^{-1/2}.
$$
By applying the Gronwall lemma and choosing
a logarithmic scaling for $\beta$ and an algebraic scaling for $\zeta$ with
respect to $N$, we infer that $E_1(t)\le C(\sigma)N^{-\mu_4}$ 
for some $\mu_4\in(0,1/4)$. For the second difference $E_2$, we need the estimates
$\|W_\beta*u-u\|_\infty\le C(\sigma)\beta$ (Lemma \ref{lem.wbeta}), and
$\|(\K_\zeta-\K)*\rho_\sigma\|\le C(\sigma)\zeta^{\afrak}$,
$\|\rho_{\sigma,\beta,\zeta}-\rho_\sigma\|_\infty
\le C(\sigma)(\beta+\zeta^{\afrak})$
(Proposition \ref{prop.u}), recalling that $\afrak= \min\{1, d-2s\}$. 
The proof of these estimates is very technical. The idea is to apply
several times fractional Gagliardo--Nirenberg inequalities 
that are proved in Appendix \ref{sec.frac} and Hardy--Littlewood--Sobolev
inequalities that are recalled in Lemmas \ref{lem.hls}--\ref{lem.hlsh}.
Then, after suitable computations,
$$
  E_2(t) \le C(\sigma)(\beta+\zeta^{\afrak}) + C(\sigma)\int_0^t E_2(s)ds,
$$
and we conclude with Gronwall's lemma that $E_2(t)\le C(\sigma)(\beta+\zeta^{\afrak})$.

Theorem \ref{thm.error} and calculations for $\sigma\to 0$
yield the following propagation of chaos result. 

\begin{theorem}[Propagation of chaos for \eqref{1.eq}]\label{thm.chaos}
Let the assumptions of Theorem \ref{thm.error} hold and let 
$\mathrm{P}_{N,\sigma,\beta,\zeta}^k(t)$ be the joint distribution of
$(X_1^N(t),\ldots,X_k^N(t))$ for a fixed $t\in(0,T)$. Then there exists
a subsequence in $\sigma$ such that
$$
  \lim_{\sigma\to 0}\lim_{N\to\infty,\,(\beta,\zeta)\to 0}
	\mathrm{P}_{N,\sigma,\beta,\zeta}^k(t) = \mathrm{P}^{\otimes k}(t),
$$
where the limit is locally uniform in $t$, the limit $N\to\infty$,
$(\beta,\zeta)\to 0$ has to be understood in the sense of Theorem \ref{thm.error},
and the measure $\mathrm{P}(t)$ is absolutely continuous with respect to the
Lebesgue measure with the probability density function $\rho(t)$ that is
a weak solution to \eqref{1.eq}.
\end{theorem}

If equation \eqref{1.eq} was uniquely solvable, we would obtain the convergence
of the whole sequence in $\sigma$. Unfortunately, the regularity of the
solution $\rho$ to \eqref{1.eq} is too weak to conclude the uniqueness of 
weak solutions. 
Up to our knowledge, none of the known methods, such as \cite{BSV15,CiJa11},
seem to be applicable to equation \eqref{1.eq}.

The paper is organized as follows. The existence of global nonnegative weak
solutions to \eqref{1.eq} is proved in Section \ref{sec.rho} by establishing an
existence analysis for \eqref{1.rho} and performing the limit $\sigma\to 0$.
Some uniform estimates for the solution $\rho_{\rho,\beta,\zeta}$ to \eqref{1.rho3}
and for the difference $\rho_{\sigma,\beta,\zeta}-\rho_\sigma$ are shown
in Section \ref{sec.rho3}. Section \ref{sec.mean} is devoted to the proof
of the error estimate in Theorem \ref{thm.error} and the propagation of
chaos in Theorem \ref{thm.chaos}. In Appendices \ref{sec.aux}--\ref{sec.regul}
we recall some auxiliary results and Hardy--Littlewood--Sobolev-type
inequalities, prove new variants of fractional Gagliardo--Nirenberg inequalities,
and formulate a result on parabolic regularity.

\subsection*{Notation}

We write $\|\cdot\|_p$ for the $L^p(\R^d)$ or $L^p(\R)$ norm with $1\le p\le \infty$.
The ball around the origin with radius $R>0$ is denoted by $B_R$.
The partial derivative $\pa/\pa x_i$ is abbreviated as $\pa_i$ for $i=1,\ldots,d$,
and $\DD^\alpha$ denotes a partial derivative of order $|\alpha|$, where 
$\alpha\in\N_0^d$ is a multiindex. The notation $\DD^k$ refers to the
$k$th-order tensor of partial derivatives of order $k\in\N$. In this situation,
the norm $\|\DD^k u\|_p$ is the sum of all $L^p$ norms of partial derivatives of $u$
of order $k$. Finally, $C>0$, $C_1>0$, etc.\ denote generic constants with values 
changing from line to line.

%%%%%%%%%%%%%%%%%%%%%%%%%%%%%%%%%%%%%%%%%%%%%%%%%%%%%%%%%%%%%%%%%%%%%%%%%%

\section{Analysis of equation \eqref{1.eq}}\label{sec.rho}

In this section, we prove the existence of global nonnegative weak solutions
to \eqref{1.eq} and an estimate for the difference
$\rho_{\sigma,\beta,\zeta}-\rho_\sigma$ of the solutions to \eqref{1.rho3} and 
\eqref{1.rho}, respectively, needed in the mean-field limit.
We first prove the existence of a solution $\rho_\sigma$ to \eqref{1.rho}
by a fixed-point argument and then perform the limit $\sigma\to 0$.
Recall definition \eqref{1.kappa} of the number $\kappa_\sigma$, which
is stated in (iv) below.

\begin{theorem}\label{thm.rho}
Let Hypotheses (H1)--(H3) hold. Then for all $\sigma > 0$, there exists a unique 
weak solution $\rho_\sigma\ge 0$ to \eqref{1.rho} satisfying {\em (i)} the regularity
\begin{align*}
  & \rho_\sigma\in L^\infty(0,\infty;L^1(\R^d)\cap L^\infty(\R^d))\cap 
	C^0([0,\infty);L^2(\R^d)), \\
	& \na\rho_\sigma \in L^2(0,\infty;L^2(\R^d)), \quad
	\pa_t\rho_\sigma\in L^2(0,\infty;H^{-1}(\R^d)),
\end{align*}
{\em (ii)} the weak formulation of \eqref{1.rho} with test functions 
$\phi\in L^2(0,T; H^1(\R^d))$,
{\em (iii)} the inital datum $\rho_\sigma(0) = \rho_\sigma^0$ in $L^2(\R^d)$, 
and {\em (vi)} the following properties for $t>0$, which are uniform in $\sigma$ 
for sufficiently small $\sigma >0$:
\begin{itemize}
 \item Mass conservation: $\|\rho_\sigma(t)\|_1 = \|\rho^0\|_1$.
 \item Dissipation of the $L^\infty$ norm: 
  $\|\rho_\sigma\|_{L^\infty(0,\infty; L^\infty(\R^d))}\le\kappa_\sigma
  \|\rho^0\|_{L^\infty(\R^d)}\leq C\|\rho^0\|_{L^\infty(\R^d)}$. 
 \item Moment estimate: $\sup_{t\in [0,\infty)}\int_{\R^d} 
\rho_\sigma(x,t) |x|^{\frac{2d}{d-2s}} dx \leq C_T$.
 \item Entropy inequality: 
 \begin{align*}
   \int_{\R^d} & h(\rho_\sigma(T))dx 
   + 4\sigma \int_0^T\int_{\R^d} f_\sigma'(\rho_\sigma)|\nabla\sqrt{\rho_\sigma}|^2
	dx dt \\
  &{}+ \int_0^T\int_{\R^d}|\na (-\Delta)^{-s/2}f_\sigma(\rho_\sigma)|^2 dx dt 
	\leq\int_{\R^d} h(\rho_\sigma^0) dx\quad \mbox{for all }T>0.
 \end{align*}
\end{itemize}
Additionally, for any $T>0$, $1<p<\infty$, and $2\le q<\infty$, there exists $C>0$,
depending on $T$, $\sigma$, $p$, and $q$, such that 
$$
  \|\rho_\sigma\|_{L^p(0,T;W^{3,p}(\R^d))}
	+ \|\pa_t\rho_\sigma\|_{L^p(0,T;W^{1,p}(\R^d))}
	+ \|\rho_\sigma\|_{C^0([0,T];W^{2,1}(\R^d) \cap W^{3,q}(\R^d))} \le C,
$$
i.e., $\rho_\sigma$ is even a strong solution to \eqref{1.rho} and
$\rho_\sigma\in C^0([0,T];W^{2,1}(\R^d) \cap W^{3,q}(\R^d))$ for $q\ge 2$.
\end{theorem}

%%%%%%%%%%%%%%%%%%%%%

\subsection{Basic estimates for $\rho_\sigma$}\label{sec.basic}

We prove a priori estimates in $L^p$ spaces and an energy-type estimate.
Let $\sigma\in(0,1)$ and 
let $\rho_\sigma$ be a nonnegative strong solution to \eqref{1.rho}.
Integration of \eqref{1.rho} in $\R^d$ and the definition of $\rho_\sigma^0$
immediately yield the mass conservation
\begin{equation}\label{3.cons}
  \|\rho_\sigma(t)\|_1 = \|\rho_\sigma^0\|_1 = \|\rho^0\|_1\quad\mbox{for }t>0.
\end{equation}

\begin{lemma}[Energy-type estimate]\label{lem.energy}
Let $F\in C^2([0,\infty))$ be convex
and let $F(\rho_\sigma^0)\in L^1(\R^d)$. Then
\begin{align}\label{3.dFdt}
  \frac{d}{dt}\int_{\R^d}&F(\rho_\sigma)dx 
	= -\sigma\int_{\R^d}F''(\rho_\sigma)|\na\rho_\sigma|^2 dx \\
	&{}-\frac{c_{d,1-s}}{2}\int_{\R^d}\int_{\R^d}
	\frac{(G(\rho_\sigma(x))-G(\rho_\sigma(y)))
	(f_\sigma(\rho_\sigma(x))-f_\sigma(\rho_\sigma(y)))}{|x-y|^{d+2(1-s)}}dxdy \le 0,
	\nonumber
\end{align}
where $G(u):=\int_0^u vF''(v)dv$ for $u\ge 0$ and $c_{d,1-s}$ is defined after
\eqref{1.def}.
\end{lemma}

\begin{proof}
First, we assume that $F''$ is additionally bounded. Then $F'(\rho_\sigma)-F'(0)$ is an
admissible test function in the weak formulation of \eqref{1.rho},
since $|F'(\rho_\sigma)-F'(0)|\le \|F''\|_\infty|\rho_\sigma|$. It follows from
definition \eqref{Def.FractionalLaplacian} of the fractional Laplacian
and integration by parts that
\begin{align*}
  \frac{d}{dt}\int_{\R^d}&F(\rho_\sigma)dx 
	 + \sigma\int_{\R^d}F''(\rho_\sigma)|\na\rho_\sigma|^2 dx 
	= -\int_{\R^d}F''(\rho_\sigma)\rho_\sigma\na\rho_\sigma\cdot
	\na(-\Delta)^{-s}f_\sigma(\rho_\sigma)dx \\
	&= -\int_{\R^d}\na G(\rho_\sigma)\cdot\na(-\Delta)^{-s}f_\sigma(\rho_\sigma)dx
	= -\int_{\R^d}G(\rho_\sigma)(-\Delta)^{1-s}f_\sigma(\rho_\sigma)dx \\
	&= -c_{d,1-s}\int_{\R^d}\int_{\R^d}G(\rho_\sigma(x))
	\frac{f_\sigma(\rho_\sigma(x))-f_\sigma(\rho_\sigma(y))}{|x-y|^{d+2(1-s)}}dx dy.
\end{align*}
A symmetrization of the last integral yields \eqref{3.dFdt}.

In the general case, we introduce $F_k(u) = F(0) + F'(0)u + \int_0^u
\int_0^v \min\{F''(w), k\} dw dv$ for $k>0$. Then $F''_k(u)$ is bounded
and \eqref{3.dFdt} follows for $F$ replaced by $F_k$. The result follows after
taking the limit $k\to\infty$ using monotone convergence. 
\end{proof}

We need a bound on $\kappa_\sigma$, defined in \eqref{1.kappa},
to derive uniform $L^\infty(\R^d)$ bounds for $\rho_\sigma$. 

\begin{lemma}[Bound for $\kappa_\sigma$]%\label{lem.kappa}
There exists $C>0$ such that, for sufficiently small $\sigma>0$,
$$
  1\le\kappa_\sigma\le \frac{1}{1-C\sigma E}, \quad\mbox{where }
  E := \frac{1}{\|\rho^0\|_1}\int_{\R^d}(1+|x|^{2d/(d-2s)})\rho^0(x)dx.
$$
\end{lemma}

\begin{proof}
By Young's convolution inequality (Lemma \ref{lem.young}), we have
$$
  \int_{\R^d}(W_\sigma*\rho^0)(x)\Xi(\sigma x)dx
	\le \|W_\sigma*\rho^0\|_1 \le \|W_\sigma\|_1\|\rho^0\|_1 
	= \|\rho^0\|_1,
$$
which shows that $\kappa_\sigma\ge 1$. To prove the upper bound, we use the
triangle inequality $|x|\le |x-y|+|y|$:
\begin{align*}
  \int_{\R^d}&(W_\sigma*\rho^0)(x)\Xi(\sigma x)dx
	\ge \int_{\{|x|\le 1/\sigma\}}\int_{\R^d}W_\sigma(x-y)\rho^0(y)dydx \\
	&= \int_{\R^d}\bigg(\int_{\R^d}W_\sigma(x-y)dx\bigg)\rho^0(y)dy
	- \int_{\{|x| > 1/\sigma\}}\int_{\R^d}W_\sigma(x-y)\rho^0(y)dydx \\
	&\ge \int_{\R^d}\rho^0(y)dy - \sigma^{2d/(d-2s)}\int_{\{|x| > 1/\sigma\}}\int_{\R^d}
	|x|^{2d/(d-2s)}W_\sigma(x-y)\rho^0(y)dydx \\
	&\ge \int_{\R^d}\rho^0(y)dy - \sigma^{2d/(d-2s)}\int_{\R^d}\int_{\R^d}
	|x-y|^{2d/(d-2s)}W_\sigma(x-y)\rho^0(y)dydx \\
	&\phantom{xx}{}-\sigma^{2d/(d-2s)}\int_{\R^d}\int_{\R^d}
	|y|^{2d/(d-2s)}W_\sigma(x-y)\rho^0(y)dydx.
\end{align*}
Using the property  
$\int_{\R^d}|z|^{2d/(d-2s)} W_\sigma(z)dz\le C\sigma^{2d/(d-2s)}$ for the second
term on the right-hand side and $\|W_\beta\|_{L^1(\R^d)}=1$ for the third term,
we find that
\begin{align*}
  \int_{\R^d}(W_\sigma*\rho^0)(x)\Xi(\sigma x)dx
	&\ge \int_{\R^d}\rho^0(y)dy - C\sigma^{4d/(d-2s)}\int_{\R^d}\rho^0(y)dy \\
	&\phantom{xx}{}-\sigma^{2d/(d-2s)}\int_{\R^d}
	|y|^{2d/(d-2s)}\rho^0(y)dy.
\end{align*}
Because of $\sigma^{2d/(d-2s)}\le\sigma$ for $\sigma\le 1$, we obtain
\begin{align*}
  \frac{\|\rho^0\|_1}{\kappa_\sigma}
	&= \int_{\R^d}(W_\sigma*\rho^0)(x)\Xi(\sigma x)dx
	\ge \int_{\R^d}\rho^0(y)dy - C\sigma\int_{\R^d}(1+|y|^{2d/(d-2s)})\rho^0(y)dy \\
	&\ge \int_{\R^d}\rho^0(y)dy - C\sigma\int_{\R^d}\rho^0(y)dy\cdot E
	= \|\rho^0\|_1(1-C\sigma E),
\end{align*}
which proves the lemma.
\end{proof}

\begin{lemma}[Bounds for $\rho_\sigma$]\label{lem.bound}
The following bounds hold:
\begin{align}
  \|\rho_\sigma(t)\|_\infty &\le \kappa_\sigma\|\rho^0\|_\infty 
	\le C\|\rho^0\|_\infty,	\quad t>0, \label{Linf} \\
	\sqrt{\sigma}\|\rho_\sigma\|_{L^2(0,T;H^1(\R^d))}
	&\le C(T,\|\rho^0\|_1,\|\rho^0\|_2), \label{L2H1}
\end{align}
where \eqref{Linf} holds for sufficiently small $\sigma>0$.
\end{lemma}

Lemma \ref{lem.bound} and mass conservation imply that
$\|\rho_\sigma(t)\|_p$ is bounded for all $t>0$ and $1\le p\le \infty$.
Observe that $\kappa_\sigma\to 1$ as $\sigma\to 0$. So, if $\rho_\sigma(t)
\to\rho(t)$ a.e., the dissipation of the $L^\infty$ norm follows, as stated
in Theorem \ref{thm.ex} (iv).

\begin{proof}
The convexity of $F$ shows that $G$, defined in Lemma \ref{lem.energy}, is
nondecreasing. Therefore, $(d/dt)\int_{\R^d}F(\rho_\sigma)dx\le 0$ and
$$
  \sup_{t>0}\int_{\R^d}F(\rho_\sigma(t))dx \le \int_{\R^d}F(\rho^0_\sigma)dx.
$$
We choose a convex function $F\in C^2([0,\infty))$ such that 
$F(u)=0$ for $u\le\|\rho_\sigma^0\|_\infty$,
$F(u)>0$ for $u>\|\rho_\sigma^0\|_\infty$ and satisfying $F(u)\le Cu$ for $u\to\infty$.
Then
$$
  0 \le\int_{\R^d}F(\rho_\sigma(t))dx \le \int_{\R^d}F(\rho^0_\sigma)dx 
	= 0\quad\mbox{for }t>0.
$$
Consequently, $\rho_\sigma(x,t)\le\|\rho_\sigma^0\|_\infty\le\kappa_\sigma
\|\rho^0\|_\infty$ for $t>0$, 
showing the $L^\infty(\R^d)$ bound. Finally, choosing $F(u)=u^2$ in
Lemma \ref{lem.energy}, the $L^2(0,T;H^1(\R^d))$ estimate follows.
\end{proof}

%%%%%%%%%%%%%%%%%%

\subsection{Entropy and moment estimates}\label{ss.entropy}

We need a fractional derivative estimate for $f_\sigma(\rho_\sigma)$, which
is not an immediate consequence of Lemma \ref{lem.energy}. 
To this end, we define the entropy density 
$$
  h_\sigma(u) = \int_0^u\int_1^v\frac{f_\sigma'(w)}{w}dwdv, \quad u\ge 0.
$$

\begin{lemma}[Entropy balance]%\label{lem.ent}
It holds for all $t>0$ that
$$
  \frac{d}{dt}\int_{\R^d}h_\sigma(\rho_\sigma)dx
	+ 4\sigma\int_{\R^d}f'_\sigma(\rho_\sigma)|\na\rho_\sigma^{1/2}|^2 dx
	+ \int_{\R^d}|\na(-\Delta)^{-s/2}f_\sigma(\rho_\sigma)|^2dx = 0.
$$
In particular, for all $T>0$, there exists $C>0$ such that
\begin{equation}\label{u.H1s}
  \|f_\sigma(\rho_\sigma)\|_{L^2(0,T;H^{1-s}(\R^d))} \le C.
\end{equation}
\end{lemma}

\begin{proof}
The idea is to apply Lemma \ref{lem.energy}. 
Since $h_\sigma\not\in C^2([0,\infty))$, we cannot use the lemma directly.
Instead, we apply it to the regularized function
$$
  h_\sigma^\delta(u) = \int_0^u\int_1^v\frac{f_\sigma'(w)}{w+\delta}dwdv, \quad u\ge 0,
$$
where $\delta>0$. Choosing $F=h_\sigma^\delta$ in Lemma \ref{lem.energy} gives
\begin{align}
  \frac{d}{dt}\int_{\R^d}&h_\sigma^\delta(\rho_\sigma)dx 
	+ 4\sigma\int_{\R^d}f_\sigma'(\rho_\sigma)\frac{\rho_\sigma}{\rho_\sigma+\delta}
	|\na\rho_\sigma^{1/2}|^2 dx \label{3.aux0} \\
	&=- \frac{c_{d,1-s}}{2}\int_{\R^d}\int_{\R^d}\frac{(f_\sigma^\delta(\rho_\sigma(x))
	- f_\sigma^\delta(\rho_\sigma(y)))(f_\sigma(\rho_\sigma(x))-f_\sigma(\rho_\sigma(y))
	}{|x-y|^{d+2(1-s)}}dxdy, \nonumber
\end{align}
where $f_\sigma^\delta(u):=\int_0^u (v/(v+\delta))f_\sigma'(v)dv$ for $u\ge 0$.

{\em Step 1: Estimate of $h_\sigma^\delta$.} The pointwise limit
$h_\sigma^\delta(\rho_\sigma)\to h_\sigma(\rho_\sigma)$ holds a.e.\ in 
$\R^d\times(0,T)$ as $\delta\to 0$. We observe that for all $0<u\le 1$, 
$$
  |h_\sigma^\delta(u)| \le \sup_{0<v<1}f'(v)\int_0^u\int_v^1\frac{dw}{w}dv
	\le Cu(|\log u|+1),
$$
while for all $u>1$, since $f_\sigma'\ge 0$ in $[0,\infty)$,
\begin{align*}
  |h_\sigma^\delta(u)| &\le \int_0^1\int_v^1\frac{f_\sigma'(w)}{w+\delta}dwdv
	+ \int_1^u\int_1^v\frac{f_\sigma'(w)}{w+\delta}dwdv \\
	&\le C + \int_1^u\int_1^v f_\sigma'(w)dwdv
	\le C + \int_0^u f_\sigma(v)dv 
	\le C + uf_\sigma(u). 
\end{align*}
The last inequality follows after integration of $f_\sigma(v)\le f_\sigma(v)
+ vf_\sigma'(v) = (vf_\sigma(v))'$ in $(0,u)$. Therefore, since
$\rho_\sigma\le\|\rho_\sigma^0\|_\infty$ a.e.\ in $\R^d\times(0,\infty)$,
we find that
$$
  |h_\sigma^\delta(\rho_\sigma)| \le C\rho_\sigma(|\log\rho_\sigma|+1)
	\mathrm{1}_{\{\rho_\sigma\le 1\}} + C\mathrm{1}_{\{\rho_\sigma>1\}}
	\le C(\rho_\sigma^\theta + \rho_\sigma),
$$
where $\theta\in(0,1)$ is arbitrary, and consequently, because of mass conservation,
\begin{equation}\label{3.hsigma}
  \int_{\R^d}|h_\sigma^\delta(\rho_\sigma)|dx
	\le C + C\int_{\R^d}\rho_\sigma^\theta dx.
\end{equation}

{\em Step 2: Estimate of $\int_{\R^d}\rho_\sigma^\theta dx$.}
Let $0<\alpha<1$ and $d/(d+\alpha)<\theta<1$. Then, by Young's inequality,
\begin{align*}
  \int_{\R^d}\rho_\sigma^\theta dx &= \int_{\R^d}(1+|x|^2)^{\alpha\theta/2}
	\rho_\sigma^\theta(1+|x|^2)^{-\alpha\theta/2}dx \\
	&\le \int_{\R^d}(1+|x|^2)^{\alpha/2}\rho_\sigma dx
	+ C\int_{\R^d}(1+|x|^2)^{-\alpha\theta/(2(1-\theta))}dx \\
	&\le \int_{\R^d}(1+|x|^2)^{\alpha/2}\rho_\sigma dx + C,
\end{align*}
since the choice of $\theta$ guarantees that $-\alpha\theta/(2(1-\theta))<-d/2$,
so $\int_{\R^d}(1+|x|^2)^{-\alpha\theta/(2(1-\theta))}dx<\infty$.
To control the right-hand side, we need to bound a suitable moment of $\rho_\sigma$.
For this, we use the test function $(1+|x|^2)^{\alpha/2}$ in the weak
formulation of \eqref{1.rho}. (Actually, we need to use a cutoff to guarantee
integrability, but we leave the technical details to the reader.) We find that
\begin{align*}
  \int_{\R^d}(1+|x|^2)^{\alpha/2}\rho_\sigma(x,t)dx
	&= \int_{\R^d}(1+|x|^2)^{\alpha/2}\rho_\sigma^0 dx 
	+ \sigma\int_0^t\int_{\R^d}\rho_\sigma
	\Delta(1+|x|^2)^{\alpha/2}dxds \\
	&\phantom{xx}{}-\alpha\int_0^t\int_{\R^d}\rho_\sigma(1+|x|^2)^{\alpha/2-1}x
	\cdot\na(-\Delta)^{-s}f_\sigma(\rho_\sigma)dxds.
\end{align*}
Since $\alpha<1$, the terms $\Delta(1+|x|^2)^{\alpha/2}$ and $x(1+|x|^2)^{\alpha/2-1}$
are bounded in $\R^d$. Thus, taking into account the assumption on $\rho^0$
and mass conservation,
$$
  \sup_{0<t<T}\int_{\R^d}(1+|x|^2)^{\alpha/2}\rho_\sigma(x,t)dx
	\le C + C\int_0^T\int_{\R^d}\rho_\sigma(-\Delta)^{-s/2}|\na(-\Delta)^{-s/2}
	f_\sigma(\rho_\sigma)|dxdt.
$$
Next, we apply the Hardy--Littlewood--Sobolev inequality (see Appendix \ref{sec.frac})
and the Young inequality (see Lemma \ref{lem.hls}) 
and use the fact that $\rho_\sigma(t)$ is bounded in any $L^p(\R^d)$:
\begin{align*}
  \sup_{0<t<T}&\int_{\R^d}(1+|x|^2)^{\alpha/2}\rho_\sigma(x,t)dx \\
	&\le  C + \int_0^T\|\rho_\sigma\|_{2d/(d+2s)}\|(-\Delta)^{-s/2}[\na(-\Delta)^{-s/2}
	f_\sigma(\rho_\sigma)]\|_{2d/(d-2s)} dt \\
	&\le C + \int_0^T\|\rho_\sigma\|_{2d/(d+2s)}\|\na(-\Delta)^{-s/2}
	f_\sigma(\rho_\sigma)\|_2 dt \\
	&\le C(\eta) + \eta\int_0^T\|\na(-\Delta)^{-s/2}f_\sigma(\rho_\sigma)\|_2^2 dt 
\end{align*}
for all $\eta>0$. This proves that
$$
  \int_{\R^d}\rho_\sigma^\theta dx 
	\le C(\eta) + \eta\int_0^T\|\na(-\Delta)^{-s/2}f_\sigma(\rho_\sigma)\|_2^2 dt.
$$

{\em Step 3: A priori estimate.}
Inserting the previous estimate into \eqref{3.hsigma} leads to
$$
  \sup_{0<t<T}\int_{\R^d}|h_\sigma^\delta(\rho_\sigma(x,t))|dx
	\le C(\eta) + \eta\int_0^T\|\na(-\Delta)^{-s/2}f_\sigma(\rho_\sigma)\|_2^2 dt.
$$
We integrate \eqref{3.aux0} in time and use the previous estimate:
\begin{align*}
	4\sigma&\int_0^T\int_{\R^d}f_\sigma'(\rho_\sigma)
	\frac{\rho_\sigma}{\rho_\sigma+\delta}
	|\na\rho_\sigma^{1/2}|^2 dxdt \\
	&\phantom{xx}{}
	+ \frac{c_{d,1-s}}{2}\int_0^T\int_{\R^d}\int_{\R^d}
	\frac{(f_\sigma^\delta(\rho_\sigma(x))
	- f_\sigma^\delta(\rho_\sigma(y)))(f_\sigma(\rho_\sigma(x))-f_\sigma(\rho_\sigma(y))
	}{|x-y|^{d+2(1-s)}}dxdydt \\
	&\le \int_{\R^d}|h_\sigma^\delta(\rho_\sigma(T))|dx 
	+ \int_{\R^d}|h_\sigma^\delta(\rho_\sigma^0)|dx
	\le C(\eta) + \eta\int_0^T\|\na(-\Delta)^{-s/2}f_\sigma(\rho_\sigma)\|_2^2 dt.
\end{align*}
We wish to pass to the limit $\delta\to 0$ in the previous inequality.
We deduce from dominated convergence 
that $f_\sigma^\delta(\rho_\sigma)\to f_\sigma(\rho_\sigma)$ a.e.\ in $\R^d\times
[0,\infty)$. The integrand of the second term on the left-hand side is nonnegative,
and we obtain from Fatou's lemma that
\begin{align}
	4\sigma&\int_0^T\int_{\R^d}f_\sigma'(\rho_\sigma)|\na\rho_\sigma^{1/2}|^2 dxdt
	+ \frac{c_{d,1-s}}{2}\int_0^T\int_{\R^d}\int_{\R^d}
	\frac{(f_\sigma(\rho_\sigma(x))-f_\sigma(\rho_\sigma(y)))^2}{|x-y|^{d+2(1-s)}}dxdydt 
	\label{3.aux} \\
	&\le C(\eta) + \eta\int_0^T\|\na(-\Delta)^{-s/2}f_\sigma(\rho_\sigma)\|_2^2 dt. 
	\nonumber
\end{align}
By the integral representation of the fractional Laplacian, 
$$
  \frac{c_{d,1-s}}{2}\int_{\R^d}\int_{\R^d}
	\frac{(f_\sigma(\rho_\sigma(x))-f_\sigma(\rho_\sigma(y)))^2}{|x-y|^{d+2(1-s)}}dxdy
	= \|\na(-\Delta)^{-s/2}f_\sigma(\rho_\sigma)\|_2^2,
$$
the last term in \eqref{3.aux} can be absorbed for sufficiently small $\eta>0$
by the second term on the left-hand side. This leads to the estimate
$$
  4\sigma\int_0^T\int_{\R^d}f_\sigma'(\rho_\sigma)|\na\rho_\sigma^{1/2}|^2 dxdt
	+ \int_0^T\int_{\R^d}|\na(-\Delta)^{-s/2}f_\sigma(\rho_\sigma)|^2 dxdt \le C.
$$
Thus, we can pass to the limit $\delta\to 0$ in \eqref{3.aux0} giving the
desired entropy balance. Finally, bound \eqref{u.H1s} follows from the 
definition of the $H^{1-s}(\R^d)$ norm and the facts that 
$f_\sigma(\rho_\sigma) \in L^2(\R^d)$ since $f_\sigma$ is locally Lipschitz 
continuous, $f_\sigma(0)=0$, and $\rho_\sigma $ is bounded both in 
$L^\infty(\R^d)$ and $L^2(\R^d)$ independently of $\sigma$.
\end{proof}

\begin{lemma}[Moment estimate]\label{lem.mom}
It holds that 
$$
  \sup_{0<t<T}\int_{\R^d}\rho_\sigma(x,t)|x|^{2d/(d-2s)}dx
	\le C,
$$
where $C>0$ depends on $T$ and the $L^1(\R^d)$ norms of $\rho^0$ and
$|\cdot|^{2d/(d-2s)}\rho^0$.
\end{lemma}

\begin{proof}
For the following computations, we would need to use cut-off functions 
to make the calculations rigorous. We leave the details to the reader, as we wish
to simplify the presentation. Let $m=2d/(d-2s)$. Since 
$|\cdot|^m\rho^0\in L^1(\R^d)$ by assumption, we can compute
\begin{align}
  \frac{d}{dt}\int_{\R^d}\rho_\sigma(t)\frac{|x|^m}{m}dx
	&= \sigma(m-2+d)\int_{\R^d}|x|^{m-2}\rho_\sigma dx
	- \int_{\R^d}\rho_\sigma|x|^{m-2}x\cdot\na(-\Delta)^{-s}f_\sigma(\rho_\sigma)dx 
	\nonumber \\
	&\le C\||\cdot|^{m-2}\rho_\sigma\|_{1}
	+ \||\cdot|^{m-1}\rho_\sigma\|_{2d/(d+2s)}
	\|\na(-\Delta)^{-s}f_\sigma(\rho_\sigma)\|_{2d/(d-2s)}. \label{3.aux2}
\end{align}
By Young's inequality and mass conservation, we have
$$
  \||\cdot|^{m-2}\rho_\sigma\|_{1} \le C\int_{\R^d}(1+|x|^{m})\rho_\sigma dx
	\le C + C\int_{\R^d}|x|^m\rho_\sigma dx.
$$
It follows from \eqref{u.H1s} that $\na(-\Delta)^{-s}f_\sigma(\rho_\sigma)$
is bounded in $L^2(0,T; H^s(\R^d))$. In particular, because of the 
Sobolev embedding $H^{s}(\R^d)\hookrightarrow L^{m}(\R^d)$, 
$$
  \|\na(-\Delta)^{-s}f_\sigma(\rho_\sigma)\|_{L^2(0,T;L^{m}(\R^d))}\le C.
$$
Furthermore, using $\rho_\sigma\in L^\infty(0,\infty;L^\infty(\R^d))$,
Young's inequality, and the property $2d/(d+2s)\ge 1$ (recall that $d\ge 2$)
\begin{align*}
  \big\||\cdot|^{m-1}\rho_\sigma\big\|_{2d/(d+2s)}^{2d/(d+2s)}
	&= \int_{\R^d}\rho_\sigma^{2d/(d+2s)}|x|^{2d(m-1)/(d+2s)}dx \\
	&\le C + C\int_{\R^d}\rho_\sigma|x|^{2d(m-1)/(d+2s)}dx.
\end{align*}
Thus, we infer from \eqref{3.aux2} and the identity $2d(m-1)/(d+2s)=m$ that
$$
  \frac{d}{dt}\int_{\R^d}\rho_\sigma(t)\frac{|x|^m}{m}dx
	\le C + C\int_{\R^d}\rho_\sigma(t)|x|^m dx,
$$
and Gronwall's lemma concludes the proof.
\end{proof}

%%%%%%%%%%%%%%%%%%

\subsection{Higher-order estimate}\label{ss.highreg}

We need some estimates in higher-order Sobolev spaces.

\begin{proposition}[Higher-order regularity]\label{prop.regul}
Let $T>0$, $1<p<\infty$ and $2\le q<\infty$. Then there exists $C>0$,
depending on $T$, $\sigma$, $p$, and $q$, such that 
$$
  \|\rho_\sigma\|_{L^p(0,T;W^{3,p}(\R^d))}
	+ \|\pa_t\rho_\sigma\|_{L^p(0,T;W^{1,p}(\R^d))}
	+ \|\rho_\sigma\|_{C^0([0,T];W^{2,q}(\R^d))} \le C.
$$
\end{proposition}

\begin{proof}
{\em Step 1: Case $s>1/2$.}
If $s>1/2$ then $w:=\rho_\sigma\na(-\Delta)^{-s}f_\sigma(\rho_\sigma)$ does not
involve any derivative of $\rho_\sigma$. Thus $w\in L^p(0,T;L^p(\R^d))$ for $p<\infty$
and Lemma \ref{lem.regul} in Appendix \ref{sec.regul} implies that 
$\rho_\sigma\in L^p(0,T;W^{1,p}(\R^d))$. Iterating the argument leads to the conclusion.
Thus, in the following, we can assume that $0<s\le 1/2$.

{\em Step 2: Estimate of $\diver w$ in $L^p(0,T;W^{-1,p}(\R^d))$.}
We claim that $w$ can be estimated in 
$L^p(0,T;L^p(\R^d))$ for any $p<\infty$. 
Then, by Lemma \ref{lem.regul}, $\na\rho_\sigma\in L^p(0,T;L^p(\R^d))$.
We use the $L^\infty$ bound for $\rho_\sigma$, the fractional Gagliardo--Nirenberg 
inequality (Lemma \ref{lem.GN1}), and Young's inequality to find that
$$
  \|w\|_p
	\le C\|\na(-\Delta)^{-s}f_\sigma(\rho_\sigma)\|_p
	\le C\|f_\sigma(\rho_\sigma)\|_p^{2s}\|\na f_\sigma(\rho_\sigma)\|_p^{1-2s}
	\le C(\eta) + \eta\|\na\rho_\sigma\|_p,
$$
where $\eta>0$ is arbitrary. By estimate \eqref{a.D1u} in Lemma \ref{lem.regul}, 
$$
  \|\rho_\sigma\na(-\Delta)^{-s}f_\sigma(\rho_\sigma)\|_p = \|w\|_p
	\le C(\eta) + \eta\big(\|\rho_\sigma\na(-\Delta)^{-s}f_\sigma(\rho_\sigma)\|_p
	+ T^{1/p}\|\na\rho^0\|_p\big).
$$
Choosing $\eta>0$ sufficiently small shows the claim. 

{\em Step 3: Estimate of $\diver w$ in $L^p(0,T;L^{p}(\R^d))$.}
%In case $s>1/2$, the regularity $\na\rho_\sigma\in L^p(0,T;L^p(\R^d))$ 
%immediately implies that $\diver w\in L^p(0,T;L^p(\R^d))$ for any $p<\infty$.
%Thus, we assume that $0<s\le 1/2$. 
We use H\"older's inequality with $1/p=2s/(d+p)+1/q$ to obtain
\begin{align*}
  \|\diver w\|_p &\le \|\na\rho_\sigma\cdot\na(-\Delta)^{-s}f_\sigma(\rho_\sigma)\|_p
	+ \|\rho_\sigma(-\Delta)^{1-s}f_\sigma(\rho_\sigma)\|_p \\
	&\le \|\na\rho_\sigma\|_{(d+p)/(2s)}\|\na(-\Delta)^{-s}f_\sigma(\rho_\sigma)\|_q
	+ C\|(-\Delta)^{1-s}f_\sigma(\rho_\sigma)\|_p.
\end{align*}
By the fractional Gagliardo--Nirenberg inequality (Lemma \ref{lem.GN2} with 
$\theta=1+d/p-d/q-2s$ and Lemma \ref{lem.GN1} with $s$ replaced by $1-s$)
and Young's inequality, it follows that
\begin{align*}
  \|\diver w\|_p &\le C\|\na\rho_\sigma\|_{(d+p)/(2s)}
	\|f_\sigma(\rho_\sigma)\|_p^{1-\theta}\|\na f_\sigma(\rho_\sigma)\|_p^\theta
	+ C\|f_\sigma(\rho_\sigma)\|_p^{s}\|\DD^2 f_\sigma(\rho_\sigma)\|_p^{1-s} \\
	&\le C\|\na\rho_\sigma\|_{(d+p)/(2s)}\|\na\rho_\sigma\|_p^{\theta} 
	+ C\|f_\sigma'(\rho_\sigma)\DD^2\rho_\sigma
	+ f_\sigma''(\rho_\sigma)\na\rho_\sigma\otimes\na\rho_\sigma\|_p^{1-s} \\
	&\le C(\eta) + C\|\na\rho_\sigma\|_{(d+p)/(2s)}^{1/(1-\theta)}
	+ C\|\na\rho_\sigma\|_p + C\|\na\rho_\sigma\|_{2p}^2 + \eta\|\DD^2\rho_\sigma\|_p,
\end{align*}
where $\eta>0$ is arbitrary. Taking the $L^p(0,T)$ norm of the previous inequality
and observing that $p/(1-\theta)=(d+p)/(2s)$ (because of $\theta=d(1/p-1/q)+1-2s$),
it follows that
\begin{align*}
  \|\diver w\|_{L^p(0,T;L^p(\R^d))} &\le C 
	+ C\|\na\rho_\sigma\|_{L^{(d+p)/(2s)}(0,T;L^{(d+p)/(2s)}(\R^d))}^{1/(1-\theta)} 
	+ C\|\na\rho_\sigma\|_{L^p(0,T;L^p(\R^d))} \\
	&\phantom{xx}{}+ C\|\na\rho_\sigma\|_{L^{2p}(0,T;L^{2p}(\R^d))}^2
	+ \eta\|\DD^2\rho_\sigma\|_{L^p(0,T;L^p(\R^d))}.
\end{align*}
Lemma \ref{lem.regul} and Step 2 ($\na\rho_\sigma\in L^p(0,T;L^p(\R^d))$) show that
$$
  \|\pa_t\rho_\sigma\|_{L^p(0,T;L^p(\R^d))}
	+ (1-C\eta)\|\DD^2\rho_\sigma\|_{L^p(0,T;L^p(\R^d))}\le C.
$$
Choosing $\eta>0$ sufficiently small, this yields 
$\pa_t\rho_\sigma\in L^p(0,T;L^p(\R^d))$ and $\rho_\sigma\in L^p(0,T;$ $W^{2,p}(\R^d))$.
We deduce from Lemma \ref{lem.embedd}, applied to $\na\rho_\sigma$,
that $\na\rho_\sigma\in L^\infty(0,T;L^q(\R^d))$ for any $2\le q<\infty$.
(At this point, we need the restriction $q\ge 2$.)

{\em Step 4: Higher-order regularity.}
To improve the regularity of $\rho_\sigma$, we differentiate \eqref{1.rho} in space.
Recall that $\pa_i=\pa/\pa x_i$, $i=1,\ldots,d$. Then
\begin{align}
  \pa_t &\pa_i\rho_\sigma - \sigma\Delta \pa_i\rho_\sigma
	= \sum_{j=1}^d\pa_i\pa_j\big(\rho_\sigma\pa_j(-\Delta)^{-s}f_\sigma(\rho_\sigma)\big)
	= \sum_{j=1}^d\big(\pa_{ij}^2\rho_\sigma\pa_j(-\Delta)^sf_\sigma(\rho_\sigma) 
	\nonumber \\
	&{}+ \pa_i\rho_\sigma\pa_{jj}^2(-\Delta)^{-s}f_\sigma(\rho_\sigma)
	+ \pa_j\rho_\sigma\pa_{ij}^2(-\Delta)^{-s}f_\sigma(\rho_\sigma)
	+ \rho_\sigma\pa_{ijj}^3(-\Delta)^{-s}f_\sigma(\rho_\sigma)\big). \label{3.aux4}
\end{align}
We estimate the right-hand side term by term. Let $0<s\le 1/2$.
First, by H\"older's inequality
with $1/p=1/q+1/r$, $1<p<q<\infty$, $\max\{2,p\}<r<\infty$ and the
fractional Gagliardo--Nirenberg inequality (Lemma \ref{lem.GN1}),
\begin{align*}
  \|\pa_{ij}^2&\rho_\sigma\pa_j(-\Delta)^sf_\sigma(\rho_\sigma)\|_{L^p(0,T;L^p(\R^d))}^p
	\le \int_0^T\|\pa_{ij}^2\rho_\sigma\|_q^p
	\|\pa_j(-\Delta)^sf_\sigma(\rho_\sigma)\|_r^p dt \\
	&\le C\int_0^T\|\pa_{ij}^2\rho_\sigma\|_q^p\|f_\sigma(\rho_\sigma)\|_r^{(1-2s)p}
	\|\na f_\sigma(\rho_\sigma)\|_r^{2sp}dt \\
	&\le C\|f_\sigma(\rho_\sigma)\|_{L^\infty(0,T;L^r(\R^d))}^{(1-2s)p}
	\|\na f_\sigma(\rho_\sigma)\|_{L^\infty(0,L^r(\R^d))}^{2sp}
	\int_0^T\|\pa_{ij}^2\rho_\sigma\|_q^pdt \le C.
\end{align*}
The second and third term on the right-hand side of \eqref{3.aux4} can be treated
in a similar way, observing that $\pa_{ij}^2(-\Delta)^{-s}=\pa_j(-\Delta)^{-s}\pa_i$.
The last term is estimated according to
\begin{align*}
  \|\rho_\sigma&\pa_{ijj}^3(-\Delta)^{-s}f_\sigma(\rho_\sigma)\|_p
	\le C\|\pa_{ijj}^3(-\Delta)^{-s}f_\sigma(\rho_\sigma)\|_p
	\le C\|\pa_{jj}^2 f_\sigma(\rho_\sigma)\|_p^{2s}
	\|\na\pa_{jj}^2f_\sigma(\rho_\sigma)\|_p^{1-2s} \\
	&\le C(\eta)\|\pa_{jj}^2 f_\sigma(\rho_\sigma)\|_p
	+ \eta\|\na\pa_{jj}^2f_\sigma(\rho_\sigma)\|_p,
\end{align*}
and the last expression can be absorbed by the corresponding estimate of
$\Delta\pa_i\rho_\sigma$ from the left-hand side of \eqref{3.aux4}. Then
we deduce from Lemma \ref{lem.regul} that
$\pa_t\pa_i\rho_\sigma$, $\pa_{ijj}^3\rho_\sigma\in L^p(0,T;$ $L^p(\R^d))$ for
all $p>1$ and Lemma \ref{lem.embedd}, applied to $\pa_{ij}^2\rho_\sigma$, yields
$\pa_{ij}^2\rho_\sigma\in C^0([0,T];L^q(\R^d))$ for all $q\ge 2$.

Next, if $1/2<s<1$, we use the second inequality in Lemma \ref{lem.GN1} and
argue similarly as before. This finishes the proof.
\end{proof}

\begin{lemma}\label{lem.regulrho}
Under the assumptions of Proposition \ref{prop.regul}, 
for every $q\geq 2$, there exists a constant
$C=C(q)>0$, depending on $\sigma$, such that 
$$
  \|\rho_\sigma\|_{C^0([0,T];W^{2,1}(\R^d)\cap W^{3,q}(\R^d))} \le C.
$$
\end{lemma}

The embedding $W^{3,q}(\R^d)\hookrightarrow W^{2,\infty}(\R^d)$
for $q>d$ yields a bound for $\rho_{\sigma}$ in 
$C^0([0,T];$ $W^{2,\infty}(\R^d))$.

\begin{proof}
We first prove the bound in $C^0([0,T];W^{3,q}(\R^d))$.
By differentiating \eqref{1.rho} twice in space, estimating similarly
as in Step 4 of the previous proof, and using the regularity results
of Proposition \ref{prop.regul}, we can show that $\rho_\sigma$ is bounded
in $L^\infty(0,T;W^{3,q}(\R^d))$ for any $q\ge 2$.  

It remains to show the $C^0([0,T];W^{2,1}(\R^d))$ bound for $\rho_\sigma$. 
In view of mass conservation and Gagliardo--Nirenberg--Sobolev's inequality, 
it suffices to show a bound for 
$\DD^2 \rho_{\sigma}$ in $L^\infty(0,T;L^{1}(\R^d))$. 
To this end, we define the weights
$\gamma_n=(1+|x|^2)^{n/2}$ for $n\ge 0$ 
and test equation \eqref{1.rho} for $\rho_\sigma$
with $v_n:=\gamma_n\rho_{\sigma}$. Then
\begin{align*}
  & \pa_t v_n - \sigma\Delta v_n 
	= \diver\big(v_n\na\K*f_\sigma(\rho_{\sigma})\big)
	+ I_n, \quad v_n(0)=\gamma_n\rho^{0}_\sigma\quad\mbox{in }\R^d, \\
	& \mbox{where }I_n = -2\sigma\na\gamma_n\cdot\na\rho_{\sigma} 
	- \sigma\rho_{\sigma}\Delta\gamma_n
	- \rho_{\sigma}\na\gamma_n
	\cdot\na\K*f_\sigma(\rho_{\sigma}).
\end{align*}
Arguing as in Step 4 of the previous proof, we can find a bound in 
$L^\infty(0,T;W^{2,p}(\R^d))$ for $v_n$.
Indeed, we can proceed by induction over $n$, since the additional terms in $I_n$
can be controlled by Sobolev norms of $v_0,\ldots,v_{n-1}$. The definition of
$\rho^0_\sigma$ implies that $\gamma_n\rho^0_\sigma$, 
$\gamma_n\nabla\rho^0_\sigma\in L^\infty(\R^d)\cap
L^1(\R^d)$ for every $n\ge 0$. Then choosing $n>d$ yields, for $0\le t\le T$, that
$$
  \|\gamma_n\DD^2\rho_{\sigma}\|_p
	\le\|\DD^2(\gamma_n\rho_{\sigma})\|_p
	+ 2\|\na\gamma_n\cdot\na\rho_{\sigma}\|_p
	+ \|\rho_{\sigma}\DD^2\gamma_n\|_p \le C(T).
$$
We conclude from $\gamma_n^{-1}\in L^\infty(\R^d)\cap L^1(\R^d)$ that
$$
  \|\DD^2\rho_{\sigma}\|_1
	\le \|\gamma_n^{-1}\|_{p/(p-1)}\|\gamma_n\DD^2\rho_{\sigma}\|_p
	\le C(T).
$$
This proves the desired bound.
\end{proof}

%%%%%%%%%%%%%%%%%

\subsection{Existence of solutions to \eqref{1.rho}}

We show that the regularized equation \eqref{1.rho} possesses a unique 
strong solution $\rho_\sigma$. 

{\em Step 1: Existence for an approximated system.}
Let $T>0$ arbitrary, define the spaces
\begin{align*}
  X_T &:= L^2(0,T; H^1(\R^d))\cap H^1(0,T; H^{-1}(\R^d))
	\hookrightarrow Y_T := C^0([0,T]; L^2(\R^d)), \\
  Y_{T,R} &:= \{u\in Y_T:\|u-\rho_\sigma^0\|_{L^\infty(0,T; L^2(\R^d))}\leq R\},
\end{align*}
and consider the mapping $S : v\in Y_T\mapsto u\in Y_T$,
\begin{equation}
\label{ex.1}
\begin{aligned}
  & \pa_t u - \sigma\Delta u = \diver(u\na \K_s^{(\delta)}*f_\sigma^{(\eta)}(v))  
	\quad\mbox{in }\R^d\times (0,T), \\
  & u(0) = \rho_\sigma^0\quad\mbox{in }\R^d ,
\end{aligned}
\end{equation}
where $\K_s^{(\delta)} : \R^d\to \R_+$ is a regularized version of $\K_s$, 
defined by
\begin{align}\nonumber %\label{ex.Kdelta}
  \K_s^{(\delta)} &= \widetilde\K_{s/2}^{(\delta)}*\widetilde\K_{s/2}^{(\delta)}, \\
  \nonumber %\label{ex.Kdelta2}
  \widetilde\K_{s/2}^{(\delta)}(x) &= c_{d,-s/2}
  \begin{cases}
    \delta^{s-d} + (s-d)\delta^{s-d-1}(|x|-\delta) & \mbox{for }|x|<\delta, \\
    |x|^{s-d} & \mbox{for }\delta\leq |x|\leq \delta^{-1}, \\
    [\delta^{d-s} + (s-d)\delta^{d+1-s}(|x|-\delta^{-1})]_+ & 
		\mbox{for }|x|>\delta^{-1},
  \end{cases}
\end{align}
and $f_\sigma^{(\eta)}$ is given by
\begin{equation*}%\label{ex.fsd}
  f_\sigma^{(\eta)}(\rho) = \int_0^{|\rho|} f_\sigma'(u)\min(1,u\eta^{-1})du 
	+ \frac{\eta}{2} \rho^2 ,\quad\rho\in\R.
\end{equation*}
The regularization with parameter $\eta$ is needed for the entropy estimates.

We derive some estimates for $f_\sigma^{(\eta)}$.
First, we have $0\leq f_\sigma^{(\eta)}(\rho)\leq C_\eta\rho^2$ 
for $\rho\in\R$, since
\begin{align*}
  f_\sigma^{(\eta)}(\rho) &
	\leq \left(\eta + \eta^{-1}\max_{[0,\eta]}f_\sigma'\right)\frac{\rho^2}{2}
	\quad\mbox{for }|\rho|\le\eta, \\
  f_\sigma^{(\eta)}(\rho) &\leq f_\sigma(|\rho|) + \frac{\eta}{2} \rho^2
  \leq \left(\|f_\sigma\|_\infty\eta^{-2} + \frac{\eta}{2}\right)\rho^2
  \quad\mbox{for }|\rho|>\eta.
\end{align*}
Furthermore, 
$$
  |\DD f_\sigma^{(\eta)}(\rho)| = \bigg| \frac{\rho}{|\rho|}f_\sigma'(|\rho|)
	\min(1,|\rho|\eta^{-1}) + \eta\rho\bigg|
	\leq (\eta + \eta^{-1}\|f_\sigma'\|_\infty)|\rho|,
$$
which implies that $|\DD f_\sigma^{(\eta)}(\rho)|\leq C_\eta|\rho|$ 
for $\rho\in\R$. This shows that there exists $C(\eta)>0$ such that
for any $\rho_1$, $\rho_2\in\R$,
$$
  |f_\sigma^{(\eta)}(\rho_1)-f_\sigma^{(\eta)}(\rho_2)|
	\leq C(\eta)(|\rho_1|+|\rho_2|)|\rho_1-\rho_2|.
$$
It follows that $f_\sigma^{(\eta)}(v)\in L^\infty(0,T; L^1(\R^d))$ for $v\in Y_T$. 

Since $\na\K_s^{(\delta)}\in L^\infty(\R^d)$, a standard argument shows that 
\eqref{ex.1} has a unique solution $u\in X_T\hookrightarrow Y_T$. Therefore, 
the mapping $S$ is well-defined. Additionally, the nonnegativity of $u$ follows 
immediately after by testing \eqref{ex.1} with $\min(0,u)$.

We show now that $S$ is a contraction on $Y_{T,R}$ for sufficiently small $T>0$.
We start with a preparation.
By testing \eqref{ex.1} with $u$ and taking into account the $L^\infty$ bound 
for $\na\K_s^{(\delta)}$, we deduce from Young's inequality for products 
and convolutions that
\begin{equation*}%\label{ex.est.1}
  \int_{\R^d}u(t)^2 dx + \frac{\sigma}{2}\int_0^t\int_{\R^d}|\na u|^2 dx d\tau
  \leq \int_{\R^d}|\rho_\sigma^0|^2 dx 
	+ C(\delta,\eta,\sigma)\int_0^t\|u\|_2^2\|v\|_2^4 d\tau,
\end{equation*}
since $\|f_\sigma^{(\eta)}(v)\|_{1}\leq C_\eta \|v \|_{2}^2$ for $v \in Y_T$.
Then, if $v\in Y_{T,R}$, we infer from Gronwall's lemma that
\begin{equation}\label{ex.est.2}
  \int_{\R^d}u(t)^2 dx + \sigma\int_0^t\int_{\R^d}|\nabla u|^2 dx d\tau 
	\leq e^{C(\sigma,\delta,\eta)R^4 t}\int_{\R^d}|\rho_\sigma^0|^2 dx
	\quad\mbox{for } 0\leq t\leq T.
\end{equation}

Let $v_i\in Y_{T,R}$ and set $u_i = S(v_i)$, $i=1,2$. We compute
\begin{align*}
  \|u_1 & \nabla \K_s^{(\delta)} * f_\sigma^{(\eta)}(v_1) 
	- u_2\nabla \K_s^{(\delta)} * f_\sigma^{(\eta)}(v_2)\|_2\\
  &\leq \|(u_1-u_2)\nabla \K_s^{(\delta)} * f_\sigma^{(\eta)}(v_1)\|_2
  + \|u_2\nabla\K_s^{(\delta)} * (f_\sigma^{(\eta)}(v_1) 
	- f_\sigma^{(\eta)}(v_2))\|_2 \\
  &\leq \|u_1-u_2\|_2\|\nabla \K_s^{(\delta)} * f_\sigma^{(\eta)}(v_1)\|_\infty 
	+ \|u_2\|_2\|\nabla\K_s^{(\delta)} * (f_\sigma^{(\eta)}(v_1) 
	- f_\sigma^{(\eta)}(v_2))\|_\infty \\
  &\leq \|u_1-u_2\|_2\|\nabla \K_s^{(\delta)}\|_\infty 
	\|f_\sigma^{(\eta)}(v_1)\|_1 + \|u_2\|_2\|\nabla\K_s^{(\delta)}\|_\infty 
	\|f_\sigma^{(\eta)}(v_1) - f_\sigma^{(\eta)}(v_2)\|_1 \\
  &\leq C(\delta,\eta)\big(\|u_1-u_2\|_2\|v_1\|_2^2 + \|u_2\|_2(\|v_1\|_2 
	+ \|v_2\|_2)\|v_1-v_2\|_2\big).
\end{align*}
Therefore, using \eqref{ex.est.2}, for $v_1, v_2\in Y_{T,R}$,
\begin{equation}\label{ex.est.3}
  \|u_1 \nabla \K_s^{(\delta)} * f_\sigma^{(\eta)}(v_1) - u_2\nabla \K_s^{(\delta)} 
	* f_\sigma^{(\eta)}(v_2)\|_2 
	\leq C(\delta,\eta,R,T)(\|u_1-u_2\|_2 + \|v_1-v_2\|_2 ).
\end{equation}

Next, we write \eqref{ex.1} for $(u_i,v_i)$ in place of $(u,v)$, $i=1,2$, take the 
difference between the two equations, and test the resulting equation with $u_1-u_2$:
\begin{align*}
  \frac{1}{2}\| & u_1-u_2\|_2^2(t) 
	+ \sigma\int_0^t\int_{\R^d}|\nabla (u_1-u_2)|^2 dx d\tau \\
  &= -\int_0^t\int_{\R^d}\nabla(u_1-u_2)\cdot (u_1 \nabla \K_s^{(\delta)}
	* f_\sigma^{(\eta)}(v_1) - u_2\nabla \K_s^{(\delta)} * f_\sigma^{(\eta)}(v_2))
	dx d\tau \\
  &\leq  \frac{\sigma}{2}\int_0^t\int_{\R^d}|\nabla (u_1-u_2)|^2 dx d\tau
  + \frac{1}{2\sigma}\int_0^t\|u_1 \nabla \K_s^{(\delta)} 
	* f_\sigma^{(\eta)}(v_1) - u_2\nabla \K_s^{(\delta)} 
	* f_\sigma^{(\eta)}(v_2)\|_2^2 d\tau.
\end{align*}
It follows from \eqref{ex.est.3} that
$$
  \| u_1-u_2\|_2^2(t) + \sigma\int_0^t\int_{\R^d}|\nabla (u_1-u_2)|^2 dx d\tau\\
  \leq C(\delta,\eta,R,T,\sigma)\int_0^t(\|u_1-u_2\|_2^2+\|v_1-v_2\|_2^2)d\tau,
$$
and we conclude from Gronwall's lemma that
$$
  \| u_1-u_2\|_2^2(t)\leq e^{C(\delta,\eta,R,T,\sigma)t}\int_0^T\|v_1-v_2\|_2^2 
	d\tau\quad\mbox{for } 0\leq t\leq T.
$$
This inequality implies that $S$ is a contraction in $Y_{T,R}$, 
provided that $T$ is sufficiently small.
Therefore, by Banach's theorem, $S$ admits a unique fixed point 
$u\in Y_{T,R}\subset Y_T$ for $T>0$ sufficiently small.

It remains to show that the local solution can be extended to a global one.
To this end, we note that the function $u\in X_T$ satisfies \eqref{ex.1} with $v=u$:
\begin{equation}\label{ex.2}
  \begin{aligned}
  & \pa_t u - \sigma\Delta u = \diver(u\nabla \K_s^{(\delta)} 
	* f_\sigma^{(\eta)}(u))\quad\mbox{in }\R^d\times (0,T),\\
  & u(\cdot,0) = \rho_\sigma^0\quad\mbox{in }\R^d.
  \end{aligned}
\end{equation}
Then, defining the truncated entropy density
$$
  h^{(\eta)}(\rho) = \int_0^\rho\int_0^u \DD f_\sigma^{(\eta)}(v)v^{-1}dv du,
	\quad\rho\geq 0,
$$
and testing \eqref{ex.2} with $\DD h^{(\eta)}(u)$ yields,
in view of the definition of $\K_s^{(\delta)}$, that
  \begin{align}\label{ex.ei}
  \int_{\R^d}h^{(\eta)}(u(t))dx 
  &+ \sigma\int_0^t\int_{\R^d}\DD f_\sigma^{(\eta)}(u)u^{-1}|\nabla u|^2 dx d\tau \\   
	\nonumber
  &+ \int_0^t\int_{\R^d}|\nabla\widetilde\K_{s/2}^{(\delta)}
	* f_\sigma^{(\eta)}(u)|^2 dx d\tau
  = \int_{\R^d}h^{(\eta)}(\rho_\sigma^0)dx
\end{align}
for $0\le t\le T$.
This inequality and the definitions of $f_\sigma^{(\eta)}$ and $h^{(\eta)}$ 
yield a $(\delta,T)$-uniform bound for $u$ in $L^2(0,T; H^1(\R^d))$, which in turn 
(together with \eqref{ex.2}) implies a $(\delta,T)$-uniform bound for $u$ in $X_T$, 
and a fortiori in $Y_T$. This means that the solution $u$ can be prolonged to 
the whole time interval $[0,\infty)$ and exists for all times.

Finally, we point out that, since $\nabla\K_s^{(\delta)}\in L^2(\R^d)$, then 
$\nabla\K_s^{(\delta)} * f_\sigma^{(\eta)}(u)\in L^\infty(0,T; L^2(\R^d))$ and so
$u\nabla\K_s^{(\delta)} * f_\sigma^{(\eta)}(u)\in L^\infty(0,T; L^1(\R^d))$. 
This fact yields the conservation of mass for $u$, i.e.\ $\int_{\R^d}u(t)dx 
= \int_{\R^d}\rho_\sigma^0 dx$ for $t>0$. Indeed, it is sufficient to test 
\eqref{ex.2} with a cutoff $\psi_R \in C^1_0(\R^d)$ satisfying $\psi_R(x)=1$ 
for $|x|<R$, $\psi_R(x)=0$ for $|x|>2R$, $|\nabla\psi_R(x)|\leq C R^{-1}$ 
for $x\in\R^d$, and then to take the limit $R\to\infty$.

{\em Step 2: Limit $\delta\to 0$.} Let $u^{(\delta)}$ be the solution to \eqref{ex.2}.
An adaption of the proof of \cite[Lemma 1]{CGZ20} shows that the embedding
$H^1(\R^d)\cap L^1(\R^d;(1+|x|^2)^{\kappa/2})\hookrightarrow L^2(\R^d)$ is compact.
Thus, because of the $\delta$-uniform bounds for $u^{(\delta)}$, 
the Aubin--Lions Lemma 
implies that (up to a subsequence) $u^{(\delta)}\to u$ strongly in 
$L^2(0,T; L^2(\R^d))$ for every $T>0$. 
We wish now to study the convergence of the nonlinear and nonlocal terms in 
\eqref{ex.2}--\eqref{ex.ei} as $\delta\to 0$.

It follows from \eqref{ex.ei} that (up to a subsequence)
\begin{equation}\label{ex.delta.0}
  \nabla\widetilde\K_{s/2}^{(\delta)} * f_\sigma^{(\eta)}(u^{(\delta)})
	\rightharpoonup U\quad \mbox{weakly in }L^2(\R^d\times (0,T))
	\quad\mbox{as }\delta\to 0.
\end{equation}
In order to identify the limit $U$, we first notice that, by construction, 
$0\leq \widetilde\K_{s/2}^{(\delta)}\nearrow\K_{s/2}$ a.e.~in $\R^d$. 
Furthermore, the Hardy--Littlewood--Sobolev inequality, the bound for 
$f_\sigma^{(\eta)}$, and then the Gagliardo-Nirenberg-Sobolev inequality yield that
\begin{align*}
  \|\K_{s/2} * f_\sigma^{(\eta)}(u)\|_{(d+2)/(d-s)}
	&\leq C\|f_\sigma^{(\eta)}(u)\|_{(d+2)/(d+2s/d)}
	\leq C(\eta)\|u\|_{(2d+4)/(d+2s/d)}^2 \\
  &\leq C(\eta)\|u\|_2^{2(s+2)/(d+2)}\|\nabla u\|_2^{2(d-s)/(d+2)}.
\end{align*}
Therefore, since $u\in L^\infty(0,T; L^2(\R^d))\cap L^2(0,T; H^1(\R^d))$,
$$
  \int_0^T\|\K_{s/2} * f_\sigma^{(\eta)}(u)\|_{(d+2)/(d-s)}^{(d+2)/(d-s)}dt
  \leq C(\eta)\|u\|_{L^\infty(0,T; L^2(\R^d))}^{2(s+2)/(d-s)}
  \int_0^T\|\nabla u\|_2^{2}dt
	\leq C(\eta,T),
$$
meaning that $\K_{s/2} * f_\sigma^{(\eta)}(u)\in L^{(d+2)/(d-s)}(\R^d\times (0,T))$. 
Taking into account that $f_\sigma^{(\eta)}(u)\geq 0$, we deduce from 
monotone convergence that
\begin{equation}\label{ex.delta.a}
  \widetilde\K_{s/2}^{(\delta)} * f_\sigma^{(\eta)}(u)
	\to \K_{s/2} * f_\sigma^{(\eta)}(u)\quad\mbox{strongly in }
	L^{(d+2)/(d-s)}(\R^d\times (0,T)).
\end{equation}

Furthermore, arguing as before and using the estimates for 
$\DD f^{(\eta)}_\sigma$ leads to
\begin{align*}
  \|\widetilde\K_{s/2}^{(\delta)} & * (f_\sigma^{(\eta)}(u^{(\delta)}) 
	- f_\sigma^{(\eta)}(u))\|_{(d+2)/(d-s)}
  \leq \|\widetilde\K_{s/2} * |f_\sigma^{(\eta)}(u^{(\delta)}) 
	- f_\sigma^{(\eta)}(u)|\|_{(d+2)/(d-s)} \\
  &\leq C\|f_\sigma^{(\eta)}(u^{(\delta)}) - f_\sigma^{(\eta)}(u)
	\|_{(d+2)/(d+2s/d)} \\
  &\leq C(\eta)\big\||u| + |u^{(\delta)}|\big\|_{(2d+4)/(d+2s/d)}
  \| u - u^{(\delta)} \|_{(2d+4)/(d+2s/d)} \\
  &\leq C(\eta)\big(\|u\|_2^{(s+2)/(d+2)}\|\nabla u\|_2^{(d-s)/(d+2)} 
	+ \|u^{(\delta)}\|_2^{(s+2)/(d+2)}\|\nabla u^{(\delta)}\|_2^{(d-s)/(d+2)}\big) \\
  &\phantom{xx}{}\times
	\|u-u^{(\delta)}\|_2^{(s+2)/(d+2)}\|\nabla (u-u^{(\delta)})\|_2^{(d-s)/(d+2)}.
\end{align*}
Since $u^{(\delta)}$ is bounded in $L^\infty(0,T; L^2(\R^d))\cap L^2(0,T; H^1(\R^d))$, 
it follows that (up to a subsequence)
$\widetilde\K_{s/2}^{(\delta)} * (f_\sigma^{(\eta)}(u^{(\delta)}) 
- f_\sigma^{(\eta)}(u))$ converges weakly to some limit in 
$L^{(d+2)/(d-s)}(\R^d\times (0,T))$.
However, H\"older's inequality and the fact that $u^{(\delta)}\to u$ strongly 
in $L^p(0,T; L^2(\R^d))$ for every $2\leq p < \infty$,
which follows from
$$
  \int_0^T\|u^{(\delta)}-u\|_2^p dt \le \sup_{0<t<T}\|(u^{(\delta)}-u)(t)\|_2^{p-2}
	\int_0^T\|u^{(\delta)}-u\|_2^2 dt \to 0\quad\mbox{as }\delta\to 0,
$$
imply that
$$
  \widetilde\K_{s/2}^{(\delta)} * (f_\sigma^{(\eta)}(u^{(\delta)}) 
	- f_\sigma^{(\eta)}(u))\to 0\quad\mbox{strongly in }
	L^p(0,T; L^{(d+2)/(d-s)}(\R^d)), \ p<\frac{d+2}{d-s}.
$$
We conclude that 
\begin{equation}\label{ex.delta.b}
  \widetilde\K_{s/2}^{(\delta)} * (f_\sigma^{(\eta)}(u^{(\delta)}) 
	- f_\sigma^{(\eta)}(u))\rightharpoonup 0\quad\mbox{weakly in }
	L^{(d+2)/(d-s)}(\R^d\times (0,T)).
\end{equation}

We deduce from \eqref{ex.delta.a}--\eqref{ex.delta.b} that
\begin{align*}
  \widetilde\K_{s/2}^{(\delta)} &* f_\sigma^{(\eta)}(u^{(\delta)}) 
	- \K_{s/2} * f_\sigma^{(\eta)}(u) \\
  &= \big(\widetilde\K_{s/2}^{(\delta)} * f_\sigma^{(\eta)}(u) 
	- \K_{s/2} * f_\sigma^{(\eta)}(u)\big) 
	+  \widetilde\K_{s/2}^{(\delta)} * (f_\sigma^{(\eta)}(u^{(\delta)}) 
	- f_\sigma^{(\eta)}(u)) \\
  &\rightharpoonup 0 \quad\mbox{weakly in }L^{(d+2)/(d-s)}(\R^d\times (0,T)),
\end{align*}
which, together with \eqref{ex.delta.0}, implies that
$U=\nabla  \K_{s/2}\ast f_\sigma^{(\eta)}(u)$, that is,
\begin{equation}\label{ex.delta.1}
  \nabla\widetilde\K_{s/2}^{(\delta)} * f_\sigma^{(\eta)}(u^{(\delta)})
	\rightharpoonup \nabla\K_{s/2} * f_\sigma^{(\eta)}(u)
	\quad \mbox{weakly in }L^2(\R^d\times (0,T)).
\end{equation}

Let $\psi\in C^\infty_0(\R^d\times (0,T))$. Because of
$$
  \nabla\K_{s}^{(\delta)} * f_\sigma^{(\eta)}(u^{(\delta)}) 
	= \widetilde\K_{s/2}^{(\delta)}\ast\big(\nabla\widetilde\K_{s/2}^{(\delta)}
	* f_\sigma^{(\eta)}(u^{(\delta)})\big),
$$
we find that
$$
  \int_0^T\int_{\R^d}\psi\cdot\nabla\K_{s}^{(\delta)} 
	* f_\sigma^{(\eta)}(u^{(\delta)}) dx dt 
	= \int_0^T\int_{\R^d}\big(\nabla\widetilde\K_{s/2}^{(\delta)}
	* f_\sigma^{(\eta)}(u^{(\delta)})\big)
	\cdot \big(\widetilde\K_{s/2}^{(\delta)} * \psi\big) dx dt.
$$
Our goal is to show that $\widetilde\K_{s/2}^{(\delta)} * \psi\to\K_{s/2} * \psi$ 
strongly in $L^2(\R^d\times (0,T))$ as $\delta\to 0$. We can assume without loss
of generality that $\psi\geq 0$ a.e.\ in $\R^d\times (0,T)$. Indeed, for general 
functions $\psi$, we may write $\psi=\psi_+ + \psi_-$, where 
$\psi_+=\max\{0,\psi\}$ and $\psi_-=\min\{0,\psi\}$, and
we have $\widetilde\K_{s/2}^{(\delta)}*\psi = \widetilde\K_{s/2}^{(\delta)}*\psi_+ 
- \widetilde\K_{s/2}^{(\delta)}\ast(-\psi_-)$. Once again, since 
$\widetilde\K_{s/2}^{(\delta)}\nearrow \widetilde\K_{s/2}$ a.e.\ in $\R^d$, it 
is sufficient to show that $\K_{s/2}\ast\psi\in L^2(\R^d\times (0,T))$. 
The Hardy--Littlewood--Sobolev inequality (see Appendix \ref{sec.frac}) yields
$$
  \int_0^T\|\K_{s/2}\ast\psi\|_2^2 dt\leq C\int_0^T\|\psi\|_{2d/(d+2s)}^2 dt.
$$
It follows from \eqref{ex.delta.1}, the previous argument, and the fact that 
$\K_s*u = (-\Delta)^{-s}u = \K_{s/2}*\K_{s/2}*u$ that
\begin{align*}
  \int_0^T\int_{\R^d}\psi\cdot\nabla\K_{s}^{(\delta)}
	* f_\sigma^{(\eta)}(u^{(\delta)}) dx dt 
	&\to \int_0^T\int_{\R^d}\big(\nabla\K_{s/2} * f_\sigma^{(\eta)}(u)\big)
	\cdot (\K_{s/2}\ast\psi) dx dt \\
  &= \int_0^T\int_{\R^d}\psi\cdot\nabla\K_{s} * f_\sigma^{(\eta)}(u) dx dt
\end{align*}
for every $\psi\in L^2(0,T; L^{2d/(d+2s)}(\R^d))$, which means that
\begin{equation}\label{ex.delta.2}
  \nabla\K_{s}^{(\delta)} * f_\sigma^{(\eta)}(u^{(\delta)})
	\rightharpoonup\nabla\K_{s} * f_\sigma^{(\eta)}(u)
	\quad\mbox{weakly in }L^2(0,T; L^{2d/(d-2s)}(\R^d)).
\end{equation}
Since $u^{(\delta)}\to u$ strongly in $L^2(0,T; L^2(\R^d))$ and $(u^{(\delta)})$ 
is bounded in $L^\infty(0,T; L^1(\R^d))$ (via mass conservation), it also holds 
that $u^{(\delta)}\to u$ strongly in $L^2(0,T; L^{2d/(d+2s)}(\R^d))$. Therefore, 
the convergence \eqref{ex.delta.2} is sufficient to pass to the limit $\delta\to 0$ 
in \eqref{ex.2}.

{\em Step 3: Limit $\eta\to 0$ and conclusion.}
The limit $\delta\to 0$ in \eqref{ex.2} shows that the limit $u$ solves
\begin{equation}\label{ex.3}
  \begin{aligned}
  & \pa_t u - \sigma\Delta u = \diver(u\nabla \K_s * f_\sigma^{(\eta)}(u))
	\quad\mbox{in }\R^d\times (0,T),\\
  & u(\cdot,0) = \rho_\sigma^0\quad\mbox{in }\R^d.
  \end{aligned}
\end{equation}
Fatou's Lemma and the weakly lower semicontinuity of the $L^2$ norm allow us to 
infer from \eqref{ex.ei} that for $t>0$,
\begin{align}\label{ex.ei.2}
  \int_{\R^d}h^{(\eta)}(u(t))dx 
  &+ \sigma\int_0^t\int_{\R^d}\DD f_\sigma^{(\eta)}(u)u^{-1}|\nabla u|^2 dx d\tau \\   
	\nonumber
  &+ \int_0^t\int_{\R^d}|\nabla \K_{s/2} * f_\sigma^{(\eta)}(u)|^2 dx d\tau
  \leq \int_{\R^d}h^{(\eta)}(\rho_\sigma^0)dx.
\end{align}
At this point, all the bounds for $u$, derived in the previous subsections,
and the moment estimate, contained in Lemma \ref{lem.mom}, can be proved
like in Sections \ref{sec.basic}--\ref{ss.entropy}.
All these estimates are uniform in $\eta$. It is rather straightforward to 
perform the limit $\eta\to 0$ in \eqref{ex.3}--\eqref{ex.ei.2} to obtain 
a weak solution to \eqref{1.rho}.
However, the higher regularity bounds obtained in Section \ref{ss.highreg} imply 
that $u$ is actually a strong solution to \eqref{1.rho}, which in turn 
yields the uniqueness of $u$ as a weak solution to \eqref{1.rho}.
This finishes the proof of Theorem \ref{thm.rho}.

%%%%%%%%%%%%%%%%%

\subsection{Limit $\sigma\to 0$}

We prove that there exists a subsequence of $(\rho_\sigma)$ that converges strongly
in $L^1(\R^d\times(0,T))$ to a weak solution $\rho$ to \eqref{1.eq}. 

The uniform $L^\infty(\R^d\times (0,T))$ bound for $\rho_\sigma$ in 
Lemma \ref{lem.bound} implies that, up to a subsequence, 
$\rho_\sigma\rightharpoonup^*\rho$ weakly* in $L^\infty(\R^d\times(0,T))$
as $\sigma\to 0$. We deduce from the uniform $L^\infty(0,T;L^1(\R^d))$ bound 
\eqref{3.cons} and the moment bound for $\rho_\sigma$ in Lemma \ref{lem.mom}
that  $(\rho_\sigma)$ is equi-integrable. Thus, by the Dunford--Pettis theorem, 
again up to a subsequence, $\rho_\sigma\rightharpoonup\rho$ weakly in 
$L^1(\R^d\times(0,T))$.
It follows from the $L^2(0,T;H^1(\R^d))$ estimate \eqref{L2H1} that 
$\sigma\Delta\rho_\sigma\to 0$ strongly in $L^2(0,T;H^{-1}(\R^d))$.
The estimates in \eqref{u.H1s} and Lemma \ref{lem.bound} show that 
$(\pa_t\rho_\sigma)$ is bounded in $L^2(0,T;H^{-1}(\R^d))$
and consequently, up to a subsequence, $\pa_t\rho_\sigma\rightharpoonup\pa_t\rho$
weakly in $L^2(0,T;H^{-1}(\R^d))$. Therefore, the limit $\sigma\to 0$ in
\eqref{1.rho} leads to
\begin{equation}\label{lim0}
  \pa_t\rho = \diver(\overline{\rho_\sigma\na(-\Delta)^{-s}f_\sigma(\rho_\sigma)})
	\quad\mbox{in }L^2(0,T;H^{-1}(\R^d)),
\end{equation}
where the overline denotes the weak limit of the corresponding sequence.

We need to identify the weak limit on the right-hand side. 
The idea is to use the div-curl lemma \cite[Theorem 10.21]{FeNo09}.
For this, we define the vector fields with $d+1$ components
$$
  U_\sigma := \big(\rho_\sigma,-\rho_\sigma\na(-\Delta)^{-s}f_\sigma(\rho_\sigma)\big),
	\quad V_\sigma := \big(f_\sigma(\rho_\sigma),0,\ldots,0\big).
$$
Let $R>0$ be arbitrary and write $B_R$ for the ball around the origin with radius $R$.
The $L^\infty(\R^d)$ bound \eqref{Linf} for $\rho_\sigma$ and the 
$L^2(0,T;H^{1-s}(\R^d))$ bound \eqref{u.H1s} for $f_\sigma(\rho_\sigma)$ show that
$(U_\sigma)$ is bounded in $L^p(B_R\times(0,T))$ for some $p>1$, while $(V_\sigma)$
is bounded in $L^\infty(B_R\times(0,T))$. Furthermore, by \eqref{u.H1s},
\begin{align*}
  & \operatorname{div}_{(t,x)}U_\sigma = \sigma\Delta\rho_\sigma\to 0
	\quad\mbox{strongly in }L^2(0,T;H^{-1}(B_R))\hookrightarrow H^{-1}(B_R\times(0,T)), \\
	& \|\operatorname{curl}_{(t,x)}V_\sigma\|_{L^2(0,T;H^{-s}(B_R))}
	\le C\|\na f_\sigma(\rho_\sigma)\|_{L^2(0,T;H^{-s}(B_R))} \le C,
\end{align*}
where $\operatorname{curl}_{(t,x)}V_\sigma$ is the antisymmetric part of 
the Jacobian matrix of $V_\sigma$.
Hence, by the compact embedding $H^{-s}(B_R \times (0,T)) \hookrightarrow 
W^{-1,r}(B_R \times (0,T))$ (since $L^2(0,T; H^{-s}(B_R)) \subset 
H^{-s}(B_R \times (0,T))$), the sequence $(\operatorname{curl}_{(t,x)}V_\sigma)$ 
is relatively compact in $W^{-1,r}(B_R\times(0,T))$ for some $r>1$. 
Therefore, we can apply the div-curl
lemma giving $\overline{U_\sigma\cdot V_\sigma}=\overline{U_\sigma}\cdot
\overline{V_\sigma}$ or
$$
  \overline{\rho_\sigma f_\sigma(\rho_\sigma)} = \rho\overline{f_\sigma(\rho_\sigma)}
	\quad\mbox{a.e. in }B_R\times(0,T).
$$

By definition \eqref{1.fsigma} of $f_\sigma(\rho_\sigma)$, it follows for
arbitrary $\rho_\sigma\in [0,L]$ and sufficiently large $L>0$, that
\begin{align*}
  f_\sigma(\rho_\sigma) &= \int_0^{\rho_\sigma} (\Gamma_\sigma 
	* (f' \mathrm{1}_{[0,\infty)}))(u)\widetilde\Xi(\sigma u) du 
  = \int_0^{\rho_\sigma}\int_0^{\infty} \Gamma_\sigma(u-w) f'(w) dw 
	\widetilde\Xi(\sigma u) du \\
  &= \int_0^{\rho_\sigma}\int_0^{\infty} \Gamma_\sigma'(u-w) f(w) dw
	\widetilde\Xi(\sigma u) du
  = \int_0^{\infty} \bigg(\int_0^{\rho_\sigma} \Gamma_\sigma'(u-w) 
	\widetilde\Xi(\sigma u) du\bigg) f(w) dw.
\end{align*}
We use the properties that $(\rho_\sigma)$ is uniformly bounded and
$\widetilde\Xi=1$ in $[-1,1]$. Then, choosing $\sigma>0$ sufficiently small,
\begin{align*}
  f_\sigma(\rho_\sigma) &= \int_0^{\infty} \bigg(\int_0^{\rho_\sigma} 
	\Gamma_\sigma'(u-w) du\bigg) f(w) dw  \\
  &= \int_0^{\infty}\Gamma_\sigma(\rho_\sigma - w)f(w)dw 
	- \int_0^{\infty}\Gamma_\sigma(-w)f(w)dw \\
  &= \int_\R\Gamma_\sigma(\rho_\sigma - w)\tilde{f}(w)dw 
	- \int_\R\Gamma_\sigma(-w)\tilde{f}(w)dw,
\end{align*}
setting $\tilde{f}:=f\mathrm{1}_{[0,\infty)}$. 
Hence, using $f(0)=0$, we find that
$$
  f_\sigma(\rho_\sigma) - f(\rho_\sigma) 
  = \int_{\R}\Gamma_\sigma(u)(\tilde f(u + \rho_\sigma) - \tilde f(\rho_\sigma))du
	- \int_\R\Gamma_\sigma(-w)(\tilde{f}(w) - \tilde{f}(0))dw.
$$
Taking into account the fundamental theorem of calculus for the function 
$\tilde{f}\in C^0\cap W^{1,1}(\R)$, we can estimate as follows:
\begin{align*}
  |f_\sigma(\rho_\sigma) - f(\rho_\sigma)| 
	&\leq \esssup_{u\in \operatorname{supp}(\Gamma_\sigma)\backslash\{0\}}
	\bigg(\frac{|\tilde f(u + \rho_\sigma) - \tilde f(\rho_\sigma)|}{|u|} 
	+ \frac{|\tilde f(u) - \tilde f(0)|}{|u|} \bigg)\int_{\R} \Gamma_\sigma(w)|w| dw \\
  &\leq\bigg(\max_{\xi\in \operatorname{supp}(\Gamma_\sigma)\cap [0,\infty)}
	(f'(\xi + \rho_\sigma) + f'(\xi))\bigg)\int_{\R} \Gamma_\sigma(w)|w| dw.
\end{align*}
Then, since $\Gamma_\sigma(u)=\sigma^{-1}\Gamma_1(\sigma^{-1}u)$,
$\operatorname{supp}(\Gamma_\sigma)\subset B_\sigma(0)$ is compact, 
$f\in C^1([0,\infty))$, and $(\rho_\sigma)$ is uniformly bounded, we conclude that
$$
  |f_\sigma(\rho_\sigma) - f(\rho_\sigma)|\leq C\sigma .
$$ 
This means that $f_\sigma(\rho_\sigma)-f(\rho_\sigma)\to 0$ 
strongly in $L^\infty(B_R\times(0,T))$, and 
it shows that $\overline{\rho_\sigma f(\rho_\sigma)}=\rho\overline{f(\rho_\sigma)}$
a.e.\ in $B_R\times(0,T)$. As $f$ is nondecreasing, we can apply
\cite[Theorem 10.19]{FeNo09} to infer that $\overline{f(\rho_\sigma)}=f(\rho)$
a.e.\ in $B_R\times(0,T)$. Consequently, $\overline{\rho_\sigma f(\rho_\sigma)}
= \rho f(\rho)$. As $u\mapsto uf(u)$ is assumed to be strictly convex, 
we conclude from \cite[Theorem 10.20]{FeNo09} that
$(\rho_\sigma)$ converges a.e.\ in $B_R\times(0,T)$. Since $(\rho_\sigma)$
is bounded in $L^\infty(\R^d\times(0,T))$, it follows that
$\rho_\sigma\to\rho$ strongly in $L^p(B_R\times(0,T))$ for all $p<\infty$.
Using the moment estimate from Lemma \ref{lem.mom}, we infer from
\begin{align*}
  \limsup_{\sigma\to 0}\int_0^T\int_{\R^d}|\rho_\sigma-\rho|dxdt
	&= \limsup_{\sigma\to 0}\int_0^T\int_{\R^d\setminus B_R}|\rho_\sigma-\rho|dxdt \\
	&\le R^{-2d/(d-2s)}\limsup_{\sigma\to 0}\int_0^T\int_{\R^d\setminus B_R}
	\rho_\sigma(t,x)|x|^{2d/(d-2s)}dx \\
	&\le R^{-2d/(d-2s)}C\to 0\quad\mbox{as }R\to\infty
\end{align*}
that $\rho_\sigma\to\rho$ strongly in $L^p(\R^d\times(0,T))$ for all $p<\infty$.
The strong convergences of $\rho_\sigma $ and $f_\sigma(\rho_\sigma)$ in 
$L^p(\R^d\times(0,T))$ for all $p<\infty$ allow us to identify the weak limit in 
\eqref{lim0}, proving the weak formulation \eqref{1.weak}. 

Finally, we deduce from the uniform
$L^2(0,T;H^{-1}(\R^d))$ bound for $\pa_t\rho_\sigma$ and the fact that
$\rho_\sigma\to\rho$ strongly in $L^p(\R^d)$ for any $p<\infty$ that
$\rho(0)=\rho^0$ in the sense of $H^{-1}(\R^d)$. 
Properties (iv) of Theorem \ref{thm.ex} follow from the corresponding expressions
satisfied by $\rho_\sigma$ in the limit $\sigma\to 0$. 

The following lemma is needed in the proof of Theorem \ref{thm.chaos}. 

\begin{corollary}\label{coro.cont}
Under the assumptions of Theorem \ref{thm.ex}, it holds for all $\phi\in
L^\infty(\R^d)$ that, possibly for a subsequence,
$$
  \int_{\R^d}\rho_\sigma\phi dx \to \int_{\R^d}\rho\phi dx
	\quad\mbox{uniformly in }[0,T].
$$
\end{corollary}

\begin{proof}
Let $\phi\in C_0^1(\R^d)$ and $0\le t_1<t_2\le T$. The uniform 
$L^2(0,T;H^{-1}(\R^d))$ bound of $\pa_t\rho_\sigma$ implies that
\begin{align*}
  \bigg|\int_{\R^d}&\rho_\sigma(t_2)\phi dx - \int_{\R^d}\rho_\sigma(t_1)\phi dx\bigg|
	= \bigg|\int_{t_1}^{t_2}\langle\pa_t\rho_\sigma,\phi\rangle dt\bigg| \\
	&\le |t_2-t_1|^{1/2}\|\pa_t\rho_\sigma\|_{L^2(0,T;H^{-1}(\R^d))}\|\phi\|_{H^1(\R^d)}
	\le C|t_2-t_1|^{1/2}\|\phi\|_{H^1(\R^d)}.
\end{align*}
Hence, the sequence of functions $t\mapsto\int_{\R^d}\rho_\sigma(t)\phi ds$ is
bounded and equicontinuous in $[0,T]$. By the Ascoli--Arzel\'a theorem, up to
a $\phi$-depending subsequence, $\int_{\R^d}\rho_\sigma\phi dx\to\xi_\phi$
strongly in $C^0([0,T])$ as $\sigma\to 0$. Since $\rho_\sigma\rightharpoonup^*\rho$
weakly* in $L^\infty(0,T;L^\infty(\R^d))$, we can identify the limit,
$\xi_\phi=\int_{\R^d}\rho\phi dx$. 
Since $H^1(\R^d)$ is separable, a Cantor diagonal argument together
with a density argument allows us to find a subsequence (which is not relabeled)
such that for all $\phi\in H^1(\R^d)$,
\begin{equation}\label{3.conv}
  \int_{\R^d}\rho_\sigma\phi dx \to \int_{\R^d}\rho\phi dx
	\quad\mbox{strongly in }C^0([0,T]).
\end{equation}
Since $(\rho_\sigma)$ is bounded in $L^\infty(0,T;L^2(\R^d))$, another
density argument shows that this limit also holds for all $\phi\in L^2(\R^d)$.

Now, let $\phi\in L^\infty(\R^d)$. Using $\phi\mathrm{1}_{\{|x|<R\}}\in L^2(\R^d)$,
it follows from \eqref{3.conv} and the moment estimate for $\rho_\sigma$ that
\begin{align*}
  \limsup_{\sigma\to 0}&\sup_{0<t<T}\bigg|\int_{\R^d}\rho_\sigma(t)\phi dx
	- \int_{\R^d}\rho(t)\phi dx\bigg| \\
	&\le \limsup_{\sigma\to 0}\sup_{0<t<T}\bigg|\int_{\R^d}\rho_\sigma(t)\phi
	\mathrm{1}_{\{|x|>R\}}dx
	- \int_{\R^d}\rho(t)\phi\mathrm{1}_{\{|x|>R\}} dx\bigg| \\
  &\le R^{-2d/(d-2s)}\|\phi\|_\infty\limsup_{\sigma\to 0}\sup_{0<t<T}\int_{\R^d}
	(\rho_\sigma(x,t)+\rho(x,t))|x|^{2d/(d-2s)}dx \\
	&\le C(T)R^{-2d/(d-2s)}\|\phi\|_\infty\to 0\quad\mbox{as }R\to\infty.
\end{align*}
This shows that
$$
  \lim_{\sigma\to 0}\sup_{0<t<T}\bigg|\int_{\R^d}\rho_\sigma(t)\phi dx
	- \int_{\R^d}\rho(t)\phi dx\bigg| = 0,
$$
concluding the proof.
\end{proof}

%%%%%%%%%%%%%%%%%%%%%%%%%%%%%%%%%%%%%%%%%%%%%%%%%%%%%%%%%%%%%%%%%%%%%%%%%%%%%%

\section{Analysis of equation \eqref{1.rho3}}\label{sec.rho3}

This section is devoted to the analysis of equation \eqref{1.rho3},
\begin{equation}\label{4.rhobeta}
\begin{aligned}
  & \pa_t\rho_{\sigma,\beta,\zeta} - \sigma\Delta\rho_{\sigma,\beta,\zeta}
	= \diver\big(\rho_{\sigma,\beta,\zeta}\na\K_\zeta*
	f_\sigma(W_\beta*\rho_{\sigma,\beta,\zeta})\big), \quad t>0, \\
	& \rho_{\sigma,\beta,\zeta}(0)=\rho_\sigma^0\quad\mbox{in }\R^d,
\end{aligned}
\end{equation}
where $\K_\zeta=\tildeK*W_\zeta$ and $W_\beta$ is defined in \eqref{1.Wbeta}, 
as well as to an estimate for the
difference $\rho_{\sigma,\beta,\zeta}-\rho_\sigma$, which is needed in the
mean-field analysis. The existence and uniqueness of a strong solution to
\eqref{4.rhobeta} follows from standard parabolic theory, since we regularized
the singular kernel and smoothed the nonlinearity.

\begin{proposition}[Uniform estimates]\label{prop.u} 
Let Hypotheses (H1)--(H3) hold and let $T>0$, $p>d$. Set $\afrak:= \min\{1,d-2s\}$, 
let $\rho_\sigma$ be the strong solution to \eqref{1.rho}, and let 
$\rho_{\sigma,\beta,\zeta}$ be the strong
solution to \eqref{1.rho3}. Then there exist constants $C_1>0$, and $\eps_0>0$,
both depending on $\sigma$, $p$, and $T$, such that if 
$\beta+\zeta^{\afrak}<\eps_0$ then
\begin{align}
  \|\rho_{\sigma,\beta,\zeta}-\rho_\sigma\|_{L^\infty(0,T;W^{2,p}(\R^d))}
	&\le C_1(\beta+\zeta^{\afrak}), \label{4.diff} \\
	\|\rho_{\sigma,\beta,\zeta}\|_{L^\infty(0,T;W^{2,p}(\R^d))} &\le C_1. 
	\label{4.bound} 
\end{align}
Furthermore, for every $q\geq 2$, there exists $C_2=C_2(q)>0$, 
depending on $\sigma$ and $T$, such that
\begin{align}
	\|(\K_\zeta-\K)*\rho_\sigma\|_{L^\infty(0,T;L^\infty(\R^d))} 
	&\le C_2\zeta^{\afrak}, \label{4.diffK} \\
	\|\rho_{\sigma,\beta,\zeta}\|_{L^\infty(0,T;W^{2,1}(\R^d)\cap W^{3,q}}(\R^d)) 
	&\le C_2. \label{4.D2}
\end{align}
\end{proposition}

The proof is presented in the following subsections.

\subsection{Proof of \eqref{4.diff}}

We introduce the difference $u:=\rho_{\sigma,\beta,\zeta}-\rho_\sigma$, which
satisfies
\begin{align}
  \pa_t u - \sigma\Delta u &= \diver\big[(u+\rho_\sigma)\na\K_\zeta
	*f_\sigma(W_\beta*(u+\rho_\sigma)) - \rho_\sigma\na\K*f_\sigma(\rho_\sigma)\big] 
	\label{4.equ} \\
	&= D[u] + R[\rho_\sigma,u] + S[\rho_\sigma,u]\quad\mbox{in }\R^d,\ t>0, \nonumber
\end{align}
and the initial datum $u(0)=0$ in $\R^d$, where
\begin{align*}
  D[u] &= \diver\big[u\na\K*f_\sigma(W_\beta*u)\big], \\
	R[\rho_\sigma,u] &= \diver \big[u\na\K*\big(f_\sigma(W_\beta*(u+\rho_\sigma))
	- f_\sigma(W_\beta*u)\big) \\
	&\hspace*{-10mm}{}+ \rho_\sigma\na\K* \big(f_\sigma( W_\beta*(u+\rho_\sigma))
	-f_\sigma(W_\beta*\rho_\sigma)\big)
	+ \rho_\sigma\na\K*\big(f_\sigma(W_\beta*\rho_\sigma)-f_\sigma(\rho_\sigma)
	\big)\big], \\
	S[\rho_\sigma,u] &= \diver\big[(u+\rho_\sigma)\na(\K_\zeta-\K)*f_\sigma
	(W_\beta*(u+\rho_\sigma))\big].
\end{align*}

We show first an estimate for $\DD^2 u$ that depends on a lower-order estimate
for $u$.

\begin{lemma}[Conditional estimate for $u$]\label{lem.u1}
For any $p>d$, there exists a number $\Gamma_p\in(0,1)$ such that, if
$\sup_{0<t<T}\|u(t)\|_{W^{1,p}(\R^d)}\le\Gamma_p$ then
$$
  \|\DD^2 u\|_{L^p(0,T;L^p(\R^d))} \le C\big(\|u\|_{L^p(0,T;W^{1,p}(\R^d))}
	+ \beta + \zeta^{\afrak}\big),
$$
recalling that $\afrak = \min\{1,d-2s\}$, and where $C>0$ is independent of $u$, 
$\beta$, and $\zeta$, but may depend on $\sigma$.
\end{lemma}

\begin{proof}
Let $\Gamma_p\in(0,1)$ be such that $\sup_{0<t<T}\|u(t)\|_{W^{1,p}(\R^d)}\le\Gamma_p$.
We will find a constraint for $\Gamma_p$ at the end of the proof. The aim is to
derive an estimate for the right-hand side of \eqref{4.equ} in
$L^p(0,T;L^p(\R^d))$. We observe that $\|u(t)\|_1\le
\|\rho_{\sigma,\beta,\zeta}\|_1+\|\rho_\sigma\|_1\le 2\|\rho^0\|_1$
for $t\in[0,T]$. 
In the following, we denote by $C>0$ a generic constant that may depend on 
$\sigma$, without making this explicit. Furthermore, we denote by $\mu$ a generic 
exponent in $(0,1)$, whose value may vary from line to line.

{\em Step 1: Estimate of $D[u]$.} Let $1/2<s<1$. Then, by the 
Hardy--Littlewood--Sobolev-type inequality \eqref{HLS2},
\begin{align*}
  \|D[u]\|_p &\le \|\na u\cdot\na\K*f_\sigma(W_\beta*u)\|_p
	+ \|u\na\K*[f'_\sigma(W_\beta*u)W_\beta*\na u]\|_p \\
	&\le C\|\na u\|_p\|f_\sigma(W_\beta*u)\|_{d/(2s-1)}
	+ C\|u\|_{d/(2s-1)}\|f'_\sigma(W_\beta*u)\|_\infty\|\na u\|_p.
\end{align*}
We use the Young convolution inequality, the Gagliardo--Nirenberg inequality, 
the smoothness of $f_\sigma$, the property
$f_\sigma(0)=0$, and the fact $\| W_\beta \|_{L^1(\R^d)} =1$ 
to estimate the terms on the right-hand side:
\begin{align*}
  \|W_\beta*u\|_\infty &\le \|u\|_\infty 
	\le \|u\|_1^{1-\lambda}\|\na u\|_p^{\lambda}\le C\Gamma_p^{\lambda} \le C, \\
	\|f_\sigma(W_\beta*u)\|_\infty &\le \max_U|f'_\sigma|\,\|W_\beta*u\|_\infty
	\le C, \\
	\|f'_\sigma(W_\beta*u)\|_\infty &\le |f_\sigma'(0)|
	+ \max_U|f''_\sigma|\|W_\beta*u\|_\infty \le C, \\
	\|u\|_{d/(2s-1)} &\le \|u\|_1^{1-\mu}\|u\|_\infty^\mu 
	\le C\|u\|_{W^{1,p}(\R^d)}^\mu
	\le C\Gamma_p^\mu \le C,
\end{align*}
where $U:=[-\|W_\beta*u\|_\infty,\|W_\beta*u\|_\infty]$ and $\lambda>0$, $\mu>0$. 
Therefore, $\|D[u]\|_p\le C\|\na u\|_p$ and 
\begin{equation}\label{4.D}
	\|D[u]\|_{L^p(0,T;L^{p}(\R^d))} \le C\|u\|_{L^p(0,T;W^{1,p}(\R^d))}.
\end{equation}

Next, let $0<s\le 1/2$. Then we write
\begin{align*}
  D[u] &= \na u\cdot\K*[f'_\sigma(W_\beta*u)W_\beta*\na u]
	+ u\K*[f''_\sigma(W_\beta*u)|W_\beta*\na u|^2] \\
	&\phantom{xx}{}+ u\K*[f'_\sigma(W_\beta*u)W_\beta*\Delta u] 
	=: D_1 + D_2 + D_3.
\end{align*}
By the Hardy--Littlewood--Sobolev-type inequality (Lemma \ref{lem.hls}),
$$
  \|D_1\|_p \le C\|\na u\|_{d/(2s)}\|f'_\sigma(W_\beta*u)W_\beta*\na u\|_p
	\le C\|\na u\|_{d/(2s)}\|\na u\|_p.
$$
Next, we apply the Gagliardo--Nirenberg inequality with
$\lambda = (1+1/d-2s/d)/(1+2/d-1/p)$:
$$
  \|\na u\|_{d/(2s)} \le C\|u\|_1^{1-\lambda}\|\DD^2u\|_p^\lambda 
	\le C\|\DD^2u\|_p^\lambda,
$$
which is possible as long as $\lambda\ge 1/2$ or equivalently $d\ge 2s$, which is true. 
Consequently, using $\Gamma_p\le 1$,
$$
  \|D_1\|_p \le C\|\na u\|_p\|\DD^2u\|_p^\lambda
	\le C\Gamma_p^\lambda\|\na u\|_p^{1-\lambda}\|\DD^2u\|_p^\lambda
	\le C(\delta)\|\na u\|_p + \delta\|\DD^2u\|_p,
$$
where $\delta>0$ is arbitrary. It follows from the Hardy-Littlewood-Sobolev-type 
inequality and the Gagliardo--Nirenberg inequality 
$$
  \|\nabla u\|_{2p}^2 \leq C \|\DD^2 u\|_{p}^{d/p} \|\nabla u\|_{p}^{2-d/p} 
  \leq C\Gamma_p\|\DD^2 u\|_{p}^{d/p}\|\nabla u\|_{p}^{1-d/p}
$$ 
that 
$$
  \|D_2\|_p \le C\|u\|_{d/2s} \Gamma_p \| D^2 u\|_{p}^{d/p}\| \nabla u \|_{p}^{1-d/p} 
	\leq C(\delta)\|\na u\|_p + \delta\|\DD^2u\|_p. 
$$
Finally, using similar ideas, we obtain
$$
  \|D_3\|_p \le C\|u\|_{d/(2s)}\|\Delta u\|_p \le C\Gamma_p^\mu\|\DD^2u\|_p.
$$
Summarizing the estimates for $D_1$, $D_2$, and $D_3$ and integrating in time
leads to
\begin{equation}\label{4.DD}
  \|D[u]\|_{L^p(0,T;L^p(\R^d))} \le C\|u\|_{L^p(0,T;W^{1,p}(\R^d))}
	+ C\Gamma_p^\mu\|\DD^2u\|_{L^p(0,T;L^p(\R^d))}.
\end{equation}

{\em Step 2: Estimate of $R[\rho_\sigma,u]$.} We write
$R[\rho_\sigma,u]=R_1+R_2+R_3$ for the three terms in the definition of
$R[\rho_\sigma,u]$ below \eqref{4.equ}. 

{\em Step 2a: Estimate of $R_1$.}
If $s>1/2$, we can argue similarly
as in the derivation of \eqref{4.D}, which gives
$$
  \|R_1\|_{L^p(0,T;L^p(\R^d))} \le C\|u\|_{L^p(0,T;W^{1,p}(\R^d))}.
$$
If $0<s\le 1/2$, we write $R_1=R_{11}+\cdots+R_{16}$, where
\begin{align*}
  R_{11} &= \na u\cdot\K*\big[f_\sigma'(W_\beta*(u+\rho_\sigma))
	W_\beta*\na\rho_\sigma\big], \\
	R_{12} &= u\K*\big[f''_\sigma(W_\beta*(u+\rho_\sigma))(W_\beta*\na\rho_\sigma)
	\cdot(W_\beta*\na(u+\rho_\sigma))\big], \\
  R_{13} &= u\K*\big[f'_\sigma(W_\beta*(u+\rho_\sigma))
	W_\beta*\Delta\rho_\sigma\big], \\
	R_{14} &= \na u\cdot\K*\big[\big(f'_\sigma(W_\beta*(u+\rho_\sigma))
	-f'_\sigma(W_\beta*u)\big)W_\beta*\na u\big], \\
	R_{15} &= u\K*\big[\big(f''_\sigma(W_\beta*(u+\rho_\sigma))
	W_\beta*\na(u+\rho_\sigma) \\
	&\phantom{xx}{}- f''_\sigma(W_\beta*u)(W_\beta*\na u)\big)\cdot
	(W_\beta*\na u)\big], \\
	R_{16} &= u\K*\big[\big(f'_\sigma(W_\beta*(u+\rho_\sigma))
	- f'_\sigma(W_\beta*u)\big)W_\beta*\Delta u\big].
\end{align*}
All terms except the last one can be treated by the Hardy--Littlewood--Sobolev and
Ga\-gliar\-do--Nirenberg inequalities as before. For the last term, we use these
inequalities and the $L^\infty(\R^d)$ bound for $\rho_\sigma$:
\begin{align*}
  \|R_{16}\|_p &\le C\|u\|_{d/(2s)}\big\|\big(f'_\sigma(W_\beta*(u+\rho_\sigma))
	- f'_\sigma(W_\beta*u)\big)W_\beta*\Delta u\big\|_p \\
	&\le C\|u\|_{d/(2s)}\|f''_\sigma\|_\infty\|W_\beta*\rho_\sigma\|_\infty
	\|W_\beta*\Delta u\|_p
	\le C\|u\|_{d/(2s)}\|\Delta u\|_p \le C\Gamma_p^\mu\|\DD^2u\|_p.
\end{align*}
We infer that (possibly with a different $\mu>0$ than before)
$$
  \|R_1\|_{L^p(0,T;L^p(\R^d))} \le C\|u\|_{L^p(0,T;W^{1,p}(\R^d))}
	+ C\Gamma_p^\mu\|\DD^2u\|_{L^p(0,T;L^p(\R^d))}.
$$

{\em Step 2b: Estimate of $R_2$.}
Since $|f'_\sigma|$ is bounded on the interval
$[-\|u\|_\infty-\|\rho_\sigma\|_\infty,\|u\|_\infty+\|\rho_\sigma\|_\infty]$, we obtain
for $s>1/2$,
$$
  \|R_2\|_{L^p(0,T;L^p(\R^d))} \le C\|u\|_{L^p(0,T;W^{1,p}(\R^d))}.
$$
For $0<s\le 1/2$, we write $R_2=R_{21}+\cdots+R_{27}$, where
\begin{align*}
  R_{21} &= \na\rho_\sigma\cdot\K*\big[f'_\sigma(W_\beta*(u+\rho_\sigma))W_\beta
	*\na u\big], \\
	R_{22} &= \rho_\sigma\K*\big[f''_\sigma(W_\beta*(u+\rho_\sigma))
	W_\beta*\na(u+\rho_\sigma)\cdot(W_\beta*\na u)\big], \\
	R_{23} &= \rho_\sigma\K*\big[f'_\sigma(W_\beta*(u+\rho_\sigma))W_\beta*\Delta u\big],
	\\
	R_{24} &= \na\rho_\sigma\cdot\K*\big[\big(f'_\sigma(W_\beta*(u+\rho_\sigma))
	- f'_\sigma(W_\beta*\rho_\sigma)\big)W_\beta*\na\rho_\sigma\big], \\
	R_{25} &= \rho_\sigma\K*\big[f''_\sigma(W_\beta*(u+\rho_\sigma))
	(W_\beta*\na u)\cdot(W_\beta*\na\rho_\sigma)\big], \\
	R_{26} &= \rho_\sigma\K*\big[\big(f''_\sigma(W_\beta*(u+\rho_\sigma))
	- f''_\sigma(W_\beta*\rho_\sigma)\big)|W_\beta*\na\rho_\sigma|^2\big], \\
	R_{27} &= \rho_\sigma\K*\big[\big(f'_\sigma(W_\beta*(u+\rho_\sigma))
	- f'_\sigma(W_\beta*\rho_\sigma)\big)W_\beta*\Delta\rho_\sigma\big].
\end{align*}
Similar estimations as before allow us to treat all terms except the third one:
\begin{align*}
  \|R_{23}\|_p &\le \big\|\rho_\sigma\K*\big[\big(f'_\sigma(W_\beta*(u+\rho_\sigma))
	- f'_\sigma(W_\beta*\rho_\sigma)\big)W_\beta*\Delta u\big]\big\|_p \\
	&\phantom{xx}{}
	+ \|\rho_\sigma\K*[f'_\sigma(W_\beta*\rho_\sigma)W_\beta*\Delta u]\|_p 
	=: Q_{231} + Q_{232}.
\end{align*}
The first term can be estimated similarly as above by 
$Q_{231}\le C\Gamma_p^\mu\|\DD^2u\|_p$, while
\begin{align*}
  Q_{232} &\le \|\rho_\sigma\Delta\K*[f'_\sigma(W_\beta*\rho_\sigma)W_\beta*u]\|_p
	+ \|\rho_\sigma\K*[\Delta f'_\sigma(W_\beta*\rho_\sigma)W_\beta*u]\|_p \\
  &\phantom{xx}{}+ 2\|\rho_\sigma\K*[\na f'_\sigma(W_\beta*\rho_\sigma)\cdot
	(W_\beta*\na u)]\|_p.
\end{align*}
It follows from $-\Delta\K*v=(-\Delta)^{1-s}v$ and the fractional 
Gagliardo--Nirenberg inequality (Lemma \ref{lem.GN1}) that
\begin{align*}
  Q_{232} &\le C\|u\|_{W^{1,p}(\R^d)} + \|\rho_\sigma(-\Delta)^{1-s}
	[f'_\sigma(W_\beta*\rho_\sigma)W_\beta*u]\|_p \\
	&\le C\|u\|_{W^{1,p}(\R^d)} + C\|\rho_\sigma\|_\infty
	\|f'_\sigma(W_\beta*\rho_\sigma)W_\beta*u\|_p^{s}
	\|\DD^2[f'_\sigma(W_\beta*\rho_\sigma)W_\beta*u]\|_p^{1-s} \\
	&\le C\|u\|_{W^{1,p}(\R^d)} 
	+ C\|u\|_p^{s}\big(\|u\|_p^{1-s} + \|\na u\|_p^{1-s} + \|\DD^2u\|_p^{1-s}\big) \\
	&\le C\|u\|_{W^{1,p}(\R^d)} + C\Gamma_p\|\DD^2u\|_p.
\end{align*}
This shows that $\|R_{23}\|_p\le C\|u\|_{W^{1,p}(\R^d)} + C\Gamma_p^\mu\|\DD^2u\|_p$, 
and we conclude that
$$
  \|R_2\|_{L^p(0,T;L^p(\R^d))} \le C\|u\|_{L^p(0,T;W^{1,p}(\R^d))} 
	+ C\Gamma_p^\mu\|\DD^2u\|_{L^p(0,T;L^p(\R^d))}.
$$

{\em Step 2c: Estimate of $R_3$.}
We write $R_3=R_{31}+\cdots+R_{37}$, where
\begin{align*}
  R_{31} &= \na\rho_\sigma\cdot\K*\big[\big(f'_\sigma(W_\beta*\rho_\sigma)
	 - f'_\sigma(\rho_\sigma)\big)W_\beta*\na\rho_\sigma\big], \\
	R_{32} &= \rho_\sigma\K*\big[\big(f''_\sigma(W_\beta*\rho_\sigma)
	- f''_\sigma(\rho_\sigma)\big)|W_\beta*\na\rho_\sigma|^2\big], \\
	R_{33} &= \rho_\sigma\K*\big[f''_\sigma(\rho_\sigma)(W_\beta*\na\rho_\sigma
	- \na\rho_\sigma)\cdot(W_\beta*\na\rho_\sigma)\big], \\
	R_{34} &= \na\rho_\sigma\cdot\K*\big[f'_\sigma(\rho_\sigma)(W_\beta*\na\rho_\sigma
	- \na\rho_\sigma)\big], \\
	R_{35} &= \rho_\sigma\K*\big[f''_\sigma(\rho_\sigma)\na\rho_\sigma\cdot
	(W_\beta*\na\rho_\sigma	- \na\rho_\sigma)\big], \\
	R_{36} &= \rho_\sigma\K*\big[f'_\sigma(\rho_\sigma)
	(W_\beta*\Delta\rho_\sigma	- \Delta\rho_\sigma)\big] \\
	R_{37} &= \rho_\sigma\K*\big[(f'_\sigma(W_\beta \ast \rho_\sigma)
	- f'_\sigma(\rho_\sigma))W_\beta*\Delta\rho_\sigma\big] 
\end{align*}
We start with the estimate of $R_{31}$. We use the 
Hardy--Littlewood--Sobolev inequality (Lemma \ref{lem.hls}) and 
Lemma \ref{lem.wbeta} to estimate $W_\beta*\rho_\sigma-\rho_\sigma$:
\begin{align*}
  R_{31} &\le C\|\na\rho_\sigma\|_{d/s}
	\|f'_\sigma(W_\beta*\rho_\sigma)-f'_\sigma(\rho_\sigma)\|_p
	\|W_\beta*\na\rho_\sigma\|_{d/s} \\
	&\le C\|\na\rho_\sigma\|_{d/s}^2\max_{[0,2\|\rho_\sigma\|_\infty]}|f''_\sigma|\,
	\|W_\beta*\rho_\sigma-\rho_\sigma\|_p \le C(\sigma)\beta,
\end{align*}
also taking into account the $L^\infty(0,T;L^q(\R^d))$ bound for $\na\rho_\sigma$; 
see Proposition \ref{prop.regul}. With this regularity, we can estimate
all other terms except $R_{34}$ and $R_{36}$. Since they have similar structures, 
we only treat $R_{34}$. This term is delicate since the factor 
$f_\sigma'(\rho_\sigma)$ cannot be bounded in $L^q(\R^d)$
for any $q<\infty$. Therefore, one might obtain via Hardy--Littlewood--Sobolev's
inequality factors like $\|\na\rho_\sigma\|_{q_1}$ and $\|\DD^2\rho_\sigma\|_{q_2}$ 
with either $q_1<2$ or $q_2<2$. However, for such factors, an $L^\infty$ bound in 
time is currently lacking (Proposition \ref{prop.regul} provides
such a bound only for $q\ge 2$). 
Our idea is to add and subtract the term $f'_\sigma(0)$ since
$$
  |f'_\sigma(\rho_\sigma)-f'_\sigma(0)|
	\le \rho_\sigma\max_{[0,\|\rho_\sigma\|_\infty]}|f_\sigma''| \le C\rho_\sigma
$$
can be controlled. This leads to
\begin{align*}
  \|R_{34}\|_p &\le \|\na\rho_\sigma\cdot\K*[(f'_\sigma(\rho_\sigma)-f'_\sigma(0))
	(W_\beta*\na\rho_\beta-\na\rho_\beta)\|_p \\
	&\phantom{xx}{}+ \|f'_\sigma(0)\na\rho_\sigma\cdot
	\K*(W_\beta*\na\rho_\sigma-\na\rho_\sigma)\|_p \\
	&\le C\beta + |f'_\rho(0)|\|\na\rho_\sigma\cdot
	\K*(W_\beta*\na\rho_\sigma-\na\rho_\sigma)\|_p =: C\beta + Q_{341},
\end{align*}
as the first term can be estimated in a standard way.
For the estimate of $Q_{341}$, we need to distinguish two cases.

If $1/2<s\le 1$, we infer from the Hardy--Littlewoord--Sobolev-type
inequality \eqref{HLS2} that
$$
  Q_{341} \le C\|\na\rho_\sigma\|_{d/(2s-1)}\|W_\beta*\rho_\sigma-\rho_\sigma\|_p
	\le C\|\na\rho_\sigma\|_{d/(2s-1)}\|\na\rho_\sigma\|_p\beta \le C\beta.
$$
Next, let $0<s\le 1/2$. %We know that $2p/(p+1)\le d\le d/(2s)$. 
Then we apply the Hardy--Littlewoord--Sobolev-type inequality \eqref{HLS1},
the standard Gagliardo--Nirenberg inequality for some $\lambda>0$, and
Lemma \ref{lem.wbeta}:
$$
  Q_{341} \le C\|\na\rho_\sigma\|_{d/(2s)}\|W_\beta*\na\rho_\sigma-\na\rho_\sigma\|_p
	\le C\|\rho_\sigma\|_1^{1-\lambda}\|\DD^2\rho_\sigma\|_p^\lambda
	(\beta\|\DD^2\rho_\sigma\|_p) \le C\beta.
$$
We conclude that $\|R_{34}\|_p\le C\beta$ and eventually 
$$
  \|R_3\|_{L^p(0,T;L^p(\R^d))}\le C\beta.
$$
Summarizing the estimates for $R_1$, $R_2$, and $R_3$ finishes this step:
\begin{equation}\label{4.R}
  \|R[\rho_\sigma,u]\|_{L^p(0,T;L^p(\R^d))}\le C\|u\|_{L^p(0,T;W^{1,p}(\R^d))}
	+ C\beta + C\Gamma_p^\mu\|\DD^2u\|_{L^p(0,T;W^{1,p}(\R^d))}.
\end{equation}

{\em Step 3: Estimate of $S[\rho_\sigma,u]$.}
We formulate this term as $S[\rho_\sigma,u]=S_1+\cdots+S_4$, where
\begin{align*}
  S_1 &= \diver\big[u\na(\K_\zeta-\K)*\big(f_\sigma(W_\beta*(u+\rho_\sigma))
	- f_\sigma(W_\beta*\rho_\sigma)\big)\big], \\
	S_2 &= \diver\big(u\na(\K_\zeta-\K)*f_\sigma(W_\beta*\rho_\sigma)\big), \\
	S_3 &= \diver\big[\rho_\sigma\na(\K_\zeta-\K)*\big(f_\sigma(W_\beta*(u+\rho_\sigma))
	- f_\sigma(W_\beta*\rho_\sigma)\big)\big], \\
	S_4 &= \diver\big(\rho_\sigma\na(\K_\zeta-\K)*f_\sigma(W_\beta*\rho_\sigma)\big).
\end{align*}
The terms $S_1$, $S_2$, and $S_3$ can be treated as the terms in $R[\rho_\sigma,u]$,
since they have the same structure and the techniques used to estimate integrals
involving $\K$ can be applied to those involving $\K_\zeta$. This leads to
(for some $\mu>0$)
\begin{equation}\label{4.S123}
  \|S_1+S_2+S_3\|_{L^p(0,T;L^p(\R^d))} \le C\|u\|_{L^p(0,T;W^{1,p}(\R^d))}
	+ C\Gamma_p^\mu\|\DD^2u\|_{L^p(0,T;L^p(\R^d))}.
\end{equation}

It remains to estimate $S_4$. We write $S_4=S_{41}+S_{42}+S_{43}$, where
\begin{align*}
  S_{41} &= \na\rho_\sigma\cdot(\K_\zeta-\K)*\big[f'_\sigma(W_\beta*\rho_\sigma)
	W_\beta*\na\rho_\sigma\big], \\
	S_{42} &= \rho_\sigma(\K_\zeta-\K)*\big[f''_\sigma(W_\beta*\rho_\sigma)
	|W_\beta*\na\rho_\sigma|^2\big], \\
	S_{43} &= \rho_\sigma(\K_\zeta-\K)*\big[f'_\sigma(W_\beta*\rho_\sigma)
	W_\beta*\Delta\rho_\sigma\big].
\end{align*}
Observe that, because of the definition of 
$\K_\zeta = \tildeK \ast W_\zeta$ with $\tildeK = \K \omega_\zeta$ 
(defined in \eqref{1.TildeK}), we have $(\K_\zeta-\K)*v=\K*(W_\zeta*v-v) 
- (\K(1-\omega_\zeta))\ast W_\zeta\ast v$ for every function $v$ 
for which the convolution is defined, and therefore, by the 
Hardy--Littlewood--Sobolev-type inequality \eqref{HLS1}, 
Young's convolution inequality, and Lemma \ref{lem.wbeta},
\begin{align*}
  \|\rho_\sigma(\K_\zeta-\K)*v\|_p 
	&\le C\|\rho_\sigma\|_{d/(2s)}\|W_\zeta*v-v\|_p
  + C\|\rho_\sigma\|_{p}\|(\K (1-\omega_\zeta))\ast v\|_\infty \\
  &\le C\|\rho_\sigma\|_{d/(2s)}\|\na v\|_p\zeta 
	+ C\|\rho_\sigma\|_p\|\K\mathrm{1}_{\R^d\backslash B(0,\zeta^{-1})}
	\|_{\infty}\|v\|_{1} \\
  &\le C\|\rho_\sigma\|_{d/(2s)}\|\na v\|_p\zeta 
  + C\zeta^{d-2s}\|\rho_\sigma\|_p\|v\|_{1},
\end{align*}
Given the regularity properties of $\rho_\sigma$ (see Lemma \ref{lem.regulrho}) 
and the assumptions on $f_\sigma$, it follows that
\begin{equation}\label{4.S4}
  \|S_4\|_{L^p(0,T;L^p(\R^d))} \le C\zeta^{\min\{1,d-2s\}}.
\end{equation}

We conclude from \eqref{4.S123} and \eqref{4.S4} that
\begin{equation}\label{4.S}
  \|S[\rho_\sigma,u]\|_{L^p(0,T;L^p(\R^d))}
	\le C\|u\|_{L^p(0,T;W^{1,p}(\R^d))} + C\zeta^{\afrak}
	+ C\Gamma_p^\mu\|\DD^2u\|_{L^p(0,T;L^p(\R^d))},
\end{equation}
where $\afrak := \min \{1,d-2s\}$.

{\em Step 4: End of the proof.} Summarizing \eqref{4.DD}, \eqref{4.R}, and \eqref{4.S},
we infer that the right-hand side of \eqref{4.equ} can be bounded (for some $\mu>0$) by
\begin{align*}
  \|D[u]&+R[\rho_\sigma,u]+S[\rho_\sigma,u]\|_{L^p(0,T;L^p(\R^d))} \\
	&\le C\|u\|_{L^p(0,T;W^{1,p}(\R^d))} + C(\beta+\zeta^{\afrak}) 
	+ C\Gamma_p^\mu\|\DD^2u\|_{L^p(0,T;L^p(\R^d))}.
\end{align*}
By parabolic regularity \eqref{a.D2u},
$$
  \|\DD^2u\|_{L^p(0,T;L^p(\R^d))} \le C\|u\|_{L^p(0,T;W^{1,p}(\R^d))}
	+ C(\beta+\zeta^{\afrak}) + C\Gamma_p^\mu\|\DD^2u\|_{L^p(0,T;L^p(\R^d))}.
$$
Choosing $\Gamma_p>0$ sufficiently small finishes the proof.
\end{proof}

It remains to estimate the ${L^p(0,T;W^{1,p}(\R^d))}$ norm of $u$. This is done
in the following lemma.

\begin{lemma}[Unconditional estimate for $u$]\label{lem.u2}
For any $p>d$, there exist constants $C>0$, and $\eps_0>0$, both
depending on $\sigma$, $p$, and $T$, such that for 
$\beta+\zeta^{\afrak}<\eps_0$,
$$
  \|u\|_{L^\infty(0,T;W^{1,p}(\R^d))} \le C(\beta+\zeta^{\afrak}).
$$ 
recalling that $\afrak:= \min\{1,d-2s\}$.
\end{lemma}

\begin{proof}
The idea is to test \eqref{4.equ} with $p|u|^{p-2}u-p\diver(|\na u|^{p-2}\na u)$.
Integration by parts and some elementary computations lead to
\begin{align*}
  \int_{\R^d} & p\diver(|\nabla u|^{p-2}\nabla u)\Delta u dx
  = -p\sum_{i,j}\int_{\R^d} |\nabla u|^{p-2}\pa_i u \pa_i\pa_{jj}^2 u dx \\
	&= p\sum_{i,j}\int_{\R^d}\pa_j(|\nabla u|^{p-2}\pa_i u)\pa_{ij}^2 u dx \\
  &= p\int_{\R^d} |\nabla u|^{p-2} |D^2 u|^2 dx 
  + \frac{p}{2}\sum_{j}\int_{\R^d}\pa_j(|\nabla u|^{p-2})\pa_{j} (|\nabla u|^2)dx \\
  &= p\int_{\R^d} |\nabla u|^{p-2} |D^2 u|^2 dx 
  + \sum_{j}\int_{\R^d}\frac{4}{p}(p-2)\big(\pa_j (|\nabla u|^{p/2})\big)^2 dx.
\end{align*} 
Consequently, we have
\begin{align}
  p\|u&(t)\|_{W^{1,p}(\R^d)}^p 
	+ \sigma p(p-1)\int_0^t\int_{\R^d}|u|^{p-2}|\na u|^2 dxds \label{4.Qu} \\
	&\phantom{xx}{}{}+ \sigma\int_0^t\int_{\R^d}\big(p|\na u|^{p-2}|D^2 u|^2
	+ 4(p-2)p^{-1}\big|\na(|\na u|^{p/2})\big|^2\big)dxds \nonumber \\
	&= p\int_0^t\int_{\R^d}\big(|u|^{p-2}u - \diver(|\na u|^{p-2}\na u)\big)
	\big(D[u]+R[\rho_\sigma,u]+S[\rho_\sigma,u]\big)dxds \nonumber \\
	&=: Q[u]. \nonumber
\end{align}

We infer from Lemmas \ref{lem.embedd} and \ref{lem.regul} that 
$u\in C^0([0,T];W^{1,p}(\R^d))$. Therefore,
since $u(0)=0$, it holds that $\|u(t)\|_{W^{1,p}(\R^d)}\le\Gamma_p$ for all
$t\in[0,T^*]$ and $T^* := \sup\{ t_0\in (0,T): \|u(t)\|_{W^{1,p}(\R^d)}\leq\Gamma_p$
for $0\leq t\leq t_0\}.$
Let $t\in[0,T^*]$. We have shown in the proof of the previous lemma that
$$
  \|D[u]+R[\rho_\sigma,u]+S[\rho_\sigma,u]\|_{L^p(0,t;L^p(\R^d))}
	\le C\|u\|_{L^p(0,t;W^{1,p}(\R^d))} + C(\beta+\zeta^{\afrak}).
$$
Hence, we can estimate the right-hand side $Q[u]$ of \eqref{4.Qu} as follows:
\begin{align*}
  Q[u] &\le C\int_0^t\int_{\R^d}\big(|u|^{p-1} + |\na u|^{p-2}|\DD^2u|\big)
	\big|D[u]+R[\rho_\sigma,u]+S[\rho_\sigma,u]\big|dx \\
	&\le C\big(\|u\|_{L^p(0,t;L^p(\R^d))}^{p-1} + \|\na u\|_{L^p(0,t;L^p(\R^d))}^{p/2-1}
	\||\na u|^{p/2-1}|\DD^2u|\|_{L^2(0,t;L^2(\R^d))}\big) \\
	&\phantom{xx}{}\times\big(\|u\|_{L^p(0,t;W^{1,p}(\R^d))} + \beta 
	+ \zeta^{\afrak}\big) \\
	&\le C(\delta,p,t)\big(\|u\|_{L^p(0,t;W^{1,p}(\R^d))}^p 
	+ (\beta+\zeta^{\afrak})^p\big)
	+ \delta\||\na u|^{p/2-1}|\DD^2u|\|_{L^2(0,t;L^2(\R^d))}^2,
\end{align*}
where $\delta>0$. Choosing $\delta$ sufficiently small, the last term is absorbed
by the corresponding expression on the left-hand side of \eqref{4.Qu}, and
we infer from \eqref{4.Qu} that for $0\le t\le T^*$,
$$
  \|u(t)\|_{W^{1,p}(\R^d)}^p \le C(p,t)\int_0^t\|u\|_{W^{1,p}(\R^d)}^p ds
	+ C(p,t)(\beta+\zeta^{\afrak})^p.
$$
We assume without loss of generality that $C(p,t)$ is nondecreasing in $t$.
Then Gronwall's lemma implies that for $0\le t\le T^*$,
$$
  \|u(t)\|_{W^{1,p}(\R^d)}^p \le C(p,T)(\beta+\zeta^{\afrak})^p
	\int_0^t e^{C(p,T)(t-s)}ds
	\le (\beta+\zeta^{\afrak}) e^{C(p,T)t}.
$$
Choosing $\eps_0=\frac12\Gamma_p\exp(-C(p,T)T/p) < 1$, we find that
$\|u(t)\|_{W^{1,p}(\R^d)}\le\Gamma_p/2$ for $\beta+\zeta^{\afrak}<\eps_0$ 
and $0\le t\le T^*$.
By definition of $T^*$, it follows that $T^*=T$. In particular,
$\|u(t)\|_{W^{1,p}(\R^d)}\le C(\beta+\zeta^{\afrak})$ for $0<t<T$, 
which finishes the proof.
\end{proof}

%%%%%%%%%%%%%%%%%

\subsection{Proof of \eqref{4.diff} and \eqref{4.bound}.}

Combining Lemmas \ref{lem.u1} and \ref{lem.u2} leads to
\begin{equation}\label{estimate_u}
  \|u\|_{L^p(0,T;W^{2,p}(\R^d))}\le C(\sigma,p,T)(\beta+\zeta^{\afrak}), 
	\quad\text{where }\afrak=\min\{1,d-2s\},
\end{equation}
as long as $\beta+\zeta^{\afrak}<\eps_0$ and $p>d$. Next, we differentiate 
\eqref{4.equ} with respect to $x_i$ (writing $\pa_i$ for $\pa/\pa x_i$):
$$
  \pa_t(\pa_i u) - \sigma\Delta(\pa_i u)
	= \pa_i\big(D[u]+R[\rho_\sigma,u]+S[\rho_\sigma,u]\big), \quad
	\pa_i u(0)=0\quad\mbox{in }\R^d.
$$
Taking into account estimate \eqref{estimate_u} and arguing as in the
proof of Lemma \ref{lem.u1}, we can show that for $\delta>0$,
$$
  \|\pa_i(D[u]+R[\rho_\sigma,u]+S[\rho_\sigma,u])\|_{L^p(0,T;L^p(\R^d))}
	\le C(p,\sigma,\delta)(\beta+\zeta^{\afrak}) 
	+ \delta\|\DD^3 u\|_{L^p(0,T;L^p(\R^d))}.
$$ 
We infer from parabolic regularity (Lemma \ref{lem.regul}) for sufficiently 
small $\delta>0$ that
$$
  \|\pa_t\DD u\|_{L^p(0,T;L^p(\R^d))} + \|\DD^3u\|_{L^p(0,T;L^p(\R^d))}
	\le C(p,\sigma)(\beta+\zeta^{\afrak}).
$$
Then Lemma \ref{lem.embedd}, applied to $\DD u$, leads to
\eqref{4.diff}, which with Proposition \ref{prop.regul} implies \eqref{4.bound}. 

%%%%%%%%%%%%%%%%%

\subsection{Proof of \eqref{4.diffK}.}\label{sec.diffK}

Let $x\in\R^d$. We use the definitions of $\K_\zeta$ and $W_\zeta$ to find that
\begin{align*}
  |(\K_\zeta&-\K)*\rho_\sigma(x)| 
	= \bigg|\int_{\R^d} W_\zeta(x-y)\big((\K*\rho_\sigma)(x)
	- ((\K \omega_\zeta)*\rho_\sigma)(y)\big)dy \bigg| \\
	&\le \int_{\R^d}W_\zeta(x-y)|x-y|
	\frac{|(\K*\rho_\sigma)(x)-(\K*\rho_\sigma)(y)|}{|x-y|}dy 
	+ \|(\K(1-\omega_\zeta))*\rho_\sigma\|_\infty \\
	&\le \|\na\K*\rho_\sigma\|_\infty\int_{\R^d}W_\zeta(z)|z|dz
	+\|\K\mathrm{1}_{\R^d\backslash B(0,\zeta^{-1})}\|_{\infty}\|\rho_\sigma\|_{1} \\
	&\le \zeta\|\na\K*\rho_\sigma\|_\infty\int_{\R^d}W_1(y)|y|dy
	+\zeta^{d-2s}\|\rho_\sigma\|_{1}.
\end{align*}
Let $\phi\in C_0^\infty(\R^d)$ be such that $\operatorname{supp}(\phi)\subset B_2$
and $\phi=1$ in $B_1$. Then (since we can assume without loss of generality that 
$\zeta <1$), by arguing like in the derivation of \eqref{4.S4}, we obtain
$$
  |(\K_\zeta-\K)*\rho_\sigma(x)| 
	\le C\zeta^{\min\{1,d-2s\}}\big(\|\na(\K\phi)*\rho_\sigma\|_\infty
	+ \|\na(\K(1-\phi))*\rho_\sigma\|_\infty + \|\rho_\sigma\|_{1}\big),
$$
A computation shows that for $p>\max\{d/(2s),2\}$,
\begin{align*}
  \|\na(\K\phi)*\rho_\sigma\|_\infty 
	&= \|(\K\phi)*\na\rho_\sigma\|_\infty\le \|\K\phi\|_{p/(p-1)}\|\na\rho_\sigma\|_p
	\le C\|\na\rho_\sigma\|_p, \\
	\|\na(\K(1-\phi))*\rho_\sigma\|_\infty
	&\le \|\na(\K(1-\phi))\|_\infty\|\rho_\sigma\|_1 
	\le C\|\rho_\sigma\|_1,
\end{align*}
where we note that $\K\mathrm{1}_{B_2} \in L^{p/(p-1)}$ if $p > d/(2s)$.
Then, in view of the regularity of $\rho_\sigma$ in Lemma \ref{lem.regulrho}, 
we find that
$$
  \|(\K_\zeta-\K)*\rho_\sigma\|_{L^\infty(0,T;L^\infty(\R^d))}
	\le C\zeta^{\afrak}.
$$

%%%%%%%%%%%%%%%%%

\subsection{Proof of \eqref{4.D2}.}

The $L^\infty(0,T;W^{2,1}(\R^d)\cap W^{3,q}(\R^d))$ bound for 
$\rho_{\sigma,\beta,\zeta}$ is shown in a similar way as
the corresponding bound for $\rho_\sigma$ in Lemma \ref{lem.regulrho}.

%%%%%%%%%%%%%%%%%%%%%%%%%%%%%%%%%%%%%%%%%%%%%%%%%%%%%%%%%%%%%%%%%%%%%%%%%%

\section{Mean-field analysis}\label{sec.mean}

This section is devoted to the proof of Theorems \ref{thm.error} and \ref{thm.chaos}.

\subsection{Existence of density functions}\label{sec.dens}

We claim that the solution $\widehat{X}^N$ to \eqref{1.hatX} is absolutely
continuous with respect to the Lebesgue measure, which implies that this
process possesses a probability density $\widehat{u}\in L^\infty(0,T;L^1(\R^d))$.
The claim follows from \cite[Theorem 2.3.1]{Nua06} if the coefficients
of the stochastic differential equation \eqref{1.hatX}, satisfied by $\widehat{X}^N$, 
are globally Lipschitz continuous and of at most linear growth. 
The latter condition follows from
\begin{align*}
  |\na&\K*f_\sigma(\rho_\sigma(x,t))| 
	\le\|\K*\na f_\sigma(\rho_\sigma)\|_{L^\infty(0,T;L^\infty(\R^d))} \\
	&\le C\|\K*\na f_\sigma(\rho_\sigma)\|_{L^\infty(0,T;W^{1,p}(\R^d))}
	\le C\|\na f_\sigma(\rho_\sigma)\|_{L^\infty(0,T;W^{1,r}(\R^d))}
	\le C(\sigma),
\end{align*}
where $p>d$ and $r=dp/(d+2s)$ according to the Hardy--Littlewood--Sobolev inequality,
and we used the regularity bounds for $\rho_\sigma$ from Lemma \ref{lem.regul}.
The global Lipschitz continuity is a consequence of the mean-value theorem,  
the Hardy--Littlewood--Sobolev inequality, and the $W^{2,\infty}(\R^d)$ 
regularity of $\rho_\sigma$ from Lemma \ref{lem.regulrho}:
\begin{align*}
  \sup_{0<t<T}&\big|\na\K*f_\sigma(\rho_\sigma(x,t))
	-\na\K*f_\sigma(\rho_\sigma(y,t))\big|
  \le \sup_{0<t<T}\|\DD^2\K*f_\sigma(\rho_\sigma(\cdot,t))\|_\infty|x-y| \\
	&= \sup_{0<t<T}\big\|\K*\big(f''_\sigma(\rho_\sigma)\na\rho_\sigma
	\otimes\na\rho_\sigma + f'_\sigma(\rho_\sigma)\DD^2\rho_\sigma\big)(\cdot,t)
	\big\|_\infty|x-y| \le C(\sigma)|x-y|.
\end{align*}
Similar arguments show that $\bar{X}_i^N(t)$ has a density function
$\bar{u}\in L^\infty(0,T;L^1(\R^d))$.

Next, we show that $\widehat{u}$ and $\bar{u}$ can be identified with the
weak solutions $\rho_\sigma$ and $\rho_{\sigma,\beta,\zeta}$, respectively,
using It\^{o}'s lemma. Indeed, let 
$\phi\in C_0^\infty(\R^d\times[0,T])$. We infer from It\^{o}'s formula that
\begin{align*}
  \phi(&\widehat{X}_i^N(t),t) = \phi(\widehat{X}_i^N(0),0)
	+ \int_0^t\pa_s\phi(\widehat{X}_i^N(s),s)ds 
	+ \sigma\int_0^t\Delta\phi(\widehat{X}_i^N(s),s)ds \\
	&{}- \int_0^t\na\K*f_\sigma(\rho_\sigma(\widehat{X}_i^N(s),s))\cdot\na\phi
	(\widehat{X}_i^N(s),s) ds
	+ \sqrt{2\sigma}\int_0^t\na\phi(\widehat{X}_i^N(s),s)\cdot dB_i^N(s).
\end{align*}
Taking the expectation, the It\^{o} integral vanishes, and we end up with
\begin{align}
  \int_{\R^d}&\phi(x,t)\widehat{u}(x,t)dx
	= \int_{\R^d}\phi(x,0)\rho^0_\sigma(x)dx 
	+ \int_0^t\int_{\R^d}\pa_s\phi(x,s)\widehat{u}(x,s)dxds \nonumber  \\
	&{}+ \sigma\int_0^t\int_{\R^d}\Delta\phi(x,s)\widehat{u}(x,s)dxds
	- \int_0^t\int_{\R^d}\na\K*f_\sigma(\rho_\sigma(x,s))\cdot\na\phi(x,s)
	\widehat{u}(x,s)dxds. \label{5.weak}
\end{align}
Hence, $\widehat{u}$ is a very weak solution in the space $L^\infty(0,T;L^1(\R^d))$
to the linear equation
\begin{equation}\label{5.lin}
  \pa_t \widehat{u} = \sigma\Delta\widehat{u} 
	+ \diver(\widehat{u}\na\K*f_\sigma(\rho_\sigma)), \quad 
	\widehat{u}(0)=\rho_\sigma^0\quad\mbox{in }\R^d,
\end{equation}
where $\rho_\sigma$ is the unique solution to \eqref{1.rho}.

It can be shown that \eqref{5.lin} is uniquely solvable
in the class of functions in $L^\infty(0,T;L^1(\R^d))$. 
This implies that $\widehat{u}=\rho_\sigma$ in $\R^d\times(0,T)$
(and similarly $\bar{u}=\rho_{\sigma,\beta,\zeta}$).
The proof is technical but standard; see, e.g., \cite[Theorem 7]{CDHJ20} 
for a sketch of a proof.

Another approach is as follows. Because of the linearity of \eqref{5.weak},
it is sufficient to prove that $\widehat{u}\equiv 0$ in $\R^d\times(0,T)$
if $\rho_\sigma^0=0$. First, we verify that $v:=\na\K*f_\sigma(\rho_\sigma)\in
L^\infty(0,T;W^{1,\infty}(\R^d))$ and $\widehat{u}\in L^p(0,T;L^p(\R^d))$
for $p<d/(d-1)$. Then, by density, \eqref{5.weak} holds for all
$\phi\in W^{1,q}(0,T;L^q(\R^d))\cap L^q(0,T;W^{2,q}(\R^d))$ with $q>d$
and $\phi(T)=0$. Choosing $\psi$ to be the unique strong solution to the dual problem
$$
  \pa_t\psi+\sigma\Delta\psi = v\cdot\na \psi + g, \quad \psi(T)=0
	\quad\mbox{in }\R^d
$$
in the very weak formulation of \eqref{5.weak}, we find that 
$\int_0^T\int_{\R^d}\widehat{u}gdxdt=0$ for all $g\in C_0^\infty(\R^d\times(0,T))$,
which implies that $\widehat{u}=0$.

%%%%%%%%%%%%%%%%%%%%%%

\subsection{Estimate of $X_i^N-\overline{X}_i^N$}

\begin{lemma}\label{lem.diff1}
Let $X_i^N$ and $\bar{X}_i^N$ be the solutions to \eqref{1.X} and
\eqref{1.barX}, respectively, and let $\delta\in(0,1/4)$. 
Under the assumptions of Theorem \ref{thm.error} on $\beta$ and $\zeta$, it holds that
$$
  \E\bigg(\sup_{0<s<T}\max_{i=1,\ldots,N}|(X_i^N-\bar{X}_i^N)(s)|\bigg) 
	\le CN^{-1/4+\delta}.
$$
\end{lemma}

\begin{proof}
To simplify the presentation, we set
$$
  \Psi(x,t) := f_\sigma\bigg(\frac{1}{N}\sum_{j=1,\,j\neq i}^{N}
	W_\beta(X^{N}_j(t)-x)\bigg),
	\ \bar\Psi(x,t) := f_\sigma\bigg(\frac{1}{N}\sum_{j=1,\,j\neq i}^{N}
	W_\beta(\bar{X}^{N}_j(t)-x)\bigg),
$$
and we write $\rho:=\rho_{\sigma,\beta,\zeta}$.
Taking the difference of equations \eqref{1.X} and \eqref{1.barX} in the
integral formulation leads to
\begin{align}\label{5.XbarX}
  \sup_{0<s<t}&|(X_i^N-\bar{X}_i^N)(s)|
	\le \int_0^t\big|\na\K_\zeta*\big(\Psi(X_i^N(s),s) 
	- f_\sigma(W_\beta*\rho(\bar{X}_i^N(s),s))\big)\big|ds \\
	&\le \int_0^t\big|\na\K_\zeta*\big(\Psi(X_i^N(s),s) - \bar\Psi(\bar{X}_i^N(s),s)
	\big)\big|ds \nonumber \\
	&\phantom{xx}{}+ \int_0^t\big|\na\K_\zeta*\big(\bar\Psi(\bar{X}_i^N(s),s)
	- f_\sigma(W_\beta*\rho(\bar{X}_i^N(s),s))\big)\big|ds
	=: I_1 + I_2. \nonumber
\end{align}

{\em Step 1: Estimate of $I_1$.} To estimate $I_1$, we formulate
$I_1=I_{11}+I_{12}+I_{13}$, where
\begin{align*}
  I_{11} &= \int_0^t\big|\na\K_\zeta*\big(\Psi(X_i^N(s),s) - \Psi(\bar{X}_i^N(s),s)
	\big)\big|ds, \\
	I_{12} &= \int_0^t\big|\na\K_\zeta*\big(\Psi(\bar{X}_i^N(s),s) 
	- \bar\Psi(X_i^N(s),s)\big)\big|ds, \\
	I_{13} &= \int_0^t\big|\na\K_\zeta*\big(\bar\Psi(X_i^N(s),s) 
	- \bar\Psi(\bar{X}_i^N(s),s)\big)\big|ds.
\end{align*}

We start with the first integral:
$$
  I_{11} \le \int_0^t\|\DD^2\K_\zeta*\Psi(\cdot,s)\|_\infty
	\sup_{0<r<s}\max_{i=1,\ldots,N}|(X_i^N-\bar{X}_i^N)(r)|ds.
$$
We claim that
\begin{equation}\label{5.DkK}
  \|\DD^k\K_\zeta*\Psi(\cdot,s)\|_\infty 
	\le C(\sigma)\beta^{-(k+1)(d+k)-1}, \quad k\in\N.
\end{equation}
For the proof, we introduce
$$
  \Phi(x,y) := f_\sigma\bigg(\frac{1}{N}\sum_{j=1}^{N-1}W_\beta(y_j-x)\bigg)
	\quad\mbox{for }x\in\R^d,\ y=(y_1,\ldots,y_{N-1})\in\R^{(N-1)d}.
$$
Then, by definition of $\K_\zeta$, 
$$
  \|\DD^k\K_\zeta*\Psi(\cdot,t)\|_\infty
	\le \sup_{y\in\R^{N-1}}\|W_\zeta*\K\omega_\zeta*\DD^k\Phi(\cdot,y)\|_\infty.
$$
We estimate the right-hand side:
\begin{align*}
  \|W_\zeta*(\K\omega_\zeta*\DD^k\Phi(\cdot,y))\|_\infty
	&\le \|W_\zeta\|_1\|\K\omega_\zeta*\DD^k\Phi(\cdot,y)\|_\infty 
	\le C\|\K\omega_\zeta*\DD^k\Phi(\cdot,y)\|_{W^{1,p}(\R^d)} \\
	&\le C\|\K*|\DD^k\Phi(\cdot,y)|\|_p + C\|\K*|\DD^{k+1}\Phi(\cdot,y)|\|_p \\
	&\le C\|\DD^k\Phi(\cdot,y)\|_r + C\|\DD^{k+1}\Phi(\cdot,y)\|_r,
\end{align*}
where we used the Hardy--Littlewood--Sobolev inequality for
$r=dp/(d+2ps)$ in the last step. It follows from the Fa\`a di Bruno formula,
after an elementary computation, that the last term is estimated according to
\begin{align*}
  \|\DD^{k+1}&\Phi(\cdot,y)\|_r^r = \int_{\R^d}\bigg|\DD^{k+1}\bigg(f_\sigma
	\bigg(\frac{1}{N}\sum_{j=1}^{N-1}W_\beta(y_j-x)\bigg)\bigg)\bigg|^r dx \\
	&\le C(k,N)\max_{\ell=1,\ldots,k+1}\|f_\sigma^{(\ell)}\|_\infty^r
	\|\DD^k W_\beta\|_\infty^{kr}\max_{0\le j\le k}\int_{\R^d}
	|\DD^{j+1}W_\beta(x)|^r dx \\
	&\le C(k,N)\max_{\ell=1,\ldots,k+1}\|f_\sigma^{(\ell)}\|_\infty^r
	\beta^{-(d+k)kr}\beta^{-(d+k+1)r+d} 
	\le C(k,N,\sigma)\beta^{-(d+k)(k+1)r-r},
\end{align*}
since $\|\DD^k W_\beta\|_\infty\le C\beta^{-(d+k)}$ and 
$\|\DD^{j+1}W_\beta\|_r\le C\beta^{-(d+j+1)+d/r}$. This verifies \eqref{5.DkK}.
We infer from \eqref{5.DkK} with $k=2$ that
$$
  I_{11} \le C\beta^{-3d-7}\int_0^t
	\sup_{0<r<s}\max_{i=1,\ldots,N}|(X_i^N-\bar{X}_i^N)(r)|ds.
$$

The term $I_{13}$ is estimated in a similar way, with $\Psi$ replaced by
$\bar\Psi$:
$$
  I_{13} \le C\beta^{-3d-7}\int_{0}^t
	\sup_{0<r<s}\max_{i=1,\ldots,N}|(X_i^N-\bar{X}_i^N)(r)|ds.
$$
The estimate of the remaining term $I_{12}$ is more involved. Since $W_\beta$
is assumed to be symmetric, we find that
\begin{align*}
  I_{12} &= \bigg|\int_0^t\int_{\R^d}\K_\zeta(y)\na\bigg\{f_\sigma\bigg(\frac{1}{N}
	\sum_{j=1,\,j\neq i}^N W_\beta(X_j^N(s)-\bar{X}_i^N(s)+y)\bigg) \\
	&\phantom{xx}{}- f_\sigma\bigg(\frac{1}{N}
	\sum_{j=1,\,j\neq i}^N W_\beta(\bar{X}_j^N(s)-X_i^N(s)+y)\bigg)\bigg\}dyds\bigg| \\
	&\le C\int_0^t\int_{\R^d}\K_\zeta(y)\bigg|f'_\sigma\bigg(\frac{1}{N}\sum_{j\neq i}
	W_\beta(X_j^N(s)-\bar{X}_i^N(s)+y)\bigg) \\
	&\phantom{xx}{}\times\frac{1}{N}\sum_{j\neq i}
	\na\big(W_\beta(X_j^N(s)-\bar{X}_i^N(s)+y)
	- W_\beta(\bar{X}_j^N(s)-X_i^N(s)+y)\big) \\
	&\phantom{xx}+ \bigg\{f'_\sigma\bigg(\frac{1}{N}\sum_{j\neq i}W_\beta
	(X_j^N(s)-\bar{X}_i^N(s)+y)\bigg) - f'_\sigma\bigg(\frac{1}{N}\sum_{j\neq i}W_\beta
	(\bar{X}_j^N(s)-X_i^N(s)+y)\bigg)\bigg\} \\
	&\phantom{xx}{}\times\frac{1}{N}\sum_{j\neq i}\na W_\beta(\bar{X}_j^N(s)-X_i^N(s)+y)
	\bigg|dyds \\
	&\le C\|f'_\sigma\|_\infty\int_0^t\sup_{0<s<t}\max_{i=1,\ldots,N}
	|(X_i^N-\bar{X}_i^N)(s)|
	\frac{1}{N}\sum_{j\neq i}\int_{\R^d}\K_\zeta(y)|
	\DD^2W_\beta(y+\xi_{ij}(s))|dyds \\
	&\phantom{xx}{}+ C\|f''_\sigma\|_\infty\int_0^t\sup_{0<s<t}\max_{i=1,\ldots,N}
	|(X_i^N-\bar{X}_i^N)(s)|\,\|\K_\zeta*\na W_\beta\|_\infty ds,
\end{align*}
where $\xi_{ij}(s)$ is a random value.
We write $\K^1=\K|_{B_1}$, $\K^2=\K|_{\R^d\setminus B_1}$ and note that 
$\tildeK \leq \K$ for all $\zeta >0$. Then
\begin{align*}
  \int_{\R^d}&\K_\zeta(y)|\DD^2W_\beta(y+\xi_{ij}(s))|dy
	\le \int_{B_{1+\zeta}}(\K^1*W_\zeta)(y)|\DD^2W_\beta(y+\xi_{ij}(s))|dy \\
	&\phantom{xx}{}
	+ \int_{\R^d\backslash B_{1-\zeta}}(\K^2*W_\zeta)(y)|\DD^2W_\beta(y+\xi_{ij}(s))|dy \\
	&\le \|\K^1*W_\zeta\|_{L^{\theta/(\theta-1)}(B_{1+\zeta})}
	\|\DD^2W_\beta(\cdot+\xi_{ij}(s))\|_{L^\theta(B_{1+\zeta})} \\
	&\phantom{xx}{}+ \|\K^2*W_\zeta\|_\infty\|\DD^2W_\beta(\cdot+\xi_{ij}(s))
	\|_{L^1(\R^d\backslash B_{1-\zeta})} \\
	&\le \|\K^1\|_{L^{\theta/(\theta-1)}(B_{1})}
	\|\DD^2W_\beta(\cdot+\xi_{ij}(s))\|_{L^\theta(B_{1+\zeta})}
	+ \|\K^2\|_\infty\|\DD^2W_\beta(\cdot+\xi_{ij}(s))\|_{L^1(\R^d)} \\
	&\le C\big(\|\DD^2W_\beta\|_\infty + \|\DD^2W_\beta\|_1\big) \le C\beta^{-d-2}.
\end{align*}
Observe that we did not use the compact support for $\tildeK$ 
(which depends on $\zeta$), because a negative exponent of $\zeta$ at this point 
would lead to a logarithmic connection between $\zeta$ and $N$ in the end, which
we wish to avoid.

Furthermore, by the convolution, Sobolev, and Hardy--Littlewood--Sobolev
inequalities as well as the fact that $|\tildeK* \nabla W_\beta|
= |(\K w_\zeta)*W_\zeta*\nabla W_\beta| \leq \K*|W_\zeta|*|\nabla W_\beta|$, 
\begin{align*}
  \|\K_\zeta*\na W_\beta\|_\infty &= \|W_\zeta*\tildeK*\na W_\beta\|_\infty
	\le \|\tildeK*\na W_\beta\|_\infty \le \|\tildeK*\na W_\beta\|_\infty \\
	&\le C\|\tildeK*\na W_\beta\|_{W^{1,p}(\R^d)} 
	\leq C(\|\K * |\nabla W_\beta |\|_p^p + \|\K * |D^2 W_\beta| \|_p^p)^{1/p} \\
	&\le C\|\na W_\beta\|_{W^{1,r}(\R^d)} \le C\beta^{-d-2+d/r},
\end{align*}
where we recall that $r>d/(2s)$ and we choose $p>d$ satisfying
$1/p=2s/d-1/r$. The previous two estimates lead to
$$
  I_{12} \le C(\sigma)\beta^{-d-2}\int_0^t
	\sup_{0<r<s}\max_{i=1,\ldots,N}|(X_i^N-\bar{X}_i^N)(r)|ds.
$$
We summarize:
\begin{equation}\label{5.I1}
  I_1 \le C(\sigma)\beta^{-3d-7}\int_0^t
	\sup_{0<r<s}\max_{i=1,\ldots,N}|(X_i^N-\bar{X}_i^N)(r)|ds.
\end{equation}

{\em Step 2: Estimate of $I_2$.} We take the expectation of $I_2$ and use the
mean-value theorem:
\begin{align}
  \E(I_2) &= \int_0^t\E\bigg|\int_{\R^d}\na\K_\zeta(y)\bigg\{f_\sigma\bigg(
	\frac{1}{N}\sum_{j\neq i}W_\beta(\bar{X}_j^N(s)-\bar{X}_i^N(s)+y)\bigg) 
	\label{5.I2aux} \\
	&\phantom{xx}{}- f_\sigma\big(W_\beta*\rho(\bar{X}_i^N(s)-y,s)\big)\bigg\}dy\bigg|ds 
	\nonumber \\
	&\le N^{-1}\|f'_\sigma\|_\infty\|\tildeK*\na W_\zeta\|_1\int_0^t\sup_{y\in\R^d}\E
	\bigg(\sum_{j\neq i}|b_{ij}(y,s)|\bigg)ds, \nonumber
\end{align}
where
$$
  b_{ij}(y,s) = W_\beta(\bar{X}_j^N(s)-\bar{X}_i^N(s)+y)
	- \frac{N}{N-1}W_\beta*\rho(\bar{X}_i^N(s)-y,s).
$$
We deduce from $\|\nabla W_\zeta\|_{L^1(\R^d)} \leq C \zeta^{-1}$ that 
$$
  \|\tildeK*\na W_\zeta\|_1\le C \zeta^{-1}\|\tildeK\|_1\le C\zeta^{-2s-1},
$$ 
due to the compact support of $\tildeK(x) = |x|^{2s-d}\omega_\zeta(x) 
\leq C|x|^{2s-d} \mathrm{1}_{|x|\leq 2\zeta^{-1}}$ and 
$$
  \int_{\{|x| < 2/\zeta\}} |x|^{2s-d} dx 
	= \int_{\{|y| < 2\}} \zeta^{-d}|y/\zeta|^{2s-d} dy = C\zeta^{-2s}.
$$

We claim that $\E(\sum_{j\neq i}|b_{ij}(y,s)|)\le C(\sigma)\beta^{-d/2}N^{1/2}$ 
for all $y \in \R^d$.
To show the claim, we compute the expectation $\E[(\sum_{j\neq i}b_{ij}(y,s))^2]$. 
We estimate first the terms with $k\neq j$ (omitting the argument $(y,s)$ to 
simplify the notation). Then an elementary but tedious computation leads to
\begin{align*}
  \E(b_{ji}b_{ki}) &= \int_{\R^d}\int_{\R^d}\int_{\R^d}\bigg(W_\beta(x_j-x_i+y)
	- \frac{N}{N-1}W_\beta*\rho(x_i-y)\bigg) \\
	&\phantom{xx}{}\times\bigg(W_\beta(x_k-x_i+y)	- \frac{N}{N-1}
	W_\beta*\rho(x_i-y)\bigg)\rho(x_i)\rho(x_j)\rho(x_k)dx_idx_jdx_k \\
	&= \int_{\R^d}\bigg(W_\beta*\rho(x_i-y) - \frac{N}{N-1}W_\beta*\rho(x_i-y)\bigg)^2
	\rho(x_i)dx_i \\
	&\le N^{-2}\|\rho\|_{L^\infty(0,T;L^\infty(\R^d))}
	\|W_\beta*\rho\|^2_{L^\infty(0,T;L^2(\R^d))} \\
	&\le C(\sigma)N^{-2}\|W_\beta\|_1^2 \le C(\sigma)N^{-2}.
\end{align*}
The diagonal terms contribute in the following way:
\begin{align*}
  \E(b_{ji}^2) &= \int_{\R^d}\int_{\R^d}\bigg(W_\beta(x_j-x_i+y)
	- \frac{N}{N-1}W_\beta*\rho(x_i-y)\bigg)^2\rho(x_i)\rho(x_j)dx_idx_j \\
	&= \int_{\R^d}\bigg((W_\beta^2*\rho)(x_i-y) 
	- \frac{2N}{N-1}(W_\beta*\rho)(x_i-y)^2 \\
	&\phantom{xx}{}
	+ \frac{N^2}{(N-1)^2}(W_\beta*\rho)(x_i-y)^2\bigg)\rho(x_i)dx_i \\
	&\le C(\sigma)\big(\|W_\beta^2*\rho\|_{L^\infty(0,T;L^1(\R^d))}
	+ \|W_\beta*\rho\|_{L^\infty(0,T;L^{2}(\R^d))}^2\big)
	\le C(\sigma)\beta^{-d},
\end{align*}
since $\|W_\beta^2*\rho\|_2\le\|W_\beta^2\|_1\|\rho\|_2 \le C\|W_\beta\|_2^2
\le\beta^{-d}C$. This shows that 
$$
  \E\bigg(\sum_{j\neq i}|b_{ji}(y, s)|\bigg)
	\le \bigg(\E\bigg[\sum_{j\neq i}b_{ji}(y, s)\bigg]^2\bigg)^{1/2}
	\le C(\sigma)\beta^{-d/2}N^{1/2}.
$$
We infer that \eqref{5.I2aux} becomes
\begin{equation}\label{5.I2}
  I_2 \le C(\sigma)\zeta^{-2s-1}\beta^{-d/2}N^{-1/2}.
\end{equation}

{\em Step 3: End of the proof.}
We insert \eqref{5.I1} and \eqref{5.I2} into \eqref{5.XbarX} to infer that
\begin{align*}
  E_1(t) &:= \E\bigg(\sup_{0<s<t}\max_{i=1,\ldots,N}|(X_i^N-\bar{X}_i^N)(s)|\bigg) \\
	&\le C(\sigma)\beta^{-3d-7}\int_0^t E_1(s)ds 
	+ C(\sigma)\zeta^{-2s-1}\beta^{-d/2}N^{-1/2}.
\end{align*}
By Gronwall's lemma,
$$
  E_1(t) \le C(\sigma)\zeta^{-2s-1}\beta^{-d/2}N^{-1/2}
  \exp\big(C(\sigma)\beta^{-3d-7}T\big),
	\quad 0\le t\le T.
$$
We choose $\eps=\widetilde{\delta}/(C(\sigma)T)$ for some arbitrary 
$\widetilde{\delta}\in(0,1/4)$.
Then, since by assumption, $\beta^{-d/2}\le\beta^{-3d-7}\le \eps\log N$ and
$\zeta^{-2s-1}\le C_1 N^{1/4}$, we find that
$$
  E_1(t) \le C(\sigma)C_1\eps\log(N)N^{-1/4}\exp\big(C(\sigma)T\eps\log N\big)
	= \frac{C_1\widetilde{\delta}}{T}\log(N)N^{-1/4+\widetilde{\delta}},
$$
proving the result.
\end{proof}

%%%%%%%%%%%%%

\subsection{Estimate of $\bar{X}_i^N-\widehat{X}_i^N$}

\begin{lemma}\label{lem.diff2}
Let $\bar{X}_i^N$ and $\widehat{X}_i^N$ be the solutions to \eqref{1.barX} and
\eqref{1.hatX}, respectively. Then there exists a constant $C>0$, depending on
$\sigma$, such that
$$
  \E\bigg(\sup_{0<t<T}\max_{i=1,\ldots,N}|(\bar{X}_i^N-\widehat{X}_i^N)(t)|\bigg)
	\le C(\beta+\zeta^{\afrak}),
$$
where $\afrak:= \min\{1,d-2s\}$.
\end{lemma}

\begin{proof}
We compute the difference
\begin{align*}
  |(\bar{X}_i^N-\widehat{X}_i^N)(t)|
	&= \bigg|\int_0^t\big(\na\K_\zeta*f_\sigma(W_\beta*\rho(\bar{X}_i^N(s),s))
	- \na\K*f_\sigma(\rho_\sigma(\widehat{X}_i^N(s),s))\big)ds\bigg| \\
	&\le J_1+J_2+J_3,
\end{align*}
where $\rho:=\rho_{\sigma,\beta,\zeta}$, the convolution is taken with respect
to $x_i$, and
\begin{align*}
  J_1 &= \bigg|\int_0^t\na\K_\zeta*\big(f_\sigma(W_\beta*\rho(\bar{X}_i^N(s),s))
	- f_\sigma(W_\beta*\rho(\widehat{X}_i^N(s),s))\big)ds\bigg|, \\
	J_2 &= \bigg|\int_0^t\na\K_\zeta*\big(f_\sigma(W_\beta*\rho(\widehat{X}_i^N(s),s))
	- f_\sigma(\rho_\sigma(\widehat{X}_i^N(s),s))\big)ds\bigg|, \\
	J_3 &= \bigg|\int_0^t\na(\K_\zeta-\K)
	*f_\sigma(\rho_\sigma(\widehat{X}_i^N(s),s))ds\bigg|.
\end{align*}

{\em Step 1: Estimate of $J_1$.} We write $\na\K_\zeta*f_\sigma(\cdots)
=\K_\zeta*\na f_\sigma$ and add and subtract the expression
$f'_\sigma(W_\beta*\rho(\bar{X}_i^N-y))\na W_\beta*\rho(\widehat{X}_i^N-y)$:
\begin{align*}
  J_1 &= \int_0^t\int_{\R^d}\K_\zeta(y)\Big(f'_\sigma(W_\beta*\rho(\bar{X}_i^N(s)-y))
	\na W_\beta*\big[\rho(\bar{X}_i^N(s)-y)-\rho(\widehat{X}_i^N(s)-y)\big] \\
	&\phantom{xx}{}- \big[f'_\sigma(W_\beta*\rho(\widehat{X}_i^N(s)-y)) 
	- f'_\sigma(W_\beta*\rho(\bar{X}_i^N(s)-y))\big]
	\na W_\beta*\rho(\widehat{X}_i^N(s)-y)\Big)dyds \\
  &\le \|f'_\sigma\|_\infty\int_0^t\int_{\R^d}\Big|\K_\zeta(y)\na W_\beta*
	\big(\rho(\bar{X}_i^N(s)-y)-\rho(\widehat{X}_i^N(s)-y)\big)\Big|dyds \\
	&\phantom{xx}{}+ \|f''_\sigma\|_\infty
	\|\na W_\beta*\rho\|_{L^\infty(0,T;L^\infty(\R^d))} \\
	&\phantom{xx}{}\times
	\int_0^t\int_{\R^d}\big|\K_\zeta(y)W_\beta*\big(\rho(\widehat{X}_i^N(s)-y)
	- \rho(\bar{X}_i^N(s)-y)\big)\big|dyds.
\end{align*}
By the mean-value theorem and using $\|W_\beta\|_{1}=1$, 
we obtain for some random variable $\xi_{ij}(s)$,
\begin{align}
  J_1 &\le \|f_\sigma\|_{W^{2,\infty}(\R)}\|\na\rho\|_{L^\infty(0,T;L^\infty(\R^d))}
	\int_0^t\sup_{0<r<s}\sup_{i=1,\ldots,N}|(\bar{X}_i^N-\widehat{X}_i^N)(r)| 
	\label{5.J1aux} \\
	&\phantom{xx}{}\times
	\int_{\R^d}\sum_{k=1}^2\big|\K_\zeta(y)\DD^{k} W_\beta*\rho(y+\xi_{ij}(s),s)\big|dyds.
	\nonumber
\end{align}

We need to estimate the last integral. For this, we write for $k=1,2$
\begin{align*}
  & \int_{\R^d}\big|\K_\zeta(y)\DD^{k} W_\beta*\rho(y+\xi_{ij}(s),s)\big|dy
	\le K^{k}_1 + K^{k}_2, \quad\mbox{where} \\
	& K^{k}_1 := \int_{B_{1+ \zeta}}\big|\K^1* W_\zeta(y)\DD^{k} 
	W_\beta*\rho(y+\xi_{ij}(s),s)\big|dy, \\
	& K^{k}_2 := \int_{\R^d\setminus B_{1-\zeta}}
	\big|\K^2* W_\zeta(y)\DD^{k} W_\beta*\rho(y+\xi_{ij}(s),s)\big|dy,
\end{align*}
where $\K^1=\K|_{B_1}$ and $\K^2=\K|_{\R^d\setminus B_1}$.  
Note that $\tildeK \leq \K$. A similar argument as for the estimate of
$I_{12}$ in the proof of Lemma \ref{lem.diff1} shows that for 
$\theta > \max\{d/(2s),d \}$,
\begin{align*}
  K_1^k+K_2^k &\leq C\big(\|\DD^kW_\beta* \rho\|_{L^\infty(0,T; L^\theta(\R^d))} 
	+ \|\DD^kW_\beta* \rho\|_{L^\infty(0,T; L^1(\R^d))}\big) \\
	&\le C\big(\|\DD^k\rho\|_{L^\infty(0,T; L^\theta(\R^d))} 
	+ \|\DD^k\rho\|_{L^\infty(0,T; L^1(\R^d))}\big) \leq C(\sigma),
\end{align*}
where we used Proposition \ref{prop.u} (\eqref{4.bound} and \eqref{4.D2})  
with $p = \theta $ in the last inequality.
We conclude from \eqref{5.J1aux} that
\begin{equation}\label{5.J1}
  J_1 \le C(\sigma)\int_0^t\sup_{0<r<s}\max_{i=1,\ldots,N}
	|(\bar{X}_i^N-\widehat{X}_i^N)(r)|ds.
\end{equation}

{\em Step 2: Estimate of $J_2$.} We treat the two cases $s<1/2$ and $s \geq 1/2$ 
separately. Let first $s \geq 1/2$. Then
\begin{align*}
	J_2 &= \bigg|\int_0^t\na\tildeK*W_\zeta*\big(f_\sigma(W_\beta
	*\rho(\widehat{X}_i^N(s),s))
	- f_\sigma(\rho_\sigma(\widehat{X}_i^N(s),s))\big)ds\bigg| \\
	&\leq T \|\nabla \tildeK*(f_\sigma(W_\beta*\rho) - f_\sigma(\rho_\sigma))
	\|_{L^\infty(0,T; L^\infty(\R^d))}.
\end{align*}
Recalling the definition of $\tildeK = \K \omega_\zeta$ in \eqref{1.TildeK}
 and writing $\nabla \tildeK\ast u = \nabla\K\ast u -
[(1-\omega_\zeta)\nabla\K]\ast u + [\K\nabla\omega_\zeta]\ast u$ for 
$u= f_\sigma(W_\beta*\rho) - f_\sigma(\rho_\sigma)$, we find that
\begin{align}\label{5.J2_case1}
  J_2 &\leq C(T) \big(\|\nabla\K * u\|_{L^\infty([0,T]; L^\infty(\R^d))} 
	+ \|[(1-\omega_\zeta)\nabla\K] * u \|_{L^\infty([0,T]; L^\infty(\R^d))} \\
  &\phantom{xx}{}+ \|[\K\nabla\omega_\zeta] * u 
	\|_{L^\infty(0,T; L^\infty(\R^d))}\big). \nonumber
\end{align}
We estimate the right-hand side term by term. Because of
\begin{align*}
  \nabla\K * v = 
  \begin{cases}
  \nabla (-\Delta)^{-1/2}v & \mbox{for }s=1/2\\
  (\nabla\K)\ast v & \mbox{for }s>1/2,
  \end{cases}
\end{align*}
we use Sobolev's embedding $W^{1,p}(\R^d)\hookrightarrow L^\infty(\R^d)$ 
for any $p>d$ and then the boundedness of the Riesz operator 
$\nabla (-\Delta)^{-1/2}:L^p (\R^d)\to L^p(\R^d)$ 
\cite[Chapter IV, \S 3.1]{Ste70} in case $s=1/2$ 
or the Hardy--Littlewood--Sobolev inequality for 
$\widetilde{\alpha}= \alpha-1/2 >0$ (see Lemma \ref{lem.hls}) in case $s>1/2$ 
to control the first norm in \eqref{5.J2_case1} by
\begin{align*}
  \|\nabla\K * u&\|_{L^\infty(0,T; L^\infty(\R^d))} 
	\leq C\bigg(\|\nabla\K * u\|_{L^\infty(0,T; L^p(\R^d))}  
	+ \sum_{j=1}^d\|\nabla\K * D^ju\|_{L^\infty(0,T; L^p(\R^d))}\bigg) \\
  &\leq C \|u\|_{L^\infty(0,T; W^{1,r}(\R^d))} 
	= C \Vert f_\sigma(W_\beta*\rho) 
	- f_\sigma(\rho_\sigma) \Vert_{L^\infty(0,T; W^{1,r}(\R^d))},
\end{align*}
where $r=p$ in case $s=1/2$ and $r= pd/(d+2s-1)$ in case $s>1/2$. 
Choosing $p>d+(2s-1)$ guarantees that $r>d$ always holds.

For the second norm in \eqref{5.J2_case1}, H\" older's inequality yields for 
$q>d$ and $1/q + 1/q'=1$, for every $t>0$,
\begin{align*}
  \|[(1-\omega_\zeta)\nabla\K] * u(t)\|_{L^\infty(\R^d)} 
	&\leq \|1- \omega_\zeta\|_{L^\infty(\R^d)}
	\|\nabla \K\|_{L^{q'}(\{|x| > 2\zeta^{-1}\})}\|u(t)\|_{L^q(\R^d)} \\ 
	&\leq \|\nabla \K\|_{L^{q'}(\{ |x| > 2\zeta^{-1}\})}\|u(t)\|_{L^q(\R^d)},
\end{align*}
which can be bounded by $C \zeta^{1-2s+d/q}\|u(t)\|_{L^{q}(\R^d)}$, since
\begin{align*}
  \|\nabla \K\|_{L^{q'}(\{|x| > 2\zeta^{-1}\})}^{q'} 
	&\leq C \int_{\{|x| > 2\zeta^{-1}\}}|x|^{(2s-d-1)q'}dx 
	= C \zeta^{-d}\int_{\{|y| > 2\}} |y/\zeta|^{(2s-d-1)q'} dy \\
  &\leq C \zeta^{-d+(1+d-2s)q'}.
\end{align*}
By similar arguments and the fact that $\|\nabla\omega_\zeta\|_{L^\infty} 
\leq C \zeta$, we find that
$$
  \|\K \nabla \omega_\zeta\|_{L^{q'}(\{ |x| < 2\zeta^{-1}\})} 
	\leq C \zeta^{1+d-2s-d/q'},
$$
and hence, using $q' = q/(q-1)$, we conclude for the second and third term
in \eqref{5.J2_case1} that
$$
  \|[(1-\omega_\zeta)\nabla\K] * u(t)\|_{L^\infty(\R^d)} 
	+ \|[\K\nabla\omega_\zeta] * u(t)\|_{L^\infty(\R^d)}
	\leq C \zeta^{1-2s+d/q}\|u(t)\|_{L^q(\R^d)}.
$$
The choice $d<q \leq d/(2s-1)$ guarantees on the one hand that $q>d$ and 
on the other hand that the exponent $1-2s+d/q$ is strictly positive 
(which allows us to use the property $\zeta^{1-2s+d/q} <1$).

Using these estimates in \eqref{5.J2_case1}, we arrive (for $s \geq 1/2$) at
\begin{equation*}%\label{5.J_2_case1_2}
  J_2 \leq C(T)\big(\|f_\sigma(W_\beta*\rho) - f_\sigma(\rho_\sigma) 
	\|_{L^\infty(0,T; W^{1,r}(\R^d))} + \|f_\sigma(W_\beta*\rho) 
	- f_\sigma(\rho_\sigma)\|_{L^\infty(0,T; L^q(\R^d))}\big),
\end{equation*}
where we recall that $r,q>d$.
These norms can be estimated by $\|f_\sigma(W_\beta*\rho(t)) 
- f_\sigma(\rho_\sigma(t))\|_{L^q(\R^d)} 
\leq \|f'_\sigma\|_{\infty}\|W_\beta * \rho(t) - \rho_\sigma(t)\|_{L^q(\R^d)}$ and 
\begin{align*}
  \|\nabla (f_\sigma(W_\beta * \rho) - f_\sigma(\rho_\sigma))(t)\|_{L^r(\R^d)} 
	&\leq \|f'_\sigma\|_{\infty}\|(W_\beta * \nabla\rho - \nabla\rho_\sigma)(t)
	\|_{L^r(\R^d)} \\
	&\phantom{xx}{}+ \|f_\sigma''\|_{\infty}\|(W_\beta \ast \rho - \rho_\sigma)(t)
	\|_{L^r(\R^d)}\|\nabla\rho_\sigma(t)\|_{L^\infty(\R^d)}.
\end{align*}
The $L^\infty(\R^d\times(0,T))$ bound for $\nabla \rho_\sigma$ from 
Lemma \ref{lem.regulrho} and the definition of $f_\sigma$ finally show for 
$s \geq 1/2$ and $r,q>d$ that
\begin{equation}\label{5.J_2_case1_3}
  J_2 \leq C(\sigma, T)\big(\|W_\beta*\rho - \rho_\sigma
	\|_{L^\infty(0,T; W^{1,r}(\R^d))} 
	+ \|W_\beta*\rho - \rho_\sigma\|_{L^\infty(0,T; L^q(\R^d))}\big).
\end{equation}

Now, let $s<1/2$. In this case, we cannot estimate $\na\K$ and put the gradient
to the second factor of the convolution.
Adding and subtracting an appropriate expression as in Step 1, 
using the embedding $W^{1,p}(\R^d)\hookrightarrow L^\infty(\R^d)$ for $p>d$,
the estimate $\K_\zeta \leq \K$, and the Hardy--Littlewood--Sobolev inequality, 
we find that
\begin{align*}
  J_2 &= \bigg|\int_0^t\int_{\R^d}\K_\zeta(y)
	\Big(\big(f'_\sigma(W_\beta*\rho(\widehat{X}_i^N(s)-y))
	- f'_\sigma(\rho_\sigma(\widehat{X}_i^N(s)-y))\big)
	\na W_\beta*\rho(\widehat{X}_i^N(s)-y) \\
	&\phantom{xx}{}
	- f'_\sigma(\rho_\sigma(\widehat{X}_i^N(s)-y))
	\big(\na\rho_\sigma(\widehat{X}_i^N(s)-y)
	- \na W_\beta*\rho(\widehat{X}_i^N(s)-y)\big)\Big)dyds\bigg| \\
	&\le \|f''_\sigma\|_\infty\|W_\beta*\na\rho\|_\infty\int_0^t\int_{\R^d}
	\K_\zeta(y)\big|\rho_\sigma(\widehat{X}_i^N(s)-y)
	- W_\beta*\rho(\widehat{X}_i^N(s)-y)\big|dyds \\
	&\phantom{xx}{}+ \|f'_\sigma\|_\infty\int_0^t\int_{\R^d}\K_\zeta(y)
	\big|\na\rho_\sigma(\widehat{X}_i^N(s)-y) 
	- W_\beta*\na\rho(\widehat{X}_i^N(s)-y)\big|dyds \\
  &\le \max\{\|\nabla \rho\|_{L^\infty(0,T; L^\infty(\R^d))},1\}
	\|f_\sigma'\|_{W^{1,\infty}}T\big(\|\K * |(W_\beta * \rho - \rho_\sigma)|
	\|_{L^\infty(0,T; L^\infty(\R^d))}  \\
	&\phantom{xx}{}+ \|\K * |(W_\beta * \nabla\rho - \nabla\rho_\sigma)| 
	\|_{L^\infty(0,T; L^\infty(\R^d))}\big) \\
	&\le C(\sigma,T)\big(\|\na\rho\|_{L^\infty(0,T;L^\infty(\R^d))}+1\big)
	\sum_{|\alpha|\le 2}\|W_\beta*\DD^\alpha\rho 
	- \DD^\alpha\rho_\sigma\|_{L^\infty(0,T;L^{r}(\R^d))},
\end{align*}
where $r>d$ is such that $1/r=2s/d+1/p$ (this is needed for the
Hardy--Littlewood--Sobolev inequality) and $p>d$ (because of Sobolev's emebdding). 
Note that $r>d$ can be only guaranteed if $s<1/2$. Together with the fact that 
$\|\na\rho\|_{L^\infty(0,T;L^\infty(\R^d))} \leq C(\sigma)$ 
(choose $q>d$ in \eqref{4.D2} and use Sobolev's embedding), 
this shows that for $s< 1/2$,
\begin{equation}\label{5.J2_case2}
  J_2 \leq C(\sigma,T) \sum_{|\alpha|\le 2}\|W_\beta*\DD^\alpha\rho 
	- \DD^\alpha\rho_\sigma\|_{L^\infty(0,T;L^{r}(\R^d))}.
\end{equation}

It follows from estimate \eqref{4.diff} and Lemma \ref{lem.wbeta} 
in Appendix \ref{sec.aux} for $p>d$ that
$$
  \|(W_\beta* D^\alpha \rho - D^\alpha \rho_\sigma)(t)\|_{L^p(\R^d)} 
  \leq C\big(\|D^\alpha \nabla \rho\|_{L^p(\R^d)} \beta + \beta 
	+ \zeta^{\afrak}\big) \leq C(\sigma,T)(\beta + \zeta^{\afrak}),
$$
where we used the $L^\infty(0,T; W^{3,p}(\R^d))$ estimate for 
$\rho = \rho_{\sigma,\beta,\zeta}$ in \eqref{4.D2}. Then
we deduce from estimates \eqref{5.J_2_case1_3} and \eqref{5.J2_case2} 
that for all $0< s < 1$,
\begin{equation*}%\label{5.J2}
  J_2 \le C(\sigma,T)(\beta+\zeta^{\afrak}),
\end{equation*}
where we recall that $\afrak= \min\{1,d-2s\}$.

{\em Step 3: Estimate of $J_3$ and end of the proof.}
Arguing similarly as in Section \ref{sec.diffK}, we have
$$
  \|(\K_\zeta-\K)*\na\rho_\sigma\|_{L^\infty(0,T;L^\infty(\R^d))}
	\le C\zeta^{\afrak}\big(\|\DD^2\rho_\sigma\|_{L^\infty(0,T;L^{p}(\R^d))}
	+\|\na\rho_\sigma\|_{L^\infty(0,T;L^1(\R^d))}\big).
$$
This implies that
\begin{equation}\label{5.J3}
  J_3 \le \|f'_\sigma\|_\infty\|(\K_\zeta-\K)*\na\rho_\sigma
	\|_{L^\infty(0,T;L^\infty(\R^d))} \le C(\sigma)\zeta^{\afrak}.
\end{equation}
Taking the expectation, we infer from \eqref{5.J1}--\eqref{5.J3} that
\begin{align*}
  E_2(t) := \E\bigg(\sup_{0<s<t}\max_{i=1,\ldots,N}
	|(\bar{X}_i^N-\widehat{X}_i^N)(s)|\bigg) 
  \le C(\sigma)(\beta+\zeta^{\afrak}) + C(\sigma)\int_0^t E_2(s)ds,
\end{align*}
An application of Gronwall's lemma gives the result.
\end{proof}

%%%%%%%%%%%%%%%%%%%%%%%%%

\subsection{Proof of Theorems \ref{thm.error} and \ref{thm.chaos}}

Lemmas \ref{lem.diff1} and \ref{lem.diff2} show that
$$
  \E\bigg(\sup_{0<s<T}\max_{i=1,\ldots,N}|(X_i^N-\widehat{X}_i^N)(s)|\bigg)
	\le C(N^{-1/4+\delta} + \beta + \zeta^{\min\{1,d-2s\}}),
$$
and this expression converges to zero as $N\to\infty$ and $(\beta,\zeta)\to 0$
under the conditions stated in Theorem \ref{thm.error}. 
This result implies the convergence in probability of
the $k$-tuple $(X_1^N,\ldots,X_k^N)$ to $(\widehat{X}_1^N,\ldots,\widehat{X}_k^N)$. 
Since convergence in probability implies convergence in distribution, we obtain
$$
  \lim_{N\to\infty,\,(\beta,\zeta)\to 0}\mathrm{P}^k_{N,\beta,\sigma}(t)
	= \mathrm{P}^{\otimes k}_\sigma(t)\quad\mbox{locally uniform in time},
$$
where $\mathrm{P}^k_{N,\beta,\sigma}(t)$ and $\mathrm{P}^{\otimes k}_\sigma(t)$ denote 
the joint distributions of $(X_1^N,\ldots,X_k^N)(t)$ and 
$(\widehat{X}_1^N,\ldots,$ $\widehat{X}_k^N)(t)$, respectively. 
By Section \ref{sec.dens}, $\mathrm{P}_\sigma(t)$ is absolutely continuous
with the density function $\rho_\sigma(t)$. 
Using the test function $\phi=\mathrm{1}_{(-\infty, x]^d}$ 
in Corollary \ref{coro.cont}, we have, up to a subsequence,
$$
  \mathrm{P}_\sigma(t,(-\infty,x]^d) = \int_{(-\infty,x]^d}\rho_\sigma(y,t)dy
	\to  \int_{(-\infty,x]^d}\rho(y,t)dy =: \mathrm{P}(t,(-\infty,x]^d)
$$
locally uniformly for $t>0$. Since the convergence also holds for the 
initial condition, the result is shown. 

%%%%%%%%%%%%%%%%%%%%%%%%%%%%%%%%%%%%%%%%%%%%%%%%%%%%%%%%%%%%%%%%%%%%%%%%%%

\begin{appendix}
\section{Auxiliary results}\label{sec.aux}

We recall some known results. 
The following result is proved in \cite[Theorem 4.33]{Bre11}.

\begin{lemma}[Young's convolution inequality]\label{lem.young}
Let $1\le p$, $r\le\infty$, $u\in L^p(\R^d)$, $v\in L^q(\R^d)$, and
$1/p+1/q=1+1/r$. Then $u*v\in L^r(\R^d)$ and
$$
  \|u*v\|_r \le \|u\|_p\|v\|_q.
$$
\end{lemma}

The following lemma slightly extends \cite[Lemma 7.3]{Rou13} from the $L^2$ to
the $L^p$ setting.

\begin{lemma}\label{lem.embedd}
Let $p\ge 2$ and $T>0$. Then the following embedding is continuous:
$$
  L^p(0,T;W^{1,p}(\R^d))\cap W^{1,p}(0,T;W^{-1,p}(\R^d))\hookrightarrow
	C^0([0,T];L^p(\R^d)).
$$
\end{lemma}

\begin{proof}
Let $u\in L^p(0,T;W^{1,p}(\R^d))\cap W^{1,p}(0,T;W^{-1,p}(\R^d))$ and 
$0\le t_1\le t_2\le T$. Then
\begin{align}
  \bigg|\int_{\R^d}|u(t_2)|^p dx - \int_{\R^d}|u(t_1)|^p dx\bigg|
	&= \bigg|\int_{t_1}^{t_2}\langle\pa_t u,p|u|^{p-2}u\rangle dt\bigg| \label{a.aux} \\
	&\le p\|\pa_t u\|_{L^p(t_1,t_2;W^{-1,p}(\R^d))}
	\||u|^{p-2}u\|_{L^{p'}(t_1,t_2;W^{1,p'}(\R^d))}, \nonumber
\end{align}
where $p'=p/(p-1)$. Direct computations using Young's inequality lead to 
\begin{align*}
  \||u|^{p-2}u\|_{L^{p'}(t_1,t_2;W^{1,p'}(\R^d))}^{p'}
	&= C\int_{t_1}^{t_2}\int_{\R^d}\big(|u|^p + |u|^{p'(p-2)}|\na u|^{p'}\big)dx dt \\
	&\le C\int_{t_1}^{t_2}\|u(t)\|_{W^{1,p}(\R^d)}^{p}dt.
\end{align*}
We infer from \eqref{a.aux} and the continuity of the integrals with respect to 
the time integration boundaries that 
$t\mapsto\|u(t)\|_p$ is continuous and
\begin{equation}\label{a.supu}
  \sup_{0<t<T}\|u(t)\|_p \le \|u(0)\|_p + C\|\pa_t u\|_{L^p(t_1,t_2;W^{-1,p}(\R^d))}
	+ C\|u\|_{L^p(0,T;W^{1,p}(\R^d))}.
\end{equation}

Next, let $t\in[0,T]$ be arbitrary and let $\tau_n\to 0$ as $n\to\infty$ such that
$t+\tau_n\in[0,T]$.  
Estimate \eqref{a.supu} implies that $(u(t+\tau_n))_{n\in\N}$ is
bounded in $L^p(\R^d)$. Thus, there exists a subsequence $(\tau_{n'})$ of $(\tau_n)$
such that $u(t+\tau_{n'})\rightharpoonup v(t)$ weakly in $L^p(\R^d)$ 
as $n'\to\infty$ for some $v(t)\in L^p(\R^d)$. We can show,
using estimate \eqref{a.supu} and dominated convergence for the integral
$$
  \int_0^T\int_{\R^d}(u(t+\tau_{n'},x)-v(t,x))\phi(t,x)dx\quad\mbox{for }
	\phi\in C_0^\infty(\R^d\times(0,T))
$$
that in the limit $n'\to\infty$
$$
  \int_0^T\int_{\R^d}(u(t,x)-v(t,x))\phi(t,x)dx = 0,
$$
which yields $v(t)=u(t)$.

Moreover, since $t\mapsto\|u(t)\|_p$ is continuous, we have
$\|u(t+\tau_{n'})\|_p\to\|u(t)\|_p$. Since $L^p(\R^d)$ is uniformly convex,
we deduce from \cite[Prop.~3.32]{Bre11} that $u(t+\tau_{n'})\to u(t)$
strongly in $L^p(\R^d)$. Since the limit is unique, the whole sequence converges.
Together with \eqref{a.supu}, this concludes the proof. 
\end{proof}

Let $W_1\in C_0^\infty(\R^d)$ be nonnegative with $\int_{\R^d}W_1(x)dx=1$ and
define $W_\beta(x)=\beta^{-d}W_1(x/\beta)$ for $x\in\R^d$ and $\beta>0$.

\begin{lemma}\label{lem.wbeta}
Let $1\le p<\infty$ and $u\in W^{1,p}(\R^d)$. Then
$$
  \|W_\beta*u-u\|_p \le C\beta\|\na u\|_p.
$$
\end{lemma}

\begin{proof}
We use H\"older's inequality and the fact that $\|W_\beta\|_{L^1(\R^d)} =1$ 
to find that
\begin{align*}
  \|W_\beta*u-u\|_p^p &= \int_{\R^d}\bigg|\int_{\R^d}W_\beta(x-y)(u(x)-u(y))dy
	\bigg|^p dx \\
	&\le \int_{\R^d}\bigg(\int_{\R^d}W_\beta(x-y)dy\bigg)^{p-1}
  \bigg(\int_{\R^d}W_\beta(x-y)|u(x)-u(y)|^pdy\bigg)dx \\
  &= \int_{\R^d}\int_{\R^d}W_\beta(z)|z|^p\frac{|u(y+z)-u(y)|^p}{|z|^p}dydz \\
	&\le \|\na u\|_p^p\int_{\R^d}W_\beta(z)|z|^p dz \le C\beta^p\|\na u\|_p^p,
\end{align*}
which shows the lemma.
\end{proof}

%%%%%%%%%%%%%%%%%%%%%%%%%%%

\section{Fractional Laplacian}\label{sec.frac}

We recall that the fractional Laplacian $(-\Delta)^s$ for $0<s<1$ can be written as the 
pointwise formula
\begin{equation}\label{Def.FractionalLaplacian}
  (-\Delta)^s u(x) = c_{d,s}\int_{\R^d}\frac{u(x)-u(y)}{|x-y|^{d+2s}}dy,
	\quad\mbox{where }c_{d,s} = \frac{4^s\Gamma(d/2+s)}{\pi^{d/2}|\Gamma(-s)|},
\end{equation}
$u\in H^s(\R^d)$, and the integral is understood
as principal value if $1/2\le s<1$ \cite[Theorem 2]{Sti18}.
The inverse fractional Laplacian $(-\Delta)^{-s}$ is defined in \eqref{1.def}.
The following lemma can be found in \cite[Chapter V, Section 1.2]{Ste70}.

\begin{lemma}[Hardy--Littlewood--Sobolev inequality]\label{lem.hls}
Let $0<s<1$ and $1<p<\infty$. 
Then there exists a constant $C>0$ such that for all $u\in L^p(\R^d)$,
$$
  \|(-\Delta)^{-s}u\|_q\le C\|u\|_p, \quad\mbox{where }
	\frac{1}{p} = \frac{1}{q} + \frac{2s}{d}.
$$
\end{lemma}

Applying H\"older's and then Hardy--Littlewood--Sobolev's inequality
gives the following result.

\begin{lemma}\label{lem.hlsh}
Let $0<s<1$ and $1\le p<q<\infty$. 
Then there exists $C>0$ such that for all $u\in L^q(\R^d)$, $v\in L^r(\R^d)$,
\begin{align}
  \|u(-\Delta)^{-s}v\|_p &\le C\|u\|_q\|v\|_r, \quad 
	\frac{1}{q}+\frac{1}{r} = \frac{1}{p} + \frac{2s}{d}, \label{HLS1} \\
	\|u\na(-\Delta)^{-s}v\|_p &\le C\|u\|_q\|v\|_r, \quad 
	\frac{1}{q}+\frac{1}{r} = \frac{1}{p} + \frac{2s-1}{d},\ s>\frac12. \label{HLS2}
\end{align}
\end{lemma}
%\textcolor{blue}{This lemma also holds for $p=1$, check Lieb, Loss, Theorem 4.3 at p.106}

\begin{lemma}[Fractional Gagliardo--Nirenberg inequality I]\label{lem.GN1}
Let $d\ge 2$ and $1<p<\infty$. Then there exists
$C>0$ such that for all $u\in W^{1,p}(\R^d)$ or $u\in W^{2,p}(\R^d)$, respectively,
\begin{align*}
  \|(-\Delta)^s u\|_p &\le C\|u\|_p^{1-2s}\|\na u\|_p^{2s} \quad\mbox{if }
	0 < s \le 1/2, \\
  \|(-\Delta)^s u\|_p &\le C\|u\|_p^{1-s}\|\DD^2 u\|_p^{s} \quad\mbox{if }
  1/2 < s \le 1.
\end{align*}
\end{lemma}

\begin{proof}
It follows from the properties of the Riesz and Bessel potentials
\cite[Theorem 3, page 96]{Ste70} that the operator
$(-\Delta)^s:W^{1,p}(\R^d) \to L^p(\R^d)$ is bounded for $0<s\le 1/2$, while
the operator $(-\Delta)^s:W^{2,p}(\R^d) \to L^p(\R^d)$ is bounded for
$1/2<s\le 1$. Thus, if $0<s\le 1/2$,
$$
  \|(-\Delta)^s u\|_p \le C(\|u\|_p+\|\na u\|_p) \quad\mbox{for }u\in W^{1,p}(\R^d).
$$
Replacing $u$ by $u_\lambda(x)=\lambda^{d/p-2s}u(\lambda x)$ with $\lambda>0$ yields
$$
  \|(-\Delta)^s u\|_p 
	= \|(-\Delta)^s u_\lambda\|_p \le C(\|u_\lambda\|_p+\|\na u_\lambda\|_p)  
	= C\lambda^{-2s}(\|u\|_p+\lambda\|\na u\|_p).
$$
We minimize the right-hand side with respect to $\lambda$ giving the value
$\lambda_0=2s(1-2s)^{-1}\|u\|_p$ $\|\na u\|_p^{-1}$ and therefore,
$$
  \|(-\Delta)^s u\|_p \le C\|u\|^{1-2s}\|\na u\|_p^{2s}.
$$
The case $1/2<s\le 1$ is proved in a similar way.
\end{proof}

\begin{lemma}[Fractional Gagliardo--Nirenberg inequality II]\label{lem.GN2}
Let $d\ge 2$, $0<s\le 1/2$, $p\in(1,\infty)$, and $q\in[p,\infty)$. If
$p<d/(2s)$, we assume additionally that $q\le dp/(d-2sp)$. Then there exists
$C>0$ such that for all $u\in W^{1,p}(\R^d)$,
$$
  \|(-\Delta)^{-s}\na u\|_q \le C\|u\|_p^{1-\theta}\|\na u\|_p^\theta,
$$
where $\theta=1+d/p-d/q-2s\in[0,1]$.
\end{lemma}

\begin{proof}
The statement is true for $s=1/2$ since the operator
$(-\Delta)^{-1/2}\na:L^q(\R^d)\to L^q(\R^d)$ is bounded for any $q\in(1,\infty)$
\cite[Theorem 3, page 96]{Ste70}. Then the inequality follows from the
standard Gagliardo--Nirenberg inequality.  

Thus, let $0<s<1/2$. We claim that it is sufficient to prove that
$(-\Delta)^{-s}\na:W^{1,p}(\R^d)\to L^q(\R^d)$ is bounded. Indeed, assume that
\begin{equation}\label{a.bd}
  \|(-\Delta)^{-s}\na u\|_q\le C(\|u\|_p+\|\na u\|_p)
	\quad\mbox{for }u\in W^{1,p}(\R^d).
\end{equation}
Replacing, as in the proof of Lemma \ref{lem.GN1}, 
$u$ by $u_\lambda(x)=\lambda^{d/q-1+2s}u(\lambda x)$ with $\lambda>0$ yields
$$
  \|(-\Delta)^{-s}\na u\|_q 
	\le C\lambda^{-\theta}(\|u\|_p + \lambda\|\na u\|_p),
$$
where $\theta$ is defined in the statement of the theorem.
Minimizing the right-hand side with respect to $\lambda$ gives the value
$\lambda_0=\theta(1-\theta)^{-1}\|u\|_p\|\na u\|_p^{-1}$  
and therefore,
$$
  \|(-\Delta)^{-s}\na u\|_q \le C\|u\|_p^{1-\theta}\|\na u\|_p^{\theta}.
$$

It remains to show \eqref{a.bd}. To this end, we distinguish two cases.
First, let $p<d/(2s)$. By assumption, $p\le q\le r(1):= dp/(d-2sp)$. We apply
the Hardy--Littlewood--Sobolev inequality (Lemma \ref{lem.hls}) to find that
$$
  \|(-\Delta)^{-s}\na u\|_{r(1)} \le C\|\na u\|_p \le C(\|u\|_p+\|\na u\|_p).
$$
Furthermore, by using (in this order) the boundedness of 
$(-\Delta)^{-1/2}\na:L^p(\R^d)\to L^p(\R^d)$, Lemma 2 in \cite[page 133]{Ste70}, 
equation (40) in \cite[page 135]{Ste70}, 
and Theorem 3 in \cite[page 135f]{Ste70},
\begin{align}
  \|(-\Delta)^{-s}\na u\|_p &= \|\na(-\Delta)^{-1/2}(-\Delta)^{1/2-s}u\|_p
	\le C\|(-\Delta)^{1/2-s}u\|_p \label{a.p} \\
	&\le C\|(I-\Delta)^{1/2-s}u\|_p \le C\|(I-\Delta)^{1/2}u\|_p
	\le C(\|u\|_p+\|\na u\|_p). \nonumber
\end{align}
These inequalities hold for any $p\in(1,\infty)$. Now, it is sufficient to
interpolate with $1/q = \mu/p+(1-\mu)/r(1)$:
$$
  \|(-\Delta)^{-s}\na u\|_q \le \|(-\Delta)^{-s}\na u\|_p^\mu
	\|(-\Delta)^{-s}\na u\|_{r(1)}^{1-\mu}
	\le C(\|u\|_p+\|\na u\|_p).
$$

Second, let $p\ge d/(2s)$. We choose $\lambda\in(0,d/(2sp))\subset(0,1)$ and apply the
Hardy--Littlewoord--Sobolev inequality:
$$
  \|(-\Delta)^{-s}\na u\|_{r(\lambda)} 
	= \|(-\Delta)^{-\lambda s}(-\Delta)^{-(1-\lambda)s}\na u\|_{r(\lambda)}
	\le C\|(-\Delta)^{-(1-\lambda)s}\na u\|_p,
$$
where $r(\lambda)=dp/(d-2s\lambda p)$. Since $(1-\lambda)s<1/2$, we deduce from
\eqref{a.p} that
$$
  \|(-\Delta)^{-s}\na u\|_{r(\lambda)} \le C(\|u\|_p+\|\na u\|_p).
$$
Since $r(\lambda)\to \infty$ as $\lambda\to d/(2sp)$, the result follows.
\end{proof}

%%%%%%%%%%%%%%%%%%%%%%

\section{Parabolic regularity}\label{sec.regul}

\begin{lemma}[Parabolic regularity]\label{lem.regul}
Let $1<p<\infty$, $T>0$ and let $u$ be the (weak) solution to the heat equation
$$
  \pa_t u - \Delta u = v, \quad u(0) = u^0\quad\mbox{in }\R^d,
$$
where $v\in L^p(0,T;L^p(\R^d))$ and $u^0\in W^{2,p}(\R^d)$. Then there exists $C>0$,
depending on $T$ and $p$,
such that 
\begin{equation}\label{a.D2u}
  \|\pa_t u\|_{L^p(0,T;L^p(\R^d))} + \|\DD^2u\|_{L^p(0,T;L^p(\R^d))}
	\le C\big(\|v\|_{L^p(0,T;L^p(\R^d))} + \|\DD^2u^0\|_{L^p(\R^d)}\big).
\end{equation}
Furthermore, if $v=\diver w$ for some $w\in L^p(0,T;L^p(\R^d;\R^d))$ then
\begin{equation}\label{a.D1u}
  \|\na u\|_{L^p(0,T;L^p(\R^d))} \le C\big(\|w\|_{L^p(0,T;L^p(\R^d))}
	+ T^{1/p}\|\na u^0\|_{L^p(\R^d)}\big).
\end{equation}
\end{lemma}

\begin{proof}
We use a known result on the parabolic regularity for the equation
\begin{equation}\label{a.widehat}
  \pa_t\widehat{u} - \Delta\widehat{u} = v, \quad \widehat{u}(0)=0\quad
	\mbox{in }\R^d.
\end{equation}
It holds that \cite{Lie96}
\begin{equation}\label{a.D2}
  \|\pa_t \widehat{u}\|_{L^p(0,T;L^p(\R^d))} + \|\DD^2\widehat{u}\|_{L^p(0,T;L^p(\R^d))}
	\le C\|v\|_{L^p(0,T;L^p(\R^d))}.
\end{equation}
We apply this result to $\widehat{u}=u-e^{t\Delta}u^0$, where $e^{t\Delta}u^0$
is the solution to the homogeneous heat equation in $\R^d$ with initial datum $u^0$. 
Then $\widehat{u}$ solves \eqref{a.widehat} and satisfies estimate \eqref{a.D2}.
Inserting the definition of $\widehat{u}$ and observing that 
$\|\DD^2(e^{t\Delta}u^0)\|_p\le C\|\DD^2u^0\|_p$, we obtain \eqref{a.D2u}.

If $v=\diver w$ for some $w\in L^p(0,T;L^p(\R^d;\R^d))$, the uniqueness of solutions
to the heat equation yields $u=e^{t\Delta}u^0+\diver U$, where $U$ solves
$$
  \pa_t U - \Delta U = w, \quad U(0)=0\quad\mbox{in }\R^d.
$$
Then we deduce from the regularity result of \cite{Lie96} with $\widehat{u}=U$ and
$v=w$ that
$$
  \|\DD^2U\|_{L^p(0,T;L^p(\R^d))} \le C\|w\|_{L^p(0,T;L^p(\R^d))}.
$$
Since $\na u = e^{t\Delta}\na u^0+\na\diver U$, inequality \eqref{a.D1u}
follows.
\end{proof}

\end{appendix}

%%%%%%%%%%%%%%%%%%%%%%%%%%%%%%%%%%%%%%%%%%%%%%%%%%%%%%%%%%%%%%%%%%%%%%%%%%

\end{document}